\documentclass[12pt, twoside, leqno]{article}
\usepackage[a4paper,
bindingoffset=0.2in,
left=0.5in,
right=0.5in,
top=0.4in,
bottom=0.6in,
footskip=.4in]{geometry}

\usepackage[T1]{fontenc}
\usepackage[utf8]{inputenc}
\usepackage[english]{babel}
\usepackage{lmodern}               
\usepackage{amsmath,amsthm,amssymb}
\usepackage{mathrsfs,esint,bbm,stmaryrd}
\usepackage{enumitem,graphicx,xcolor,subcaption}

\usepackage{xfrac}
\usepackage{nicefrac,xfrac}
\usepackage{soul}

\usepackage{hyperref}

\newtheorem{theorem}{Theorem}           
\newtheorem{corollary}[theorem]{Corollary}
\newtheorem{lemma}[theorem]{Lemma}
\newtheorem{prop}[theorem]{Proposition}

\newtheorem{mainthm}{Theorem}           

\theoremstyle{definition}              
\newtheorem{definition}{Definition}

\theoremstyle{remark}                  
\newtheorem{step}{Step}
\newtheorem{case}{Case}
\newtheorem{remark}{Remark}

\DeclareMathOperator{\dist}{dist}
\DeclareMathOperator{\spt}{spt}

\let\div\relax
\DeclareMathOperator{\div}{div}
\DeclareMathOperator{\hc}{hc}
\DeclareMathOperator{\rad}{rad}

\newcommand{\abs}[1]{\left| #1 \right|}
\newcommand{\norm}[1]{\left\| #1 \right\|}
 
\newcommand{\mres}
{\mathbin{\vrule height 1.6ex depth 0pt width 0.13ex\vrule height 0.13ex depth 0pt width 1.3ex}}
\newcommand{\csubset}{\subset\!\subset} 

\DeclareMathAlphabet{\mathpzc}{OT1}{pzc}{m}{it}

\newcommand{\T}{\mathrm{T}}
\renewcommand{\d}{\mathrm{d}}

\newcommand{\N}{\mathbb{N}}       
\newcommand{\R}{\mathbb{R}}
\newcommand{\Z}{\mathbb{Z}}

\newcommand{\M}{\mathbb{M}}
\newcommand{\F}{\mathbb{F}}

\renewcommand{\SS}{\mathbb{S}}
\newcommand{\G}{\mathbf{G}} 
\renewcommand{\P}{\mathbb{P}}

\newcommand{\e}{\mathbf{e}}

\newcommand{\NN}{\mathscr{N}}     
\newcommand{\MM}{\mathscr{M}}

\newcommand{\GG}{\mathscr{G}}
\renewcommand{\H}{\mathscr{H}}

\newcommand{\eps}{\varepsilon}
\newcommand{\RR}{\varrho}

\newcommand{\GN}{\pi_{k-1}(\mathscr{N})}

\newcommand{\X}{\mathscr{X}}

\SetSymbolFont{stmry}{bold}{U}{stmry}{m}{n}  

\newcommand{\nablaT}{\nabla_{\top}}

\renewcommand{\S}{\mathbf{S}}

\definecolor{lightblue}{rgb}{0.22,0.45,0.70}   
\definecolor{darkgray}{gray}{0.4}    
\definecolor{lightgray}{gray}{0.8}

\title{$\Gamma$-convergence of the $p$-Dirichlet energy\\ for manifold-valued maps}

\author{Giacomo Canevari\thanks{Dipartimento di Informatica, Università di Verona. Strada le Grazie 15, 37134 Verona, Italy. E-mail: \texttt{giacomo.canevari@univr.it}, \texttt{ramon.oliverbonafoux@univr.it}, \texttt{giandomenico.orlandi@univr.it}} \and Van Phu Cuong Le\thanks{Institut für Mathematik, Universität Heidelberg,
Im Neuenheimer Feld 205, 69120 Heidelberg, Germany. E-mail: \texttt{cuong.le@uni-heidelberg.de}} \and Ramon Oliver-Bonafoux\footnotemark[1] \and Giandomenico Orlandi\footnotemark[1]}

\date{\today}

\begin{document}

\maketitle

\maketitle
\begin{abstract}
We prove a ${\Gamma}$-convergence result for the $p$-Dirichlet energy functional defined on maps from a smooth bounded domain $\Omega \subseteq \R^{n+k}$ to $\NN$, a $(k-2)$-connected and smooth closed Riemannian manifold with Abelian fundamental group, where $n$ and $k$ are integers, $n \geq 0$, $k \geq 2$. We focus on the regime $p \to~k^-$ under Dirichlet boundary conditions. The result provides a description of the asymptotic behavior of the \textit{topological singular sets} for families of $\NN$-valued Sobolev maps which satisfy suitable energy bounds. Such topological singular sets are $n$-dimensional flat chains with coefficients in $\pi_{k-1}(\NN)$ endowed with a suitable norm. As a consequence of our main result, it follows that the topological singular sets of energy minimizing $p$-harmonic maps converge to a $n$-dimensional flat chain $S$ with coefficients in $\pi_{k-1}(\NN)$ which has finite mass and solves the Plateau problem within the homology class associated to the boundary datum. 

	
\end{abstract}

\section*{Introduction}
\pdfbookmark[1]{Introduction}{introduction}
Let $\Omega \subseteq \R^{n+k}$ be a smooth bounded domain, with $n$ and $k$ natural numbers such that $k \geq 2$. Let $\NN$ be a smooth closed Riemannian manifold isometrically embedded into $\R^m$ for some integer $m$ such that $m \geq 2$. We assume that $\NN$ is $(k-2)$-connected with Abelian fundamental group, that is
\begin{equation}\tag{H}\label{hp:N}
\pi_0(\NN)= \ldots = \pi_{k-2}(\NN) = 0, \hspace{2mm} \pi_{k-1}(\NN) \not = \{0\}, \hspace{2mm} \pi_1(\NN) \mbox{ is Abelian}.  
\end{equation}
Homotopy groups~$\pi_{j}(\NN)$ with~$j\geq 2$
are always Abelian, so the assumption~\eqref{hp:N}
implies that~$\pi_{k-1}(\NN)$ is Abelian.
For any $p \in [1,\, k]$, let us denote as usual
\begin{equation*}
W^{1,p}(\Omega,\,\NN):= \{ u \in W^{1,p}(\Omega, \, \R^m): u(x) \in \NN \mbox{ for a. e. } x \in \Omega\}
\end{equation*}
and define the trace spaces $W^{1-1/p,p}(\partial \Omega,\, \NN)$ in an analogous fashion. For any $v \in W^{1-1/p,p}(\Omega,\NN)$, we define $W^{1,p}_v(\Omega,\, \NN)$ as the set of maps in $W^{1,p}(\Omega,\, \NN)$ with trace $v$. A key consequence of assumption~\eqref{hp:N}, due to Bethuel and Demengel~\cite[Theorem 4]{BethuelDemengel}, is that there exists $v \in W^{1-1/k,k}(\partial \Omega,\, \NN)$ such that $W^{1,k}_v(\Omega,\, \NN)$ is empty. In particular, there are no minimizing $k$-harmonic maps which coincide with such $v$ on $\partial \Omega$. In other words, under Assumption~\eqref{hp:N} the minimization problem
\begin{equation}\label{min_k_harmonic_maps}
\mbox{Find } v_k \in W^{1,k}_v(\Omega,\, \NN) \mbox{ such that } \int_\Omega \lvert \nabla v_k(x) \rvert^k \mathrm{d}x=\inf\left\{ \int_\Omega \lvert \nabla u(x) \rvert^{k} \mathrm{d}x: u \in W^{1,k}_v(\Omega,\, \NN)  \right\}
\end{equation}
does not always have a solution. Nevertheless, if one takes $p \in (k-1, \, k)$, then the space $W^{1,p}_v(\Omega,\, \NN)$ is always non-empty, see Hardt and Lin~\cite[Theorem 6.2]{HardtLin-Minimizing}. Hence, it is natural to consider the minimization problem
\begin{equation}\label{min_p_harmonic_maps}
\mbox{Find } v_p \in W^{1,p}_v(\Omega, \, \NN) \mbox{ such that } D_p(v_p)=\inf\left\{D_p(u): u \in W^{1,p}_v(\Omega,\, \NN) \right\},
\end{equation}
where, for all $u \in W^{1,p}_v(\Omega,\, \NN)$,
\begin{equation}\label{pDirichletFunctional_intro}\tag{$p$-DE}
D_p(u):= \int_{\Omega}\lvert \nabla u(x) \rvert^p \mathrm{d}x.
\end{equation}
By applying the direct method, one checks that a solution to~\eqref{min_p_harmonic_maps} exists. One might then hope that~\eqref{min_p_harmonic_maps} can be used as a relaxation for the possibly ill-posed problem~\eqref{min_k_harmonic_maps}. This leads to the question of studying the asymptotic behavior of families of maps $(u_p)_{p \in (k-1, \, k)}$ along with $(D_p(u_p))_{p \in (k-1, \, k)}$ as $p \to k^{-}$. 
\medskip

A contradiction argument shows that $\lim_{p \to k}D_{p}(u_{p}) = +\infty$  if (and only if) $W^{1,k}_v(\Omega,\, \NN)$ is empty  because of topological obstructions carried by the boundary datum. This suggests the presence of singularities of codimension~$k$, which cost an infinite amount of~$k$-energy while costing a finite amount of $p$-energy when $p<k$. In this paper we establish the previous guess in a rigorous manner. Roughly speaking, we show that, up to extracting a (countable) subsequence, the \textit{topological singular sets} (a notion introduced in~\cite{PakzadRiviere, CO1}) of~$(u_{p})_{p \in (k-1, \, k)}$ converge in a suitable sense as $p \to k$ to some $n$-dimensional surface $S$ which has minimal volume within a certain class. We note that the terms ``surface'' and ``volume'' are to be understood in a generalized sense, which in this paper is that relative to \textit{flat chains} with coefficients in a normed Abelian group. Standard references on this topic are the works by Whitney~\cite{Whitney-GIT}, Fleming~\cite{Fleming}, and White~\cite{White-Rectifiability}.
The theory of flat chains has been applied to the analysis of the space~$W^{1,p}(\Omega,\,  \NN)$ since the work of Pakzad and Rivière~\cite{PakzadRiviere}. It also provides a more general setting for the notion of \textit{minimal connection} introduced in the classical work of Brezis, Coron and Lieb~\cite{BrezisCoronLieb} (see also~\cite{BaLe,GiacomoCuong} for recent extensions). 
Within this framework, the topological singular set~$\S(u_p)$
of~$u_p\in W^{1,p}(\Omega, \, \NN)$ is defined in~\cite{PakzadRiviere, CO1}
as an $n$-dimensional flat chain with coefficients in the normed group $(\pi_{k-1}(\NN)$, $|\, \,|_{k})$. In order to prove lower bounds for the energy of~$u_p$
in terms of~$\S(u_p)$, the group norm~$|\, \,|_{k}$ must be chosen
in a specific way, depending on the target manifold~$\NN$
and the parameter~$k$; see Subsection~\ref{normsonhomotopygroups} 
for the details of the definition. The limiting chain~$S$ solves the homological Plateau problem --- that is, it minimizes the mass among all the chains in its homology class, which is determined by the domain~$\Omega$ and the boundary datum~$v$. Moreover, the sequence of energies~$((k-p)D_p(u_{p}))_{p \in (k-1, \, k)}$ converges towards the mass of $S$.
A precise statement of this convergence result
is given in Theorem~\ref{MainThm_harmonic} below.
In turn, Theorem~\ref{MainThm_harmonic} is a direct corollary 
of our main result, Theorem~\ref{MainThm}, which is
formulated in the framework of~$\Gamma$-convergence.

Let us now give a precise statement of our results. 
For any integer~$q$ between $0$ and $n+k$, let us denote by $\F_q(\Omega;\,  \pi_{k-1}(\NN))$ the set of $q$-dimensional flat chains relative to the open set $\Omega$ with coefficients in the norm group ($\pi_{k-1}(\NN)$, $|\, \,|_{k}$). (We will recall the definition in Section~\ref{Section:preliminaries} below.) Given a chain $S \in \F_q(\Omega;\, \pi_{k-1}(\NN))$, we denote its 
mass as~$\M_k(S)$ and its flat norm relative to~$\Omega$ as~$\F_{\Omega,k}(S)$.
Given a boundary datum~$v \in W^{1-1/k,k}(\Omega, \, \NN)$, there is a uniquely defined cobordism class~$\mathscr{C}(\Omega, \, v)$ of finite-mass $n$-chains that is associated with (the domain~$\Omega$ and)~$v$. The definition of~$\mathscr{C}(\Omega, \, v)$ is given in~\eqref{C_Omega_v} below. Heuristically speaking, the chains~$\mathscr{C}(\Omega, \, v)$ describe the topological singularities that are compatible with the obstructions set by the boundary datum~$v$.
For any open set $U \subseteq \R^{n+k}$, $p \in (k-1, \, k)$ and $u \in W^{1,p}(U,\, \NN)$ we set
\begin{equation*}
D_p(u,\, U):= \int_U \lvert \nabla u(x) \rvert^p \mathrm{d}x.
\end{equation*}
We can now state our main result, which reads as follows:
\begin{mainthm}[$\Gamma$-convergence of the $p$-Dirichlet energies and topological singular sets]\label{MainThm}
Assume that~\eqref{hp:N} holds and let $v \in W^{1-1/k,k}(\Omega, \, \NN)$. Then, we have:
\begin{enumerate}[label=(\roman*)]
\item\label{MainThm1}\emph{Compactness and lower bound.}
Let $(u_p)_{p \in (k-1, \, k)}$  be a family such that $u_p \in W^{1,p}_v(\Omega,\, \NN)$ for all $p$ and
\begin{equation*}
\sup_{p \in (k-1, \, k)}(k-p)D_p(u_p) <+\infty.
\end{equation*}
Then, there exists a (non relabelled) countable
sequence $p\to k$ and a finite-mass chain $S$ in $\mathscr{C}(v,\, \Omega)$ such that~$\F_{\Omega,k}(\S(u_p)- S) \to 0$ as $p \to k$ and, for any open subset~$A\subseteq\R^{n+k}$,
\[
\M_k(S\mres A) \leq \liminf_{p\to k}
(k-p)D_p(u_p, \, A\cap\Omega).
\]	
\item\label{MainThm2} \emph{Upper bound.} 
For any $S$ in $\mathscr{C}(v,\Omega)$ 
there exists a family of maps~$(u_p)_{p \in (k-1, \, k)}$ such that $u_p \in W^{1,p}_v(\Omega,\, \NN)$ for all $p$, $\F_{\Omega,k}(\S(u_p)- S) \to 0$ as $p \to k$ and
\[
\limsup_{p\to k} (k-p)D_p(u_p)
\leq \M_k(S).
\]
\end{enumerate}
\end{mainthm}

One could also prove a $\Gamma$-convergence result without prescribed boundary conditions in the spirit of~\cite[Theorem 1.1]{ABO2} and~\cite[Proposition D]{CO2}; see Proposition~\ref{prop:Gamma-nobd} below.
In the particular case $n=0$ and $k=2$ but for arbitrary $\NN$ (possibly, with non-Abelian fundamental group), Theorem~\ref{MainThm} has been recently established by Van Vaerenbergh and Van Schaftingen~\cite{VanVanVaeren}. 
Soon before the completion of this paper, Caselli, Freguglia and Picenni~\cite{CaselliFregugliaPicenni} announced a proof of Theorem~\ref{MainThm} for the particular case in which $\NN=\SS^{k-1}$.
In this case, the flat chain $\mathbf{S}(u)$ has integer multiplicities and coincides with the distributional Jacobian of~$u$,  as considered in~\cite{ABO1, JerrardSoner-jacobians} and in the references therein. Compactness and lower bound for maps between closed manifolds
had been given in Chapter~12 of Stern's PhD Thesis~\cite{SternPhD}. His approach is different than ours, mainly because his definition of topological singular set,
contrary to the one in~\cite{CO1}, is based on differential forms. As one does for any result of $\Gamma$-convergence type, one can apply Theorem~\ref{MainThm} to families of minimizers (rather than arbitrary families), yielding the following:
\begin{mainthm}[Compactness and convergence for $p$-harmonic maps as $p \to k^-$]\label{MainThm_harmonic}
Assume that~\eqref{hp:N} holds and let $v \in W^{1-1/k,k}(\Omega, \, \NN)$. Consider a family $(v_p)_{p \in \N}$ such that $v_p \in W^{1,p}_v(\Omega,\, \NN)$ solves~\eqref{min_p_harmonic_maps} for each $p$. Then, up to an extraction of a subsequence, there exists $S \in \mathscr{C}(\Omega,\, v)$ such that $\F_{\Omega,k}(\mathbf{S}(v_p)-S) \to 0$ as $p \to k$ and 
\begin{equation*}
\lim_{p \to k}(k-p)D_p(v_p)=\mathbb{M}_k(S).
\end{equation*}
Moreover, $S$ has minimal mass among the flat chains in the class $\mathscr{C}(\Omega,\, v)$.
\end{mainthm}
Theorem~\ref{MainThm_harmonic} provides a description of the asymptotic behavior of $p$-harmonic maps in a rather general setting. For the case $n=0$ and $\NN=\mathbb{S}^{k-1}$, Theorem~\ref{MainThm_harmonic} had already been proven by Hardt, Lin and Wang~\cite{LingWang}, building upon on an earlier result by Hardt and Lin~\cite{RobertHardtFanghuaLin} devoted to the case $k=2$, see also Hardt and Chen~\cite{HardtChen}. For $\NN$ having a (possibly non-Abelian) finite fundamental group, $n=1$ and $k=2$, a version of Theorem~\ref{MainThm_harmonic} has been established recently by Bulanyi, Van Vaerenbergh and Van Schaftingen~\cite{VanVanVaeren2}. The results in~\cite{RobertHardtFanghuaLin,LingWang, VanVanVaeren} also address higher-order expansions and renormalized energies. Nevertheless, to the best of our knowledge, for the setting of domains with boundary only the cases $n=0$ and $n=1$ (corresponding to point and line singularities, respectively) had been addressed before with the exception of the recent work~\cite{CaselliFregugliaPicenni}. Here we deal with the case of arbitrary~$n$, and, as a consequence, the approach we take here is substantially different from~\cite{RobertHardtFanghuaLin,LingWang, VanVanVaeren, VanVanVaeren2}. Besides the already mentioned results in~\cite[Chapter 12]{SternPhD}, Stern~\cite{DanielStern1} also studied the asymptotic analysis, as~$p\to 2$, of $p$-harmonic maps from a closed Riemannian manifold of arbitrary dimension to $\NN=\mathbb{S}^1$. The spirit of the results in~\cite{DanielStern1} is rather different than ours. In particular, they do not rely on $\Gamma$-convergence and deal instead with families of stationary $p$-harmonic maps. Hence, the results in~\cite{DanielStern1} do not imply $\Gamma$-convergence, and viceversa.

One can also consider a different approximation for~\eqref{min_k_harmonic_maps}, based on functionals of Ginzburg-Landau type, written in a general form as follows:
\begin{equation}\label{GL_intro_general} \tag{$\eps$-GL}
E_\eps(u):= \int_\Omega \left( \frac{1}{k}\lvert \nabla u \rvert^k+\frac{1}{\eps^k}f(u) \right)\mathrm{d}x,
\end{equation}
where now $u \in W^{1,k}_v(\Omega,\, \R^m)$, $\eps$ is a positive parameter and $f$ is a nonnnegative smooth potential vanishing exactly on $\NN$ and satisfying some additional assumptions of a more technical purpose. Functionals of the type~\eqref{GL_intro_general} have been thoroughly studied since the seminal book of Bethuel, Brezis and Hélein~\cite{BBH}. The results in~\cite{BBH} deal with the case $n=0$, $k=2$ and $f(u):= \frac{1}{4}(\lvert u \rvert^2-1)^2$, which implies that $\NN=\mathbb{S}^1$. For such a choice,~\eqref{GL_intro_general} is a reduced model for superconductivity. Roughly speaking, passing to the limit $\eps \to 0^+$ for~\eqref{GL_intro_general} is analogous to performing the limit $p \to k^-$ for~\eqref{pDirichletFunctional_intro}. Generally speaking, results based on the functionals~\eqref{GL_intro_general} were established before the counterpart based on~\eqref{pDirichletFunctional_intro} was proven. For instance, in~\cite[Chapter XI, Problem 9]{BBH} the authors asked whether one could provide an analog of their results for the problem~\eqref{pDirichletFunctional_intro}, which was later answered affirmatively in~\cite{RobertHardtFanghuaLin}. The question of $\Gamma$-convergence for the planar Ginzburg-Landau functionals was not addressed in~\cite{BBH}. The lower bound (with a compactness result) was obtained later by Jerrard and Soner~\cite{JerrardSoner-GL}, while a full $\Gamma$-convergence result for abitrary~$n$, $k$ and~$\NN = \mathbb{S}^{k-1}$ was announced in \cite{AlbertiBUMI} and proved in~\cite{ABO2}. The first and second-order $\Gamma$-limits in the planar case were studied by Alicandro and Ponsiglione~\cite{AlicandroPonsiglione}, based on the so-called ball construction established in the independent works by Jerrard and Sandier~\cite{Jerrard,Sandier}. 

The cases~$\NN = \SS^1$ and~$\NN = \mathbb{S}^{k-1}$ are probably the most common ones in the literature on variational problems of the form~\eqref{GL_intro_general}. (Circle-valued Sobolev maps are also the topic of a recent monograph~\cite{BrezisMironescu-book}.) Circle-valued and sphere-valued maps arise naturally, e.g., in the context of superconductivity, micromagnetism, and materials science. However, other manifolds such as the real projective plane~$\NN=\mathbb{RP}^2$ (with~$k=2$) have also deserved particular attention, as they have physical relevance as well. With the choice~$\NN=\mathbb{RP}^2$, \eqref{GL_intro_general} can be interpreted as the Landau-de Gennes free energy of nematic liquid crystals in the one-constant approximation (see e.g.~\cite{MajumdarZarnescu, NguyenZarnescu, GolovatyMontero}).
In turn, the $p$-Dirichlet energy with~$\NN = \mathbb{RP}^2$ reads
\begin{equation} \label{pEnergyQ}
D_p(\mathbf{Q})=\int_{\Omega} \vert \nabla \mathbf{Q} \vert^p \mathrm{d}x,
\end{equation}
where $p \in (1,2)$ and $\mathbf{Q}\in W^{1,p}(\Omega,\,  \mathbb{RP}^2)$ is the so-called $Q$-tensor, which describes the configuration of the crystal. Functionals with sub-quadratic growth such as~\eqref{pEnergyQ} have also been proposed as a model for nematic liquid crystals, see~\cite{GiaquintaMucci, Mucci2012, BallBeford, CaMaStro}.
Regarding more general classes of target manifolds, several $\Gamma$-convergence results for~\eqref{GL_intro_general} have been established. Monteil, Rodiac and Van Schaftingen~\cite{MonteilAntonin,AntoRemyRodiacJean} focused on the case $n=0$ and $k=2$ with $\NN$ any closed manifold. The case of arbitrary~$n$ and~$k$ and with $\NN$ satisfying~\eqref{hp:N} was studied in~\cite{CO2}.

The functionals~\eqref{pDirichletFunctional_intro} and~\eqref{GL_intro_general} are certainly not the only ways to approximate~\eqref{min_k_harmonic_maps}. For instance, Caselli, Freguglia and Picenni~\cite{CaselliFregugliaPicenni-nonlocal} have considered an approximation in terms of nonlocal functionals and proved a $\Gamma$-convergence result for the fractional Gagliardo seminorm in~$H^{\frac{1 + s}{2}}(\Omega, \, \mathbb{S}^1)$ as~$s\to 1$.
Another possible approach is the approximation of~\eqref{min_k_harmonic_maps} by discrete problems, such as the XY-model of statistical mechanics (see e.g.~\cite{AlicandroCicalese, ADGP} and the erferences therein).

The proof of Theorem~\ref{MainThm} follows the same scheme of~\cite[Theorem C]{CO2}. In particular, using the approach of~\cite{Jerrard,Sandier} we prove a lower energy bound for~\eqref{pDirichletFunctional_intro} in the critical case $n=0$  (Proposition~\ref{prop:lowerbound-k}) which is of independent interest.  However, along the proof we encounter several technical obstacles which are not present in~\cite{CO2}. One of these obstacles arises in the proof of Proposition \ref{prop:lowerbound-k}, as at some point one needs to have suitable uniform (as $p \to k^-$) gradient bounds for $p$-harmonic maps from $\SS^{k-1}$ into $\NN$ which are minimizers in a given homotopy class, see Theorem \ref{th:p-harmonic} below. To our knowledge, the proof of such bounds was not available in the literature and here we provide one in a more general setting (Theorem \ref{th:p_harmonic_appendix}). Another difficulty arises from the fact that the singular sets of maps in $W^{1,p}_v(\Omega,\, \NN)$ do not necessarily have finite mass, which does not allow to apply the deformation theorem for flat chains as in \cite{CO2} in the proof of the compactness and lower bound statement. This is circumvented by applying an argument based on performing suitable modifications on the sequence $(u_p)_{p \in (k-1,k)}$, which represents a deviation with respect to \cite{CO2}. See Section \ref{Section:lower_bound} below for more details. Furthermore, the manifold constraint does not allow to perform extensions and interpolations as in the Ginzburg-Landau setting, where the absence of the constraint allows to perform such procedures more freely. This needs to be circumvented by using approaches which are specific to the manifold-constrained setting, see e. g. Lemma \ref{lemma:goodlemma} below.

We point out that the assumption~\eqref{hp:N} is standard in the field of manifold-valued maps (see e. g.~\cite{Hajlasz,RobertHardtFanghuaLin,PakzadRiviere,CO2}) and completely essential to our approach, for several reasons. To start off, there is no available theory of flat chains with coefficients in a non-Abelian group.
Moreover, if $k \geq 3$ and $\pi_{k-2}(\NN) \not = \{0\}$ one might have that $W^{1,p}_v(\Omega,\, \NN)=\emptyset$ for all $p$ in a left neighbourhood of~$k$. (For instance, this is the case if the domain~$\Omega$ is the
unit ball~$B^2\subseteq\R^2$, the target manifold is~$\NN=\SS^1$, $k=3$,
and the boundary datum is~$v(x) = x$, as~$W^{1,p}_v(\Omega, \, \SS^1)$ is empty for any~$p\geq 2$.) At the technical level, our proof relies heavily on the use of a retraction map introduced by Hardt, Kinderlehrer and Lin~\cite{HKL,HardtLin-Minimizing} and the existence of such map requires that $\pi_i(\NN)=0$ for all integer $i$ ranging from~$0$ to~$k-2$, see~\cite{CO1,CO2} for additional comments. 

The paper is organized as follows. In Section \ref{Section:preliminaries} we provide some preliminary notions and results on retraction maps, flat chains, norms on Abelian groups, the and topological singular sets. In Section \ref{Section:lower_bound} we prove the compactness and lower bound statement of Theorem \ref{MainThm1}. In Section \ref{Section:upper_bound} we provide the proof of the upper bound statement of Theorem \ref{MainThm1}. In Section \ref{Section:GammaNB} we write a precise statement of the analogous to Theorem \ref{MainThm1} without prescribed boundary conditions. Finally, in Appendix \ref{appendix} we prove some uniform gradient estimates for supercritical harmonic maps between manifolds.

\medskip

\textbf{Acknowledgments.} R.O.-B. acknowledges support from the Program Horizon Europe Marie Sklodowska-Curie Post-Doctoral Fellowship (HORIZON-MSCA-2023-PF-01). Grant agreement: \linebreak 101149877. Project acronym: NFROGS. V. P. C. Le gratefully acknowledges the support of the Department of Mathematics at Heidelberg University through a postdoctoral fellowship, as well as the support of STRUCTURES, Cluster of Excellence at Heidelberg University.

\medskip

\numberwithin{equation}{section}
\numberwithin{definition}{section}
\numberwithin{theorem}{section}
\numberwithin{remark}{section}
\numberwithin{example}{section}

\section{Notations and preliminary results}\label{Section:preliminaries}

Following~\cite{CO1,CO2}, we recall some preliminary notations and properties. We also provide proofs of some technical results whose statements we were not able to find in the literature. Throughout the paper, given nonnegative real numbers $a$ and $b$ we will write $a \lesssim b$ whenever there exists a positive constant $C$ depending only on $n$, $k$, $\Omega$ and $\NN$ such that $a \leq C b$. Given $\ell \in \N^*$, the ball of center $x \in \R^{\ell}$ and radius $r>0$ will be denoted as $B^\ell(x,\, r)$. When the center is the origin, we will write $B^\ell_r$ and $B^\ell$ for $r=1$. Given a set~$E\subseteq\R^{n+k}$, we will denote by~$\overline{E}$ the closure of~$E$ and write~$E\csubset F$ if and only if~$\overline{E}\subseteq F$.

\subsection{Norms on~\texorpdfstring{$\pi_{k-1}(\NN)$}{the homotopy group of N}}\label{normsonhomotopygroups}
Let~$\G$ be an Abelian group. We recall that a norm on~$\G$
is a function $\abs{\, \,}\colon \G\to [0, \, +\infty)$ 
that satisfies the following properties:
\begin{enumerate}[label=(\roman*)]
 \item $|\sigma| = 0$ if and only if~$\sigma=0$;
 \item $|-\sigma| = |\sigma|$ for any~$\sigma\in \G$;
 \item $|\sigma_1 + \sigma_2|\leq |\sigma_1|+|\sigma_2|$ for any~$\sigma_1$, $\sigma_2\in \G$.
\end{enumerate}
In particular, we do \emph{not} require that $|n\sigma| = n|\sigma|$ 
for any~$n\in\N$, $\sigma\in \G$. This is consistent with the theory of flat chains as developed 
in~\cite{Fleming, White-Rectifiability}. However, for technical reasons (see~\cite{CO1}),
we are interested in norms 
that satisfy the additional condition
\begin{equation}\label{discrete_norm}
  \inf\left\{ \abs{\sigma}\colon \sigma\in \G, \ \sigma\neq 0\right\}  > 0.
\end{equation}
Condition~\eqref{discrete_norm} implies that~$\abs{\, \,}$
induces the discrete topology on~$\G$.

We consider a suitable family of norms
on the group~$\G = \pi_{k-1}(\NN)$, depending on $p>k-1$. For any such $p$
and any~$\sigma\in\pi_{k-1}(\NN)$, we define
\begin{equation} \label{E_p}
 E_p(\sigma) := \inf\left\{
 \int_{\SS^{k-1}}\abs{\nablaT u}^p \,\d\H^{k-1}
 \colon u\in W^{1,p}(\SS^{k-1}, \, \NN)\cap\sigma \right\} \! ,
\end{equation}
where~$\nablaT$ denotes the tangential gradient 
on~$\SS^{k-1} = \partial B^k_1\subseteq\R^k$,
that is, the restriction of the Euclidean gradient~$\nabla$
to the tangent plane to the sphere. Because of the compact 
embedding
\begin{equation*}
W^{1,p}(\SS^{k-1}, \NN)\hookrightarrow C(\SS^{k-1}, \, \NN),\quad \mbox{ for } p>k-1,
\end{equation*}
the set~$W^{1,p}(\SS^{k-1}, \, \NN)\cap\sigma$ is 
sequentially weakly closed and, hence,
the infimum at the right-hand side of~\eqref{E_p} is achieved.
However, the function~$E_p$ is not a norm, in general,
because it may not satisfy the triangle inequality~(iii).
Instead, we define 
\begin{equation} \label{groupnorm}
 |\sigma|_p := 
   \inf\left\{\sum_{i=1}^h E_p(\sigma_i)\colon h\in\N, \ (\sigma_i)_{i=1}^h\in\pi_{k-1}(\NN)^h,
   \ \sum_{i=1}^h \sigma_i = \sigma\right\}
\end{equation}
for any~$p > k-1$ and~$\sigma\in\pi_{k-1}(\NN)$. By reasoning as in~\cite[Proposition~A.1]{CO2}, one checks that the infimum in \eqref{groupnorm}.
Then, the function~$\abs{\, \,}_p$ is a norm on~$\pi_{k-1}(\NN)$
and satisfies
\begin{equation} \label{alpha_p}
 \alpha_p := \inf\left\{\abs{\sigma}_p\colon 
 \sigma\in \pi_{k-1}(\NN), \ \sigma\neq 0 \right\} > 0
\end{equation}
(see e.g. the argument in~\cite[Proposition~2.1]{CO2}). The norm $|\, \,|_{k}$ is the one considered in~\cite{CO2} up to the multiplicative constant~$1/k$. This small modification does not have any relevant effect on the results of~\cite{CO2}. Alternatively, one could modify the definition of the functional $D_p$ in~\eqref{pDirichletFunctional_intro} by multiplying it by~$1/p$ so that both norms would coincide. 
\subsection{Retraction maps}\label{Subsection:retraction}
In~\cite[Lemma 6.1]{HardtLin-Minimizing}, see also \cite[Lemma~6.1]{HardtLin-Minimizing}, 
\cite[Proposition~2.1]{BousquetPonceVanSchaftingen}, \cite[Lemma~4.5]{Hopper},
it was proven that:
\begin{prop}\label{prop:retractionmap}
There exist a set~$\X\subseteq\R^m$ and a smooth map $\RR\colon\R^m\setminus\X\to\NN$, such that~$\X$ is a finite union of compact polyhedra of dimension~$m-k$, while~$\RR$ satisfies $\RR(z) = z$ for any~$z\in\NN$ and
\[
 \sup_{z\in\R^m}\dist(z, \, \X)|\nabla\RR(z)| < +\infty.
\]
\end{prop}
A map $\RR$ satisfying the properties in Proposition~\ref{prop:retractionmap} will be called a \textit{retraction map}.
Such a map is non-unique. However, we will choose a set~$\X$ and a map~$\RR$ satisfying the properties above and keep them fixed throughout the sequel. We collect here further properties of~$\RR$. As in~\cite[Subsection 3.3]{CO2}, we define
\begin{equation*}
\tilde{\RR}_y: z \in \R^m \setminus (\X+y) \to \tilde{\RR}_y(z):=\RR(z-y) \in \NN
\end{equation*}
Since~$y\mapsto(\tilde{\RR}_y)|_{\NN}$ defines a smooth map
from a neighbourhood of the origin in~$\R^m$ to~$C^1(\NN, \, \NN)$
and~$\tilde{\RR}_0$ is the identity on~$\NN$,
there exists $r_\RR>0$ such that $(\tilde{\RR}_y)|_{\NN}$ is a diffeomorphism of $\NN$ for all $y$ in $B^m_{r_\RR}$. Therefore, for all such $y$ we can define
\begin{equation}\label{rho_y}
\RR_y: z \in \R^m \setminus (\X+y) \to \RR_y(z):=((\tilde{\RR}_y)|_{\NN})^{-1}(\RR(z-y)) \in \NN,
\end{equation}
which is a smooth retraction of $\R^m \setminus (\X+y)$ into $\NN$. In particular, $\RR_y(z)=z$ for all $z \in \NN$. Moreover, we have that for all $y \in B^m_{r_\RR}$ and $z \in  \R^m \setminus (\X+y)$
\begin{equation}\label{ineq_gradient_rm_y}
\lvert \nabla \RR_y (z) \rvert \lesssim \frac{1}{\mathrm{dist}(z-y,\X)}.
\end{equation}
With the notations above, we have:
\begin{lemma}\label{lemma:retractionmap}
	Under the assumptions of Proposition~\ref{prop:retractionmap}, for any $R>0$ there exists a constant $C_R>0$ depending only on $\NN$ and $R$ such that for any $p \in (k-1, \, k)$ one has
	\begin{equation}\label{retractionmap_est}
	(k-p) \int_{B^m_R}\lvert \nabla \RR(z) \rvert^p \mathrm{d}z \leq C_R.
	\end{equation}
	Similarly, for all $z \in B^m_R$ we have
	\begin{equation}\label{retractionmap_y_est}
	(k-p) \int_{B^m_{r_\RR}}\lvert \nabla \RR_y(z) \rvert^p \mathrm{d}y \leq C_R
	\end{equation}
\end{lemma}
\begin{proof}
We proceed as in~\cite[Lemma~2.3]{HKL}. Let $S_k:= [0, \, 1]^{m-k} \times \{0\}^{k}$ and
	\begin{equation*}
	g_p: z \in [0, \, 1]^m \to g_p(z):= \frac{1}{\dist(z,S_k)^p} \in (0,+\infty).
	\end{equation*}
	By Fubini's Theorem, we have
	\begin{equation*}
	\int_{[0, \, 1]^m}g_p(z)\mathrm{d}z = \int_{[0, \, 1]^{m-k}}\int_{[0, \, 1]^k}\frac{1}{\lvert z_1\rvert^p} \mathrm{d}z_1\mathrm{d}z_2=\int_{[0, \, 1]^k}\frac{1}{\lvert z_1 \rvert^p}\mathrm{d}z_1,
	\end{equation*}
	so that a change of variables into polar coordinates gives
	\begin{equation}\label{ineq_gp}
	(k-p)\int_{[0, \, 1]^m}g_p(z)\mathrm{d}z \lesssim 1
	\end{equation}
Thus,~\eqref{retractionmap_est} follows by Proposition~\ref{prop:retractionmap} combined with scaling and rotation arguments. In order to prove inequality \eqref{retractionmap_y_est}, notice that by applying \eqref{rho_y} along with the change of variables $\tilde{y}=z-y$ one gets
\begin{equation*}
(k-p) \int_{B^m_{r_\RR}}\lvert \nabla \RR_y(z) \rvert^p \mathrm{d}y \lesssim \int_{z+B^m_{r_\RR}}\frac{1}{\mathrm{dist}(\tilde{y},\X)^p}\mathrm{d}\tilde{y}
\end{equation*}
and the conclusion then follows again by \eqref{ineq_gp} after rotations and scalings, up to possibly choosing a larger constant $C_R$.
\end{proof}
From Lemma~\ref{lemma:retractionmap} one readily deduces Lemma~\ref{lemma:average_projection} below, a result which is analogous to~\cite[Proposition 6.4 (iii)]{ABO2},~\cite[Lemma 2 (iii)]{CO2} and that will be used repeatedly in the sequel.
\begin{lemma}\label{lemma:average_projection}
Let $V$ and open bounded set such that $\partial \Omega \subseteq \partial V$. Let $\tilde{v} \in W^{1,1-1/k}(\partial V,\, \NN)$ be such that $\tilde{v}$ agrees with $v$ a. e. in $\Omega$. Let $u \in W^{1,k}_{\tilde{v}}(V,\,\mathbb{R}^m) \cap L^\infty(V,\,\mathbb{R}^m)$. Then for all $p \in (k-1, \, k)$,
\begin{equation}\label{average_projection_inequality}
(k-p) \int_{B^m_{r_\RR}} \lVert \nabla w_y \rVert^p_{L^p(V)}\mathrm{d}y \leq C \lVert \nabla u \rVert_{L^k(V)}^p
\end{equation}
where $w_y:= \RR_y \circ u$ for all $y \in B^m_{r_\RR}$ and $C>0$ is a constant depending on $\lVert u \rVert_{L^\infty(V)}, \tilde{v}, k$ and $\NN$. In particular, $W^{1,p}_{\tilde{v}}(V, \,\NN)$ is non-empty.
\end{lemma}
\begin{proof}
We argue again as in \cite[Lemma 2.3]{HKL}. By applying Fubini's Theorem and the chain rule, it follows
\begin{equation*}
\int_{B^m_{r_\RR}}\lVert \nabla w_y \rVert^p_{L^p(V)}\mathrm{d}y \leq \int_{V} \abs{\nabla u(x)}^p\int_{B^m_{r_\RR}}\abs{\nabla \RR_y(u(x))}^p \mathrm{d}y\mathrm{d}x,
\end{equation*}
so that by inequality \eqref{retractionmap_y_est} in Lemma \ref{lemma:retractionmap} we get
\begin{equation*}
(k-p)\int_{B^m_{r_\RR}}\lVert \nabla w_y \rVert^p_{L^p(V)}\mathrm{d}y \leq C\lVert \nabla u \rVert^p_{L^p(V)}
\end{equation*}
and then~\eqref{average_projection_inequality} is obtained by applying Hölder's inequality.
\end{proof}

\subsection{Flat chains with coefficients in a normed Abelian group}\label{Subs:flat-chains}
We now recall some basic definitions on flat chains with coefficients in a normed group following \cite{CO1}. The reader is referred to \cite{CO1}, as well as the classical references \cite{Fleming,White-Rectifiability,Whitney-GIT}, for a much more detailed exposition. Let $(\G,\abs{\, \,})$ be an arbitrary normed Abelian group as in Subsection \ref{normsonhomotopygroups}. Given $q$ an integer between $0$ and $n+k$, consider the set of all compact, convex and oriented polyhedra of dimension $q$ in $\R^{n+k}$. One can then define an Abelian group by considering formal sums of elements of the previous set with coefficients in $\G$. We now consider $\sim$ a relation in such a module, characterized by $-\sigma \sim \sigma'$ if and only if $\sigma$ and $\sigma'$ differ only by the orientation and $\sigma \sim \sigma_1+\sigma_2$ if and only if $\sigma$ is obtained by gluing $\sigma_1$ and $\sigma_2$ along a common face and the orientations coincide. This defines an equivalence relation on the Abelian group, and the corresponding quotient is the group of polyhedral $q$-chains with coefficients in $\G$. Such a group is denoted as $\P_q(\R^{n+k};\, \G)$. If $S$ is an element of $\P_q(\R^{n+k};\, \G)$, then one can write
\begin{equation*}
S=\sum_{i=1}^{\ell} \alpha_i \llbracket \sigma_i \rrbracket,
\end{equation*}
where $\alpha_i$ belongs to $\G$ for each $i$ and  $\sigma_1, \, \ldots \, ,\sigma_{\ell}$ are compact, convex, oriented and non-overlaping polyhedra of dimension $q$, with $\llbracket \cdot \rrbracket$ being the equivalence class associated to the relation $\sim$. The mass of $S$, denoted as $\M(S)$, is then defined as
\begin{equation*}
\M(S):= \sum_{i=1}^{\ell}\abs{\alpha_i}\mathscr{H}^{q}(\sigma_i).
\end{equation*}
If $q$ is not smaller than $1$, then one can define a linear operator $\partial: \P_q(\R^{n+k};\, \G) \to \P_{q-1}(\R^{n+k};\, \G)$, called the boundary operator, in such a way that given $\sigma$ a polyhedron, $\partial\llbracket \sigma \rrbracket$ is the sum of the boundary faces of $\sigma$ with the orientation induced by $\sigma$ and multiplicity $1$. As expected, the boundary operator satisfies the property~$\partial \circ \partial=0$. The flat norm of $S$ an element of $\P_q(\R^{n+k};\, \G)$ is then 
defined as
\begin{equation*}
\F(S):= \inf\left\{ \M(P)+\M(Q): P \in \P_{q+1}(\R^{n+k};\, \G), \hspace{2mm} Q \in \P_{q}(\R^{n+k};\, \G), \hspace{2mm} S=\partial P+Q \right\}.
\end{equation*}
It can be shown that $\F$ defines indeed a norm on $\P_q(\R^{n+k};\, \G)$. The completion of $(\P_q(\R^{n+k};\, \G),\F)$ as a metric space is $\F_q(\R^{n+k},\, \G)$, called the set of flat chains of dimension $q$ with coefficients in $\G$, endowed with the metric $\F$. One can then extend the mass $\M$ as a functional from $\F_q(\R^{n+k};\, \G)$ into $[0,+\infty]$ which is lower semicontinuous with respect to $\F$. The set of $q$-dimensional flat chains with coefficients in $\G$ and finite mass is denoted as $\M_q(\R^{n+k};\, \G)$. It also turns out that
\begin{equation*}
\F(S)= \inf\left\{ \M(P)+\M(Q): P \in \F_{q+1}(\R^{n+k}M\, \G), \hspace{2mm} Q \in \F_{q}(\R^{n+k};\, \G), \hspace{2mm} S=\partial P+Q \right\},
\end{equation*}
whenever $S$ belongs to $\F_q(\R^{n+k},\, \G)$.

Given $A$ a Borel subset of $\R^{n+k}$ and $S$ a chain in $\M_q(\R^{n+k};\, \G)$, the restriction of $S$ to $A$, denoted as $S \mres A$, is defined naturally by approximating with polyhedral chains. One then finds that $S \to \M(S \mres A)$ is lower semicontinuous with respect to $\F$ and $A \to \M(S \mres A)$ is a Radon measure. A flat chain $S$ in $\F_q(\R^{n+k};\, \G)$ is said to be supported on a closed set $K$ if for any $V$ an open neighborhood of $K$ one can approximate $S$ with respect to $\F$ by chains in $\P_q(\R^{n+k};\, \G)$ which are supported on $V$, in the sense that any polyhedron of the chain is contained in $V$. If $S$ is supported on some compact set, one can define its support, $\spt S$, as the smallest of such sets. Given $K$ a closed subset of $\R^{n+k}$, the set $\F_q(K;\, \G)$ is defined as the set of chains in $\F_q(\R^{n+k};\, \G)$ which are supported in $K$. Likewise, one defines $\M_q(K;\, \G)$. One readily checks that both $\F_q(K;\, \G)$  and $\M_q(K;\, \G)$ are closed with respect to $\F$-convergence. 

Finally, let us recall the definition of relative flat chains on an open set. Given $U$ an open subset of $\R^{n+k}$ and $K$ a closed set contaning $U$, the space $\F_q(U;\, \G)$ is defined as the quotient $\F_q(K;\, \G)/\F_q(\R^{n+k}\setminus U; \, \G)$. It can be shown that such a definition does not depend on the closed set $K$. The space $\F_q(U;\, \G)$ is a complete Abelian group with respect to the quotient norm
\begin{equation} \label{flatrelative}
\F_U(S):= \inf\left\{\F(R): R \in \F_q(\R^{n+k};\, \G), \hspace{2mm} \spt(R) \subseteq K, \hspace{2mm} \spt(R-S) \subseteq K \setminus U   \right\}.
\end{equation}
In the right-hand side of~\eqref{flatrelative}, we have implicitely identified~$S\in\F_q(U; \, \G)$ with an arbitrary representative of the equivalence class~$S$ in~$\F_q(K; \, \G)$. 
We will implicitely make the same identification below, when convenient. The quotient norm defined in~\eqref{flatrelative} can be equivalently written as
\begin{align*}
\F_U(S)= \inf\{ & \M( P \mres U) + \M(Q \mres U): P \in \M_{q+1}(\R^{n+k};\, \G), \hspace{2mm} Q \in \M_{q}(\R^{n+k};\, \G),\\
& \spt(S- \partial P-Q) \subseteq \R^{n+k} \setminus U \}.
\end{align*}
If a class~$S\in\F_q(U; \G)$ admits a finite-mass representative in~$\M_q(K; \, \G)$, still denoted~$S$ by abuse of notation, we write~$\M_U(S) := \M(S\mres U)$. 
Equivalently, $\M_U$ could be defined by approximating elements of~$\F_q(U;\, \G)$ by polyhedral chains supported in~$U$.

\subsection{Topological singular sets}
Following~\cite{CO1,CO2}, we recall the definition and some properties of the operator $\S$ which describes the topological singular sets of vector-valued maps. We also proof some technical results which will be needed in the sequel. We consider the group $\GN$ endowed with the norm $\abs{\, \,}_k$ defined in Subsection \ref{normsonhomotopygroups}. We will use the definitions and results stated in Subsection \ref{Subs:flat-chains} for flat chains arising from such a specific choice of group and norm. The associated flat norm and the mass will be denoted by~$\F_k$, $\M_k$, respectively. These are as in~\cite{CO2}, where the subscript $k$ is not written. Given a bounded domain $U \subseteq \R^{n+k}$ and $S_1, S_2 \in \M_q(\overline{U}; \, \GN)$, we say that $S_1$ and $S_2$ are \textit{cobordant} in $\overline{U}$ if and only if there is $R\in \M_{q+1}(\overline{U}; \, \GN)$ such that  
\begin{align}\label{equivalent}
\partial R=S_1-S_2.
\end{align}
One has that~\eqref{equivalent} defines an equivalence relation on $\M_q(\overline{U}; \, \GN)$ and, moreover, cobordism classes are closed with respect to the flat-norm convergence (see Lemma~\ref{lemma:cobordism} below). The flat norm and mass for flat chains relative to $U$ will be denoted by $\F_{U,k}$ and $\M_{U,k}$, respectively.

Let~$\delta^*\in (0, \, \dist(\NN, \, \X))$ be fixed, and let $B^* := B^m_{\delta^*}\subseteq\R^m$. Let 
\[
 Y(U) := L^1(B^*, \, \F_{n}(U; \, \GN))
\]
be the set of Lebesgue-measurable maps
\begin{equation*}
S\colon B^*\to S_y \in \F_{n}(U; \, \GN)
\end{equation*}
such that
\[
\norm{S}_{Y(U)} := \int_{B^*} \F_{U,k}(S_y) \, \d y < +\infty.
\]
The set $Y(U)$ is then a complete normed modulus when endowed with the norm~$\|\cdot\|_{Y}$ and the space $\F_{n}(U; \, \GN)$ embeds canonically into~$Y(U)$. Then, the topological singular set for maps in $X(U):= W^{1,k-1}(U,\, \mathbb{R}^m) \cap L^\infty(U,\,\mathbb{R}^m)$ is defined in~\cite{CO1} as an operator $\S^U \colon X(U) \to Y(U)$.  Its construction  relies on the choice of a particular retraction map~$\RR$, as given by Lemma~\ref{lemma:retractionmap}. In case~$u$ is an~$\NN$-valued map, i.e.~$u\in W^{1,k-1}(U, \, \NN)$, the operator~$\S^U(u)$ keeps track of topological singularities of~$u$ and coincides with an object defined previously by Pakzad and Rivière~\cite{PakzadRiviere}. In particular, since the construction of~\cite{PakzadRiviere} does not make use of retraction maps, for~$u\in W^{1,k-1}(U, \, \NN)$ the object~$\S^U(u)$ turns out to be independent of the choice of~$\X$ and~$\RR$.
As shown in~\cite[Corollary 1]{CO1}, if $u_0$ and $u_1$ in $X(U)$ coincide in an open set $V \subseteq U$, then $\mathrm{spt}(\S_y^U(u_0)-\S_y^U(u_1)) \subseteq \overline{U}\setminus V$ for a. e. $y \in \mathbb{R}^m$. That is, $\S^U$ is well behaved with respect to restrictions. Therefore, 
we shall simply write $\S$ and drop the superscript indicating the domain, which will be clear from the context in all cases. In order to complete this subsection, we collect a few properties of $\S$ that will be useful in the sequel:
\begin{prop}\label{prop:operator_S}
Let $U\subseteq\R^{n+k}$ be a bounded, Lipschitz domain. Then,
\begin{enumerate}
	\item\label{S_Nvalued} If~$u \in W^{1,k-1}(U, \, \NN)$, then one has $\S_{y_1}(u)=\S_{y_2}(u)$ for a.e. $y_1, y_2 \in B^*$.
	\item\label{S:cobord-app} If~$u_0$ and $u_1$ belong to $W^{1,k}(U, \, \mathbb{R}^m) \cap L^\infty(U,\,\mathbb{R}^m)$ and
	\begin{equation*}
	u_{0|\partial U} = 
	u_{1|\partial U}\in W^{1-1/k, k}(\partial U, \, \mathbb{R}^m) \cap L^\infty(\partial U,\, \mathbb{R}^m)
	\end{equation*}
	in the sense of traces, then $\S_{y_0}(u_0)$ and $\S_{y_1}(u_1)$ are cobordant for a.e.~$y_0$, $y_1\in B^*$.
	\item \label{S_rel_boundary} For any $u \in X(U)$ and a. e. $y \in \mathbb{R}^m$, $\S_y(u)$ is a relative boundary of finite mass in $\Omega$,
	\item Let $U$ be an open subset of $\Omega$ and $u_0$, $u_1$ belong to $X(U)$. Let $\Lambda$ be equal to the maximum between $\lVert u_0 \rVert_{L^\infty(U)}$ and $\lVert u_1 \rVert_{L^\infty(U)}$. Then,
	\begin{equation}\label{basic_estimate_flat_norm_0}
	\int_{B^*}\F_{U,k}(\S_y(u_{0}) - \S_y(u_{1})) \mathrm{d}y \leq C_{\Lambda} \int_{U} \left( \abs{\nabla u_0}^{k-1} + \abs{\nabla u_{1}}^{k-1} \right)\vert u_0-u_{1} \vert \, \d x,
	\end{equation}
	where $C_{\Lambda}$ depends only on $\Lambda$, $k$ and $\NN$.
\end{enumerate}
\end{prop}
The proof of~\ref{S_Nvalued} in Proposition~\ref{prop:operator_S} can be found in~\cite[Proposition 3]{CO1}. As for~\ref{S:cobord-app}, its proof is given in~\cite[Proposition 2]{CO1}. The proof of~\ref{S_rel_boundary} is given in~\cite[$(P_3)$ Theorem 3.1]{CO1}. Regarding estimate~\eqref{basic_estimate_flat_norm}, it corresponds to~\cite[$(P_5)$ Theorem 3.1]{CO1}. 

\begin{remark} \label{rk:flatcontinuity}
Actually, it is possible to write explicitely the difference~$\S_y(u_0) - \S_y(u_1)$ as a boundary, in such a way that the estimate~\eqref{basic_estimate_flat_norm_0} follows. Indeed, given~$u_0\in X(U)$, $u_1\in X(U)$ with~$\norm{u_0}_{L^\infty(U)}\leq\Lambda$, $\norm{u_0}_{L^\infty(U)}\leq\Lambda$, let~$u\colon [0, \, 1]\times U\to \R^m$ be given as~$u(t, \, x) := (1-t) u_0(x) + t u_1(x)$, and let
 \begin{equation} \label{connection}
  R_y := \pi_{*}\S_y(u),
 \end{equation}
 where~$\pi_{*}$ denotes the push-forward through the canonical projection~$\pi\colon [0, \, 1]\times\R^{n+k}\to\R^{n+k}$. (See~e.g.~\cite[Section~5]{Fleming} for the definition of~$\pi_{*}$). Then, $\spt(\S_y(u_1) - \S_y(u_0) - \partial R_y)\subseteq\R^{n+k}\setminus U$ and we have the estimate
 \[
  \int_{B^*} \M_U(R_y) \, \d y 
  \leq C_{\Lambda} \int_{U} \left( \abs{\nabla u_0}^{k-1} + \abs{\nabla u_{1}}^{k-1} \right)\vert u_0-u_{1} \vert \, \d x,
 \] 
 which implies~\eqref{basic_estimate_flat_norm_0} by definition of the flat norm.
 These facts follow from~\cite[Proposition~4]{CO1}, combined with a density argument.
\end{remark}

As a consequence of~\ref{S_Nvalued} in Proposition~\ref{prop:operator_S}, the operator $\S$ is constant a. e. when applied to $\NN$-valued maps. Hence, given $u \in W^{1,k-1}(U,\,\NN)$ we shall identify $\S(u)$ with the flat chain $S_u \in \F_n(U;\,\GN)$ such that $\S_y(u)=S_u$ for a. e. $y \in B^*$. As a consequence,~\eqref{basic_estimate_flat_norm_0} reads
\begin{equation}\label{basic_estimate_flat_norm}
 \F_{U,k}(\S(u_{0}) - \S(u_{1})) \lesssim \int_{U} \left( \abs{\nabla u_0}^{k-1} + \abs{\nabla u_{1}}^{k-1} \right)\vert u_0-u_{1} \vert \,dx,
\end{equation}
whenever $u_0$ and $u_1$ belong to $W^{1,k-1}(U,\,\NN)$. Finally, by~\ref{S:cobord-app} in Proposition~\ref{S:cobord-app}, one can show that any $v \in W^{1-\frac{1}{k},k}(U, \NN)$ induces a cobordism class~$\mathscr{C}(U, \, v)\subseteq\M_n(\overline{U}; \, \GN)$ such that
\begin{equation}\label{C_Omega_v}
\S_y(u)\in\mathscr{C}(U, \, v) \quad
\textrm{ for any } u\in W^{1,k}_v(U, \, \R^m)
\cap L^\infty(U, \, \R^m) 
\textrm{ and a.e. } y\in B^*.
\end{equation}
Actually, it is possible to define~$\S_y(u)$ not only when~$u\in W^{1,k-1}(U, \, \R^m)\cap L^\infty(U, \, \R^m)$, but also when~$u\in W^{1,k}(U, \, \R^m)$. In the latter case, the higher integrability of the gradient compensates for the possible unboundedness of~$u$.
Moreover, the property~\eqref{C_Omega_v} remains true for arbitrary~$u\in W^{1,k}_v(U, \, \R^m)$ (see~\cite[Appendix~B]{CO2}), which is helpful in dealing with generalized Ginzburg-Landau functionals of the form~\eqref{GL_intro_general}. In this paper, we will not need such generality, since we are only interested in maps with values in a compact manifold. However, we will need further properties of the class~$\mathscr{C}(U, \, v)$. First of all, we provide a proof of the fact that~$\mathscr{C}(U, \, v)$ is closed with respect to the~$\F_k$-norm.

\begin{lemma} \label{lemma:cobordism}
Let~$U\subseteq\R^{n+k}$ be a bounded domain with Lipschitz boundary. Let~$(S_i)_{i\in\N}$, $(R_i)_{i\in\N}$ be sequences in~$\F_{n}(\overline{U}; \, \pi_{k-1}(\NN))$, $\F_{n+1}(\overline{U}; \, \pi_{k-1}(\NN))$ respectively. Suppose that~$S_i = \partial R_i$ for any~$i\in\N$ and that~$\F_{k}(S_i - S)\to 0$ as~$i\to\infty$, for some~$S\in\F_n(\overline{U}; \, \pi_{k-1}(\NN))$.
 Then, there exists a finite-mass chain~$R\in\M_{n+1}(\overline{U}; \, \pi_{k-1}(\NN))$ such that~$S = \partial R$.
 In particular, if a sequence in~$\mathscr{C}(U, \, v)$ converges to a finite-mass chain with respect to the~$\F_k$-norm, then the limit, too, belongs to~$\mathscr{C}(U, \, v)$.
\end{lemma}

In case the sequence~$(S_i)_{i\in\N}$ satisfies a uniform bound on the mass, then Lemma~\ref{lemma:cobordism} follows rather easily from the isoperimetric inequality (see e.~g.~\cite[(7.6)]{Fleming}). Here, however, we do \emph{not} assume that the chains~$S_i$, $R_i$ have finite mass. We break down the proof of Lemma~\ref{lemma:cobordism} into separate statements, i.~e.~Lemma~\ref{lemma:flatsupport}, Lemma~\ref{lemma:finitemass} and Lemma~\ref{lemma:smallcycles}. In the proofs, we will use the following notation: given an open set~$U\subseteq\R^{n+k}$ and~$\rho > 0$, we define
\begin{equation} \label{Urho}
 U_\rho := \left\{x\in\R^{n+k}\colon\dist(x, \, U) < \rho\right\} \! .
\end{equation}

If~$U$ is bounded and has Lipschitz boundary, and~$\rho> 0$ is small enough, then there exists a Lipschitz map~$\pi_U\colon \overline{U_\rho}\to\overline{U}$ that coincides with the identity on~$\overline{U}$. (This follows, e.g., from~\cite[Propositon~8.1]{ABO2}). Given a chain~$S$ supported in~$\overline{U_\rho}$, we will denote by~$\pi_{U,*}S$ the push-forward of~$S$ via~$\pi_U$, as defined in~\cite[Section~5]{Fleming}. By construction, $\pi_{U,*}S$ is supported in~$\overline{U}$ and has finite mass if~$S$ has. Moreover, $\pi_{U,*}$ commutes with the boundary operator, i.~e.~$\pi_{U,*}\partial S = \partial (\pi_{U,*}S)$.

\begin{lemma} \label{lemma:flatsupport}
Let~$U\subseteq\R^{n+k}$ be a bounded domain with Lipschitz boundary. Let~$(S_i)_{i\in\N}$ be a sequence in~$\F_{n}(\overline{U}; \, \pi_{k-1}(\NN))$ and let~$S\in\F_n(\overline{U}; \, \pi_{k-1}(\NN))$ be such that~$\F_k(S_i - S)\to 0$ as~$i\to\infty$. Then, there exist chains~$P_i\in\M_{n}(\overline{U}; \, \pi_{k-1}(\NN))$ and~$Q_i\in\M_{n+1}(\overline{U}; \, \pi_{k-1}(\NN))$ such that 
 \begin{gather}
  S_i - S = P_i + \partial Q_i \qquad \textrm{for any } i\in\N, \label{flatsuppPQ}\\
  \M_k(P_i) + \M_k(Q_i) \to 0 \qquad \textrm{as } i\to\infty. \label{flatsuppmass}
 \end{gather}
\end{lemma}
\begin{proof}
The statement would be an immediate consequence of the definition of the flat norm, were it not that we must ensure that the chains~$P_i$, $Q_i$ are supported in~$\overline{U}$. We obtain this by an argument based on truncations and projections.
 By definition of~$\F_k$, there exist sequences~$(\widetilde{P}_i)_{i\in\N}$ and $(\widetilde{Q}_i)_{i\in\N}$ in $\F_n(\R^{n+k}; \, \pi_{k-1}(\NN))$ and $\F_{n+1}(\R^{n+k}; \, \pi_{k-1}(\NN))$ respectively such that
 \begin{equation} \label{flatsupp1}
  S_i - S = \widetilde{P} + \partial \widetilde{Q}_i \quad \textrm{for any } i\in\N, \qquad 
  \M_k(\widetilde{P}_i) + \M_k(\widetilde{Q}_i) \to 0 \quad \textrm{as } i\to\infty.
 \end{equation}
 Let~$\rho > 0$ be small enough that the projection~$\pi_U\colon\overline{U_\rho}\to\overline{U}$ is well-defined and Lipschitz-continuous. For any~$i\in\N$ and any~$t\in (0, \, \rho)$, we define 
 \[
  \Gamma_i(t) := \partial\left(\widetilde{Q}_i\mres U_t\right) - \left(\partial\widetilde{Q}_i\right)\mres U_t.
 \]
We claim that~$\Gamma_i(t)$ is well-defined, with the possible exception of a negligible set of~$t$'s. Indeed, the restriction~$\widetilde{Q}_i\mres U_t$ is always well-defined, because~$\widetilde{Q}_i$ has finite mass. If~$S_i - S$ has finite mass, then~$\partial Q_i$, too, has finite mass and~$(\partial\widetilde{Q}_i)\mres U_t$ is well-defined for any~$t\in (0, \, \rho)$. Moreover, in this case we have~\cite[Theorem~5.3]{Fleming}
 \begin{equation} \label{flatsupp2}
  \int_0^\rho \M_k\left(\Gamma_i(t)\right) \d t 
  \leq \M_k\!\left(\widetilde{Q}_i\mres\left(\overline{U_\rho}\setminus U\right)\right) 
  \leq \M_k(\widetilde{Q}_i).
 \end{equation} 
 In case~$S_i - S$ has infinite mass, the restriction~$(\partial\widetilde{Q}_i)\mres U_t$ is still well-defined for a.e.~$t\in(0, \, \rho)$. Indeed, the restriction operator induces a well-defined and continuous map
 \[
  R\colon\F_n(\R^{n+k}; \, \pi_{k-1}(\NN)) \to L^1((0, \, \rho); \, \F_{n}(\R^{n+k}; \, \pi_{k-1}(\NN)))
 \]
 where~$R(T) (t) := T\mres U_t$ for any chain~$T$ and a.e.~$t\in(0, \, \rho)$ (see e.g.~\cite[Theorem~5.2.3.(2)]{DePauwHardt}). Then, a density argument shows that~\eqref{flatsupp2} remains valid even when~$S_i - S$ has infinite mass. The inequality~\eqref{flatsupp2}, combined with~\eqref{flatsupp1}, implies
 \begin{equation} \label{flatsupp3}
  \int_0^\rho \M_k\left(\Gamma_i(t)\right) \d t \to 0 \qquad \textrm{as } i\to\infty.
 \end{equation}
 Therefore, we can extract a (non-relabelled) subsequence in such a way that~$\M_k(\Gamma_i(t))\to 0$ as~$i\to\infty$ for a.e.~$t\in (0, \, \rho)$. 
 
 Now, since~$S_i - S$ is supported in~$\overline{U}\subseteq U_t$, we have
 \[
  S_i - S = \left(S_i - S\right)\mres U_t 
  = \widetilde{P}_i\mres U_t - \Gamma_i(t) + \partial\left(\widetilde{Q}_i\mres U_t\right) \!.
 \]
 By construction, the chains~$S_i - S$, $\widetilde{P}_i\mres U_t$, $\widetilde{Q}_i\mres U_t$, $\Gamma_i(t)$ are supported in~$\overline{U_\rho}$. Therefore, their push-forwards via~$\pi_{U,*}$ are well-defined and satisfy
 \begin{gather} 
  \M_k\!\left(\pi_{U,*}\left(\widetilde{P}_i\mres U_t\right)\right) \leq (\mathrm{Lip}\, \pi_U)^n \, \M_k(\widetilde{P}_i), \qquad 
  \M_k\!\left(\pi_{U,*}\Gamma_i(t)\right) \leq (\mathrm{Lip}\, \pi_U)^n \, \M_k\!\left(\Gamma_i(t)\right)\label{flatsupp3.5} \\
  \M_k\!\left(\pi_{U,*}\left(\widetilde{Q}_i\mres U_t\right)\right) \leq (\mathrm{Lip}\, \pi_U)^{n+1} \, \M_k(\widetilde{Q}_i)\label{flatsupp3.6}
 \end{gather}
 because of the area formula (see e.~g.~\cite[(5.1)]{Fleming}). Moreover, we have
 \begin{equation} \label{flatsupp4}
  S_i - S = \pi_{U,*}\left(S_i - S\right) 
  = \pi_{U,*}\left(\widetilde{P}_i\mres U_t - \Gamma_i(t)\right) 
  + \partial \pi_{U,*} \left(\widetilde{Q}_i\mres U_t\right) \! .
 \end{equation}
 The equality~$S_i - S = \pi_{U,*}(S_i - S)$ follows because~$S_i$, $S$ are supported in~$\overline{U}$ and~$\pi_{U,*}$ coincides with the identity there. We define
 \[
  P_i := \pi_{U,*}\left(\widetilde{P}_i\mres U_t - \Gamma_i(t)\right), \qquad
  Q_i := \pi_{U,*} \left(\widetilde{Q}_i\mres U_t\right) \!.
 \]
 The equality~\eqref{flatsuppPQ} is an immediate consequence of~\eqref{flatsupp4}, while~\eqref{flatsuppmass} follows from~\eqref{flatsupp1}, \eqref{flatsupp3.5}, and~\eqref{flatsupp3.6}.
\end{proof}

\begin{lemma} \label{lemma:finitemass}
Let~$U\subseteq\R^{n+k}$ be a bounded domain with Lipschitz boundary. For any chain~$R\in\F_{n+1}(\overline{U}; \, \pi_{k-1}(\NN))$ there exists a \emph{finite-mass} chain-$R^\prime\in\M_{n+1}(\overline{U}; \, \pi_{k-1}(\NN))$ such that $\partial R^\prime = \partial R$.
\end{lemma}
\begin{proof}
Let~$\rho > 0$ be small enough that the projection~$\pi_U\colon \overline{U_\rho}\to\overline{U}$ is well-defined and Lipschitz. By definition of support, there exists a sequence of polyhedral $n$-chains~$R_i$, supported in~$\overline{U_\rho}$, such that $\F(R_i -R) \to 0$ as~$i\to\infty$. By Lemma~\ref{lemma:flatsupport}, there exists sequences~$(P_i)_{i\in\N}\subset\M_n(\overline{U_\rho}; \, \pi_{k-1}(\NN))$, $(Q_i)_{i\in\N}\subset\M_{n+1}(\overline{U_\rho}; \, \pi_{k-1}(\NN))$ such that $R_i - R = P_i + \partial Q_i$ and~$\M_k(P_i) + \M_k(Q_i)\to 0$. 
 By taking the push-forward under~$\pi_U$, we obtain
 \[
  R = \pi_{U,*}R = \pi_{U,*} R_i + \pi_{U,*} P_i +\partial \left(\pi_{U,*} Q_i\right) \! .
 \]
 Letting~$R^\prime := \pi_{U,*} R_i + \pi_{U,*} P_i$ for an arbitrary~$i$, we see that~$R^\prime$ is supported in~$\overline{U}$, has finite mass and satisfies~$\partial R = \partial R^\prime$.
\end{proof}

\begin{lemma} \label{lemma:smallcycles}
For any bounded domain~$U\subseteq\R^{n+k}$ with Lipschitz boundary, there exists a number~$\delta_U > 0$ with the following property: if~$P\in\mathrm{M}_n(\overline{U}; \, \pi_{k-1}(\NN))$ is such that~$\partial P = 0$ and~$\M_k(P) \leq \delta_U$, then there exists~$T\in\M_{n+1}(\overline{U}; \, \pi_{k-1}(\NN))$ such that~$\partial T = P$.
\end{lemma}
\begin{proof}
This lemma is a variant of the isoperimetric inequality for flat chains, see e.~g.~\cite[(7.6)]{Fleming}. Let~$\eps > 0$. According to the deformation theorem~\cite[Theorem~7.3]{Fleming} (see also~\cite{White-Deformation}), any finite-mass $n$-chain~$P$ can be written in the form 
 \begin{equation} \label{deformationthm}
  P = Q + A + \partial B,
 \end{equation}
 Here~$Q$ is a polyhedral $n$-chain of the form~$Q = \sum_{j=1}^N\sigma_j\llbracket K_j\rrbracket$, where~$\sigma_j\in\pi_{k-1}(\NN)$ and each~$K_j$ is a~$n$-cube of size~$\eps$, while~$A$, $B$ are finite-mass chains that satisfy the following estimates:
 \begin{gather}
  \M_k(Q) \leq C_{\mathrm{def}} \left(\M_k(P) + \M_k(\partial P)\right), \qquad \M_k(\partial Q) \leq C_{\mathrm{def}} \M_{k}(\partial P) \label{defth1} \\
  \M_k(A) \leq C_{\mathrm{def}} \eps\, \M_k(\partial P), \qquad \M_k(B) \leq C_{\mathrm{def}} \eps \,\M_k(P) \label{defth2}. 
 \end{gather}
 The constant~$C_{\mathrm{def}}$ can be chosen in a way that depends only on~$n + k$.
 Moreover, the supports of~$Q$, $A$, $B$ are contained in the~$2(n+k)\eps$-neighbourhood of the support of~$P$.
 Suppose that~$\partial P = 0$ and $\M(P) \leq \delta_U$, for some~$\delta_U > 0$ to be chosen later. We define~$\eps$ in such a way that
 \begin{equation} \label{smallcycles0}
  \alpha_k \eps^k = 2 C_{\mathrm{def}} \M(P),
 \end{equation}
 where~$\alpha_k$ is given by~\eqref{alpha_p}, and apply the deformation theorem with this choice of~$\eps$.
 The estimate~\eqref{defth2} implies that~$A = 0$, while~\eqref{defth1} gives
 \begin{equation} \label{smallcycles1}
  \M_k(Q) \leq C_{\mathrm{def}} \, \M(P)
 \end{equation}
 On the other hand, since~$Q$ is a finite combinations of $n$-cubes of size~$\eps$, if~$Q\neq 0$ then we must have
 \begin{equation} \label{smallcycles2}
  \M_k(Q) \geq \alpha_k \eps = 2C_{\mathrm{def}} \, \M(P)
 \end{equation}
 because of~\eqref{smallcycles0}. The inequality~\eqref{smallcycles2} contradicts~\eqref{smallcycles1}, so we must have~$Q = 0$. Then, \eqref{deformationthm} reduces to the equality~$P = \partial B$. If~$\delta_U$ is taken small enough, then the projection~$\pi_U$ onto~$U$ is well-defined and Lipschitz-continuous in a neighbourhood of~$\spt B$, so we can write
 \[
  P = \pi_{U,*} P = \partial\left(\pi_{U,*} B\right) \!.
 \]
 The lemma follows with~$T := \pi_{U,*}B$. 
\end{proof}

\begin{proof}[Proof of Lemma~\ref{lemma:cobordism}]
Let~$(S_i)_{i\in\N}\subset\F_{n}(\overline{U}; \, \pi_{k-1}(\NN))$, $S\in \F_{n}(\overline{U}; \, \pi_{k-1}(\NN))$ be such that~$\F_{k}(S_i - S)\to 0$ as~$i\to\infty$. By Lemma~\ref{lemma:flatsupport}, there exist~$P_i\in\M_{n}(\overline{U}; \, \pi_{k-1}(\NN))$ and~$Q_i\in\M_{n+1}(\overline{U}; \, \pi_{k-1}(\NN))$ such that
 \begin{equation} \label{cobord1}
  S_i - S = P_i + \partial Q_i \qquad \textrm{for any } i\in\N
 \end{equation}
 and 
 \begin{equation} \label{cobord2}
  \M_k(P_i) + \M_k(Q_i) \to 0 \qquad \textrm{as } i\to\infty.
 \end{equation}
 Suppose now that each~$S_i$ takes the form~$S_i = \partial R_i$ for some~$R_i\in\M_{n+1}(\overline{U}; \, \pi_{k-1}(\NN))$. Then, for each~$i\in\N$ we have~$\partial S_i = 0$ and hence, $\partial S = 0$, because the boundary operator is continuous with respect to~$\F$-convergence. By taking the boundary of both sides of~\eqref{cobord1}, we deduce that~$\partial P_i = 0$ for any~$i\in\N$. 
 By Lemma~\ref{lemma:smallcycles} and~\eqref{cobord2}, for any~$i$ small enough there exists~$T_i\in\M_{n+1}(\overline{U}; \, \pi_{k-1}(\NN))$ such that~$P_i = \partial T_i$. Then, we can write
 \[
  S = S_i - P_i - \partial Q_i = \partial\widetilde{R}_i,
 \]
 where~$\widetilde{R}_i := R_i - T_i - Q_i$ is supported in~$\overline{U}$. By Lemma~\ref{lemma:finitemass}, we can assume without loss of generality that~$\widetilde{R}_i$ has finite mass. This completes the proof.
\end{proof}

Now, we state a consequence of Lemma~\ref{lemma:cobordism} that will be useful in the analysis of Section~\ref{Section:lower_bound}.

\begin{lemma} \label{lemma:CUv}
Let~$U$, $U^\prime$ be bounded domains with Lipschitz boundary in~$\R^{n+k}$, with~$U\csubset U^\prime$. Let $v\in W^{1-1/k,k}(\partial U, \, \NN)$ and let~$u\in W^{1,k}(U^\prime, \, \R^{n+k})\cap L^\infty(U^\prime, \, \R^{n+k})$ be a map with trace~$u$ on~$\partial U$. Let~$(u_i)_{i\in\N}$ be a sequence in~$W^{1,k-1}(U^\prime, \, \R^{n+k})\cap L^\infty(U^\prime, \, \R^{n+k})$ such that
 \[
  u_i(x) = u(x) \quad \textrm{for a.e. } x\in U^\prime \setminus U 
  \textrm{ and any } i, \qquad
  u_i(x)\in\NN \quad \textrm{for a.e. } x\in U \textrm{and any } i.
 \]
 Let~$S\in\F_n(U^\prime; \, \pi_{k-1}(\NN))$ be such that~$\M_{U,k}(S) < +\infty$ and
 \begin{equation} \label{hp:CUv}
  \F_{U^\prime,k}\left(\S_y(u_i) - S\right) \to 0
 \end{equation}
 as~$i\to+\infty$, for any~$y$ in a subset~$A\subset B^*$ of positive measure. Then, $S\mres\overline{U}\in\mathscr{C}(U, \, v)$.
\end{lemma}
\begin{proof}
Heuristically, we would like to apply Statement~\ref{S:cobord-app} in Proposition~\ref{prop:operator_S} to conclude that~$\S_y(u_i)$ and~$\S_y(u)$ are cobordant in~$\overline{U^\prime}$ for a.e.~$y$, hence~$\S_y(u_i)\mres\overline{U}\in\mathscr{C}(U, \, v)$, then pass to the limit as~$i\to\infty$. Unfortunately, this is not possible because Statement~\ref{S:cobord-app} does not apply to the maps~$u_i\in W^{1,k-1}(U^\prime, \, \NN)$. (In fact, $\S_y(u_i)$ may not have finite mass and its restriction to~$\overline{U}$ may not be well-defined.) Nevertheless, by Remark~\ref{rk:flatcontinuity}, for every~$i\in\N$ and almost every~$y\in B^*$ there exists a chain of finite mass~$R_{y,i}\in\M_{n+1}(\overline{U}; \, \pi_{k-1}(\NN))$ (defined as in~\eqref{connection}) that satisfies
 \begin{equation} \label{CUv1}
  \S_y(u_i) - \S_y(u) = \partial R_{y,i}.
 \end{equation}
 Actually, Remark~\ref{rk:flatcontinuity} only implies that~$\spt(\S_y(u_i) - \S_y(u) - \partial R_{y,i})\subseteq\R^{n+k}\setminus U^\prime$ and that~$R_{y,i}$ has finite mass inside~$U^\prime$. However, since~$u = u_i$ out of~$U$, we have that both~$\S_y(u_i) - \S_y(u)$ and~$R_{y,i}$ are supported in~$\overline{U}\subseteq U^\prime$ and hence, \eqref{CUv1} follows.
 In fact, we can even write~\eqref{CUv1} as
 \begin{equation} \label{CUv2}
  \S_y(u_i|_{U}) - \S_y(u|_{U}) = \partial R_{y,i}.
 \end{equation}
 On the other hand, since~$\S_y(u_i|_{U^\prime\setminus \overline{U}) = \S_y(u|_{U^\prime\setminus\overline{U}}})$ does not depend on~$i$, the assumption~\eqref{hp:CUv} implies that~$\F_{U^\prime,k}(\S_y(u_i|_{U}) - S\mres\overline{U})\to 0$ as~$i\to+\infty$, for a.~e.~$y\in A$. Since~$\S_y(u_i|_U)$, $S\mres\overline{U}$ are supported in the compact set~$\overline{U}\subseteq U^\prime$, convergence with respect to the~$\F_{U^\prime,k}$-norm implies convergence with respect to the~$\F_k$-norm~\cite[Remark~2.2, Equation~(2.10)]{CO1}. In other words, we have
 \begin{equation} \label{CUv3}
  \F_{k}\left(\S_y(u_i|_U) - S\mres\overline{U}\right) \to 0
  \qquad \textrm{as } i\to\infty, \textrm{ for any } y\in A.
 \end{equation}
 Combining~\eqref{CUv2} and~\eqref{CUv3} with Lemma~\ref{lemma:cobordism}, we conclude that for a.~e.~$y\in B^*$, the chain~$S\mres\overline{U} - \S_y(u|_{U})$ is the boundary of a finite-mass chain supported in~$\overline{U}$. Therefore, $S\mres\overline{U}\in\mathscr{C}(U, \, v)$.
\end{proof}
For further properties of the operator~$\S$, the reader is referred to~\cite{CO1,CO2}.

\section{Lower \texorpdfstring{$p$}{p}-energy bounds via the ball construction}\label{sect_lower_p_energy_bounds}
Our aim in this section is to obtain lower bounds on the $p$-Dirichlet 
energy of maps from a bounded domain $D \subseteq \R^k$ to $\NN$ in terms of 
the homotopy class of their boundary datum, 
along the lines of~\cite{Jerrard, Sandier}. Notice that we are now focusing on the case $n=0$, meaning that we are lead to the study of point-like singularities, see Figure~\ref{sincriticaldim}. The results of this section will be essential for the proof of Statement~(i) in Theorem~\ref{MainThm}, given  
in Section~\ref{Section:lower_bound} below.

\subsection{Lower bounds on the \texorpdfstring{$p$}{p}-energy of maps in the critical dimension}

Let~$D\subseteq\R^k$ be a given domain,
which we assume to be bounded. We also assume that~$\overline{D}$ is homeomorphic 
to the closed ball~$\overline{B}^k$, so that~$\partial D$
is homeomorphic and the homotopy class of
any continuous map~$\partial D\to\NN$ 
is well-defined as an element of~$\pi_{k-1}(\NN)$.
For any~$r>0$, we define
\begin{equation} \label{Gamma_r}
  \Gamma_r := \left\{x\in\overline{D}\colon 
   \dist(x, \, \partial D) \leq r\right\}.
\end{equation}
The main result of this section is the following one:
\begin{prop} \label{prop:lowerbound-p}
Let~$r\in (0, \, 1/2]$, $p\in (k-1, \, k)$,
 let~$u\in W^{1,p}(D, \, \NN)$ be
 a map that is continuous in~$\Gamma_r$,
 and let~$\sigma\in\pi_{k-1}(\NN)$
 be the homotopy class of~$u$ on~$\partial D$.
 Then, there holds
 \begin{equation} \label{bound_JS}
  \int_{D} \abs{\nabla u}^p  \mathrm{d}x
  \geq \frac{\abs{\sigma}_p}{k - p}
   - C\abs{\sigma}_p\log\left(\frac{\abs{\sigma}_p}{\alpha_p \, r}\right) \!,
 \end{equation}
 where~$\alpha_p$ is the number given by~\eqref{alpha_p}
 and~$C$ is a universal constant --- for instance, $C = 5/(\log 2)$.
\end{prop}
The proof of Proposition~\ref{prop:lowerbound-p}
proceeds along the lines of~\cite[Theorem~1.2]{Jerrard}.
We will use the notation
\begin{equation} \label{Lambda_p}
 \Lambda_p(s) := \frac{\left(\alpha_p \, s\right)^{k-p}}{k-p}
 \qquad \textrm{for } s> 0,
\end{equation}
where~$\alpha_p$ is the number given by~\eqref{alpha_p}.
The first main ingredient in the proof of 
Proposition~\ref{prop:lowerbound-p} is the following
estimate, which bounds the energy on a map on a spherical 
shell in terms of its homotopy class.
Let~$B^k_b\setminus\overline{B}^k_a\subseteq\R^k$
be a spherical shell, with~$0 \leq a < b$, and
let~$u\in W^{1, \, p}(B^k_b\setminus\overline{B}^k_a, \, \NN)$.
If~$p > k-1$, then the homotopy class of~$u$
on the sphere~$\partial B^k_\rho$ is well-defined
for almost any value of the radius~$\rho\in (a, \, b)$,
by Fubini's theorem and Sobolev embeddings.
\begin{lemma} \label{lemma:lowerbound-annulus}
 Let~$a$, $b$ be real numbers with~$0 \leq a < b$,
 let~$p > k-1$, let~$u\in W^{1,p}(B_b^k\setminus\overline{B}_a^k, \, \NN)$
 and let~$\sigma\in\pi_{k-1}(\NN)$.
 Assume that, for almost any~$\rho\in (a, \, b)$, 
 the homotopy class of~$u$ restricted to~$\partial B^k_\rho$ is~$\sigma$. Then, there holds
 \begin{equation} \label{lowerbound-annulus}
  \int_{B^k_b\setminus\overline{B}^k_a} \abs{\nabla u}^p \mathrm{d}x
  \geq \abs{\sigma}_p \left(
   \Lambda_p\left(\frac{b}{\abs{\sigma}_p}\right) 
   - \Lambda_p\left(\frac{a}{\abs{\sigma}_p}\right) \right) \! ,
 \end{equation}
 under the agreement that the right-hand side
 of~\eqref{lowerbound-annulus} is zero if~$\sigma = 0$.
\end{lemma}
\begin{figure}
	\centering
	\includegraphics[width=0.5\textwidth]{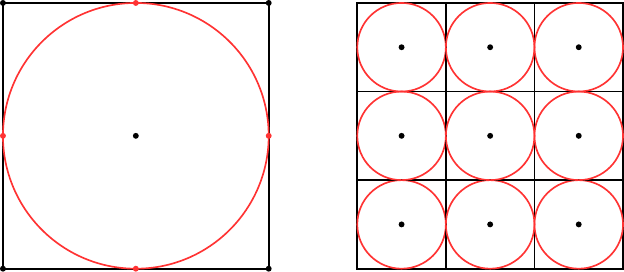}
	\caption{\label{sincriticaldim} Maps have singularities that are uniformly distributed in the critical dimension with $k=2$.
	This figure serves as an illustration of how to obtain a lower bound for the $p$-Dirichlet energy of maps with uniformly distributed singularities, via a suitable rescaling. In the general case, one must use ball constructions to prove the lower bound.}
\end{figure}
\begin{proof}
 If~$\sigma = 0$ there is nothing to prove, 
 so we assume without loss of generality that~$\sigma\neq 0$.
 The definition of~$E_p$ (Equation~\eqref{E_p}),
 combined with a scaling argument, immediately implies that
 \begin{equation} \label{boundsphere}
  \int_{\partial B^k_\rho} \abs{\nablaT u}^p \,\d\H^{k-1}
  \geq E_p(\sigma) \, \rho^{k - p - 1}
 \end{equation}
 for almost any~$\rho\in (a, \, b)$. The definition 
 of~$\abs{\, \,}_p$, Equation~\ref{groupnorm},
 immediately implies that~$E_p(\sigma) \geq \abs{\sigma}_p$.
 Then, by integrating both sides of~\eqref{boundsphere}
 with respect to~$\rho\in (a, \, b)$, we obtain
 \[
  \begin{split}
   \int_{B^k_b\setminus\overline{B}^k_a} \abs{\nabla u}^p \mathrm{d}x
   \geq \frac{\abs{\sigma}_p}{k - p}\left(b^{k-p} - a^{k-p}\right) 
   &= \frac{\abs{\sigma}_p^{1 + k - p}}{k - p} 
    \left(\left(\frac{b}{\abs{\sigma}_p}\right)^{k-p} 
   - \left(\frac{a}{\abs{\sigma}_p}\right)^{k-p}\right) \\
   &\geq \frac{\abs{\sigma}_p \, \alpha_p^{k-p}}{k - p} 
    \left(\left(\frac{b}{\abs{\sigma}_p}\right)^{k-p} 
   - \left(\frac{a}{\abs{\sigma}_p}\right)^{k-p}\right) \! ,
  \end{split}
 \]
 where the last inequality follows by the definition of~$\alpha_p$,
 Equation~\eqref{alpha_p}.
\end{proof}
The main step in the proof of Proposition~ ~\ref{prop:lowerbound-p}
is the so-called ``ball construction'' (see~\cite{Jerrard, Sandier}). We follow ~\cite[Proposition~4.1]{Jerrard}.
Let~$k - 1 < p < k$, $r > 0$
and~$u\in W^{1,p}(D, \, \NN)$ be given,
such that~$u$ is continuous in
the set~$\Gamma_r$ defined by~\eqref{Gamma_r}.
It is not restrictive to assume that~$u$ satisfies
\begin{equation} \label{upper_bd}
 \int_{\Omega} \abs{\nabla u}^p \mathrm{d}x \leq 
  \frac{\abs{\sigma}_p}{k - p},
\end{equation}
where~$\sigma\in\pi_{k-1}(\NN)$ is the homotopy
class of~$u$ on~$\partial\Omega$ ---
for if~\eqref{upper_bd} does not hold, then
\eqref{bound_JS} follows immediately.
Moreover, it suffices to prove the estimate~\eqref{bound_JS}
in case~$u$ is smooth except for a finite number of 
point singularities, for the class of such maps 
is (strongly) dense in~$W^{1,p}(D, \, \NN)$
\cite[Theorem~2]{bethuel-Density}. Therefore, 
we assume without loss of generality that there exists 
a finite set~$S\subseteq D\setminus\Gamma_r$
such that~$u$ is smooth in~$D\setminus S$.

After these reductions, we set some notation.
For each~$a\in S$, we define~$\sigma_a\in\pi_{k-1}(\NN)$
as the homotopy class of~$u$ on a sphere~$\partial B^k_\rho(a)$,
where the radius~$\rho$ is small enough 
that~$\overline{B}^k_\rho(a)\subseteq D$
and~$S\cap \overline{B}_\rho(a) = \{a\}$.
We write~$S_{\mathrm{top}}$ for the set of 
``topologically non-trivial'' singularities, that is,
\begin{equation} \label{S_top}
 S_{\mathrm{top}} := \left\{a\in S\colon \sigma_a\neq 0\right\} \! .
\end{equation}
If~$B$ is a ball (or a set homeomorphic to a ball, 
such as~$D$ itself), the closure of~$B$ is contained in~$D$
and~$\partial B\cap S_{\mathrm{top}} = \emptyset$, 
then we define~$\hc(u, \, \partial B)\in \pi_{k-1}(\NN)$ as
\begin{equation} \label{hc}
 \hc(u, \, \partial B) 
 := \sum_{a\in S_{\mathrm{top}}\cap B} \sigma_a \, .
\end{equation}
If~$u$ is continuous on~$\partial B$, then~$\hc(u, \, \partial B)$
is precisely the homotopy class of~$u$ restricted to~$\partial B$.
However, $u$ is still well-defined if~$\partial B$
contains a point of~$S\setminus S_{\mathrm{top}}$.
(Essentially, defining~$\hc(u,\, \partial B)$
as in~\eqref{hc} allows us to neglect
``topologically trivial'' singularities,
i.e.~elements of~$S\setminus S_{\mathrm{top}}$.) Finally, we will write~$\rad B$ for the radius of a ball~$B$.

\begin{lemma} \label{lemma:ballconstruction}
 Let~$p\in (k-1, \, k)$, $r \in (0, \, 1/2]$, $\tau > 0$
 and~$u\in W^{1,p}(D, \, \NN)$ be given.
 Suppose that~$u$ is smooth in~$D\setminus S$
 for some finite set~$S\subseteq D\setminus\Gamma_r$,
 that~$u$ satisfies~\eqref{upper_bd}, and that
 \begin{equation} \label{tau-k-p}
  4\tau\abs{\hc(u, \, \partial D)}_* \leq r, \qquad 
  \left(\frac{\alpha_p \tau}{2}\right)^{k-p} > \frac{1}{2}
 \end{equation}
 Then, there exists a finite collection~$\mathscr{B}$ of closed balls
 that satisfy the following properties:
 \begin{enumerate}[label=(\roman*)]
  \item $S_{\mathrm{top}}\subseteq \cup_{B\in\mathscr{B}} B$ 
  and~$B\cap S_{\mathrm{top}}\neq\emptyset$ for any~$B\in\mathscr{B}$;
  \item the balls in~$\mathscr{B}$ are contained in~$D$
  and their interiors are pairwise disjoint;
  \item letting~$s:= \min_{B\in\mathscr{B}} 
  \left(\rad B/\abs{\hc(u, \, \partial B)}_p\right)$, we have
  \[
   \int_{B} \abs{\nabla u}^p \mathrm{d}x \geq \frac{\rad B}{s} \Lambda_p(s)
  \]
  for any~$B\in\mathscr{B}$;
  \item $\tau/2 \leq s \leq\tau$;
  \item $\sum_{B\in\mathscr{B}}\rad B 
  \leq 2\tau\abs{\hc(u, \, \partial D)}_p$.
 \end{enumerate}
\end{lemma}
\begin{proof}
 The proof proceeds along the lines of~\cite[Proposition~4.1]{Jerrard}.
 We summarize the argument, focusing on the points that require some adaptations, and refer the reader to~\cite{Jerrard} for the details.
 Let~$s_0$ be a small parameter, with~$0 < s_0 < \tau/2$.
 We write~$B_a$ for the closed ball of centre~$a$
 and radius~$s_0\abs{\sigma_a}_p$. By choosing~$s_0$ 
 small enough (depending on~$u$),
 we can make sure that the balls~$B_a$ are pairwise
 disjoint and contained in~$D$.
 Moreover, Lemma~\ref{lemma:lowerbound-annulus} implies that
 \[
  \int_{B_a} \abs{\nabla u}^p \mathrm{d}x
  \geq \abs{\sigma_a}_p \Lambda_p(s_0)
  = \frac{\rad B_a}{s_0} \Lambda_p(s_0)
 \]
 for any~$a\in S_{\mathrm{top}}$.
 Now, we consider finite collections of closed balls,
 $\mathscr{B}_0$, $\mathscr{B}_1$, \ldots, $\mathscr{B}_N$,
 with suitable properties. The initial collection
 is defined as $\mathscr{B}_0 := \{B_a\colon a\in S_{\mathrm{top}}\}$.
 Each collection, $\mathscr{B}_n$ for~$n\geq 1$, is obtained with
 from the previous one, $\mathscr{B}_{n-1}$,
 by proceeding exactly as in~\cite[proof of the Proposition~4.1]{Jerrard}.
 At each step~$n$, the collection~$\mathscr{B}_n$ 
 satisfies the property~(i) above, as well as
 \begin{itemize}
  \item[(ii$^\prime$)] the interiors of the balls in~$\mathscr{B}_n$
  are pairwise disjoint;
  \item[(iii$^\prime$)] for any~$B\in\mathscr{B}_n$,
  we have
  \[
   \int_{D\cap B} \abs{\nabla u}^p \geq \frac{\rad B}{s_n} \Lambda_p(s_n)
  \]
  where~$s_n:= \min_{B\in\mathscr{B}_n} 
  \left(\rad B/\abs{\hc(u, \, \partial B)}_p\right)$;
  \item[(iv$^\prime$)] $s_n\leq \tau$.
 \end{itemize}
 Moreover, the process terminates after a finite number~$N$
 of steps, and the final collection~$\mathscr{B}_{N}$ satisfies
 \begin{equation} \label{finalballs}
  \frac{\tau}{2} \leq s_N \leq \tau.
 \end{equation}
 The construction of~$(\mathscr{B}_n)_{n=1}^N$ 
 combines ``expansion steps'' and ``amalgamation steps'', 
 which are defined in a rather technical way.
 However, since the arguments in~\cite{Jerrard} carry over
 with no significant difference, we do not provide the proof 
 of properties~(i), (ii$^\prime$)--(iv$^\prime$), \eqref{finalballs} above.
 (The main tool is the estimate provided by 
 Lemma~\ref{lemma:lowerbound-annulus}, which is
 analogous to the one in~\cite[Proposition~3.2]{Jerrard}.)
 Instead, we prove that each ball~$B\in\mathscr{B}_N$
 satisfies~$B\subseteq D$ and that~$\mathscr{B}_N$ satisfies~(v); 
 then, it will follow that the collection~$\mathscr{B} := \mathscr{B}_N$ 
 has all the desired properties.
 
 \smallskip
 \noindent
 \textit{Each ball in~$\mathscr{B}_N$ is contained in~$D$.}
 Suppose, towards a contradiction, that~$B\in\mathscr{B}_N$
 is a ball not entirely contained in~$D$.
 We know that~$B$ contains a singularity of~$u$
 (by Property~(i) above) and that~$u$ is continuous in~$\Gamma_r$,
 by assumption. Therefore, we must have~$\rad B \geq r/2$.
 Then, the assumption~\eqref{tau-k-p} implies
 $\rad B \geq 2\tau\abs{\sigma}_p$, where~$\sigma:=\hc(u, \, \partial D)$.
 From the upper bound~\eqref{upper_bd},
 Property~(iii$^\prime$) above and~\eqref{finalballs},
 we deduce
 \[
  \frac{\abs{\sigma}_p}{k-p}
  \stackrel{\eqref{upper_bd}}{\geq} \int_{B\cap D} \abs{\nabla u}^p \mathrm{d}x 
  \stackrel{(\mathrm{iii}^\prime)}{\geq}
   \frac{\rad B}{s_N} \Lambda_p(s_N)
  \geq \frac{2\tau \abs{\sigma}_p}{k - p} 
   \alpha_p^{k-p} s_N^{k-p-1}
  \stackrel{\eqref{finalballs}}{\geq} 
   \frac{2\abs{\sigma}_p}{k - p} 
   \alpha_p^{k-p} \tau^{k-p}
 \]
 (the function~$s\mapsto s^{k-p-1}$ is decreasing, because~$p > k-1$).
 As a consequence, we obtain the inequality
 \[
  2\alpha_p^{k-p} \tau^{k-p} \leq 1,
 \]
 which contradicts\footnote{To obtain a contradiction here,
 it would be enough to assume that~$(\alpha_p\tau)^{k-p} > 1/2$,
 which is a slightly weaker condition than~\eqref{tau-k-p}.
 However, the assumption~\eqref{tau-k-p}
 is used in other points of the proof 
 --- most importantly, to prove that~(iv$^\prime$)
 persists after an amalgamation step, 
 see~\cite[Proposition~4.1, Step~7]{Jerrard}
 --- and there, we do need the slightly stronger condition 
 $(\alpha_p\tau/2)^{k-p} > 1/2$.
 }
 the assumption~\eqref{tau-k-p}.
 Therefore, we must have~$B\subseteq D$,
 as claimed.
 
 \smallskip
 \noindent
 \textit{$\mathscr{B}_N$ satisfies~(v).}
 The upper bound~\eqref{upper_bd}, Properties~(ii$^\prime$) 
 and~(iii$^\prime$) above and~\eqref{finalballs} imply
 \[
  \begin{split}
   \frac{\abs{\sigma}_p}{k-p}
    \stackrel{\eqref{upper_bd}}{\geq} \int_{D} \abs{\nabla u}^p \mathrm{d}x
    \stackrel{(\mathrm{ii}^{\prime})-(\mathrm{iii}^\prime)}{\geq}
    \sum_{B\in\mathscr{B}_N} 
    \frac{\rad B}{s_N} \Lambda_p(s_N)
   \stackrel{\eqref{finalballs}}{\geq}
    \sum_{B\in\mathscr{B}_N} 
    \frac{\rad B}{s_N} \Lambda_p\left(\frac{\tau}{2}\right) \! .
  \end{split}
 \]
 By writing down explicitely the definition of~$\Lambda_p(\tau/2)$
 (see~\eqref{Lambda_p}), we obtain
 \[
  \left(\frac{\alpha_p\,\tau}{2}\right)^{k-p} 
   \sum_{B\in\mathscr{B}_N} 
   \frac{\rad B}{s_N} \leq \abs{\sigma}_p, 
 \]
 and~(v) follows because of~\eqref{tau-k-p} and~\eqref{finalballs}.
\end{proof}

\begin{proof}[Proof of Proposition~\ref{prop:lowerbound-p}]
 Let~$p\in (k-1, \, k)$, $r\in (0, \, 1/2]$
 and~$u\in W^{1,p}(D, \, \NN)$ be given.
 As explained before, there is no loss of generality 
 in assuming that~$u$ is smooth away from a
 finite set~$S\subseteq D\setminus\Gamma_r$ 
 and that it satisfies~\eqref{upper_bd}. 
 Let~$\sigma := \hc(u, \, \partial D)\in\pi_{k-1}(\NN)$.
 We will prove that~$u$ satisfies
 \begin{equation} \label{bound_JS0}
  \int_{D} \abs{\nabla u}^p \mathrm{d}x
  \geq \frac{\abs{\sigma}_p}{k - p} - \frac{\abs{\sigma}_p}{\log 2} \log\left(\frac{16\abs{\sigma}_p}{\alpha_p \, r}\right) \! .
 \end{equation}
 Keeping in mind that~$\abs{\sigma}_p/\alpha_p\geq 1$
 because of~\eqref{alpha_p} and that~$r\leq 1/2$, 
 from~\eqref{bound_JS0} we obtain the desired 
 estimate~\eqref{bound_JS}, with a constant~$C = 5/\log 2$.
 In order to prove~\eqref{bound_JS0}, we distinguish two cases.
 
 \setcounter{case}{0}
 \begin{case}
  Assume first that
  \begin{equation} \label{bound_JS-case1}
   \left(\frac{16 \abs{\sigma}_p}{\alpha_p\,r}\right)^{k-p}
   \geq 2 .
  \end{equation}
  Then, taking the logarithm of both sides of~\eqref{bound_JS-case1},
  we deduce that the right-hand side of~\eqref{bound_JS0}
  is non-positive, so~\eqref{bound_JS0} is satisfied trivially.
 \end{case}
 
 \begin{case}
  Assume now that
  \begin{equation} \label{bound_JS-case2}
   \left(\frac{16 \abs{\sigma}_p}{\alpha_p\,r}\right)^{k-p}
   < 2 .
  \end{equation}
  In this case, we apply Lemma~\ref{lemma:ballconstruction} with the choice
  \begin{equation} \label{tau}
   \tau := \frac{r}{8\abs{\sigma}_p} \! .
  \end{equation}
  This value of~$\tau$ satisfies the condition~\eqref{tau-k-p},
  because of~\eqref{bound_JS-case2}.
  Then, keeping in mind that 
  \[
   \abs{\hc(u, \, \partial B)}_p\leq \frac{\rad B}{s}
   \qquad \textrm{for any } B\in\mathscr{B},
  \]
  we obtain
  \[
   \begin{split}
   \int_{D} \abs{\nabla u}^p \mathrm{d}x
   &\geq
    \sum_{B\in\mathscr{B}} \frac{\rad B}{s} \Lambda_p(s) 
   \geq \sum_{B\in\mathscr{B}}
    \abs{\hc(u, \, \partial B)}_p \Lambda_p(s) 
   \geq \abs{\sigma}_p 
   \Lambda_p\left(\frac{\tau}{2}\right) \! .
   \end{split}
  \]
  Recalling the definition of~$\Lambda_p$,
  Equation~\eqref{Lambda_p}, we obtain
  \begin{equation} \label{bound_JS1}
   \int_{D} \abs{\nabla u}^p \mathrm{d}x
   \geq \frac{\abs{\sigma}_p}{k-p} 
    - \abs{\sigma}_p \, X,
  \end{equation}
  where
  \[
   X := \frac{1}{k-p}\left(1 - 
    \left(\frac{\alpha_p \tau}{2}\right)^{k-p}\right)
   \stackrel{\eqref{tau}}{=}
    \frac{1}{k-p}\left(1 - 
    \left(\frac{\alpha_p\, r}{16 \abs{\sigma}_p}\right)^{k-p}\right) \! .
  \]
  The quantity~$X$ can be estimated from above
  by making use of the inequality
  \begin{equation} \label{LagrangeMVT}
   \frac{1}{y}\left(1 - z^y\right) < \log\left(\frac{1}{z}\right) 
   \qquad \textrm{for any } y\in (0, \, 1), \ z\in (0, \, 1),
  \end{equation}
  which can be proved by applying Lagrange's mean value theorem to the 
  function~$x\in [0, \, y]\mapsto z^x$.
  We obtain
  \begin{equation} \label{bound_JS2}
   X \leq \log\left(\frac{16 \abs{\sigma}_p}{\alpha_p \, r}\right) \! .
  \end{equation}
  Together,~\eqref{bound_JS1} and~\eqref{bound_JS2}
  imply~\eqref{bound_JS0}.
  \qedhere
 \end{case}

\end{proof}
\subsection{Dependence on the parameter~\texorpdfstring{$p$}{p}}

The goal of this section is to study
how~$E_p$ and~$\abs{\, \,}_p$ behave
as a function of~$p$. Given $k-1 \leq p \leq q$, we set 
\begin{equation}\label{f_function}
 f(p, \, q):=\frac{p-k+1}{q-k+1}.
\end{equation}

\begin{prop} \label{prop:pqnorm}
 For any numbers~$p$, $q$ such that~$q > p > k-1$, there exists
 a constant~$\lambda(p, \, q)\in (0, \, 1)$ that satisfies
 \begin{equation} \label{pqnorm}
  \lambda(p, \, q) \, \abs{\sigma}_q^{f(p, \,q)}
  \leq \abs{\sigma}_p
  \leq \lambda(p,\, q)^{-1} \, \abs{\sigma}_q 
 \end{equation}
 for any~$\sigma\in\pi_{k-1}(\NN)$. Moreover, 
 we can choose~$\lambda(p, \, q)$ in such a way that $\lambda(p, \, q) \to 1$ as $p\to q^-$ and
\begin{equation} \label{pqnorm-lambda}
\limsup_{p\to q^-} \frac{1 - \lambda(p, \, q)}{q - p} < +\infty.
 \end{equation}
 In particular, $\abs{\sigma}_p\to\abs{\sigma}_q$
 as~$p\to q$, for any~$\sigma\in\pi_{k-1}(\NN)$.
\end{prop}

Proposition~\ref{prop:pqnorm}, combined with
Proposition~\ref{prop:lowerbound-p},
implies the following result.

\begin{prop} \label{prop:lowerbound-k}
 There exists a number~$p_0 = p_0(k, \, \NN) > k-1$
 such that the following statement holds.
 Let~$D\subseteq\R^k$ be a bounded domain,
 such that~$\overline{D}$ is homeomorphic 
 to~$\overline{B}^k$. Let~$r\in (0, \, 1/2]$, $p\in [p_0, \, k)$,
 let~$u\in W^{1,p}(D, \, \NN)$ be
 a map that is continuous in the set~$\Gamma_r$
 (defined by~\eqref{Gamma_r}),
 and let~$\sigma\in\pi_{k-1}(\NN)$
 be the homotopy class of~$u$ on~$\partial D$.
 Then, we have
 \begin{equation} \label{bound_JS-k}
  \int_{D} \abs{\nabla u}^p \mathrm{d}x
  \geq \frac{\abs{\sigma}_k}{k - p}
   - C\abs{\sigma}_k
   \left(\log\left(\frac{\abs{\sigma}_k}{r}\right) + 1\right)\!,
 \end{equation}
 where~$C$ is a constant that depends only on~$k$, $\NN$.
\end{prop}

Now, we give the proofs of Propositions~\ref{prop:pqnorm}
and~\ref{prop:lowerbound-k}.
Proposition~\ref{prop:pqnorm} ultimately relies upon
an elliptic regularity estimate for minimising $p$-harmonic
maps. We recall that, by the direct method of the Calculus
of Variations, for each homotopy class~$\sigma\in\pi_{k-1}(\NN)$
there exists a map~$v_{p,\sigma}\in W^{1,p}(\SS^{k-1}, \, \NN)\cap\sigma$
that attains the minimum in~\eqref{E_p}, that is
\begin{equation*} 
 E_p(\sigma) = \int_{\SS^{k-1}} \abs{\nabla_\top v_{p,\sigma}}^p \,\d\H^{k-1}.
\end{equation*}
We will say that~$v_{p,\sigma}$ is a minimising
$p$-harmonic map. Since $p>k-1$, $v_{p,\sigma}$ is Hölder continuous and classical elliptic estimates (e. g.~\cite{Uhlenbeck, DiBenedetto, HardtLin-Minimizing}) imply that $v_{p,\sigma}$ belongs to $C^{1,\alpha}(\SS^{k-1}, \, \NN)$ for a certain $\alpha >0$. 

\begin{theorem} \label{th:p-harmonic}
 Let~$p_0> k-1$ be a given parameter.
 If~$q > p\geq p_0$, then every minimising~$p$-harmonic 
 map~$v\in W^{1,p}(\SS^{k-1}, \, \NN)$ is of class~$C^1$
 and satisfies
 \begin{equation} \label{p-harmonic}
  \norm{\nabla_\top v}_{L^\infty(\SS^{k-1})}
  \leq C \norm{\nabla_\top v}_{L^p(\SS^{k-1})}^{1/\alpha} \! ,
 \end{equation}
 where $\alpha:=1-p/(k-1)$ and ~$C \geq 1$ depends on $k$, $\NN$ and~$p_0$,
 but can be chosen uniformly with respect to~$p\geq p_0$ and $\sigma \in \pi_{k-1}(\NN)$.
\end{theorem}
The main point of Theorem~\ref{th:p-harmonic} (proven in Appendix~\ref{appendix}) is that the constant $C_q$ does not depend on $p$. Such a fact does not follow directly from classical regularity estimates, as they involve constants depending also on $p$. However, if $p$ varies away from $k-1$ (as it does in Theorem~\ref{th:p-harmonic}) there is no reason why such constants should blow up. Hence, as suggested in~\cite[Lemma 3.1]{HardtChen}, one can repeat the proofs in~\cite[Section 3]{HardtLin-Minimizing} and show that the constants in the regularity estimates might be taken uniformly on $p$, which we do in Appendix~\ref{appendix}. Let us also point out that a related result was proven Battacharya, DiBenedetto and Manfredi~\cite[Part III]{BattDiBenedettoManfredi}. Such a result does not apply to our setting, as it deals with a standard $p$-Laplace equation on a Euclidean domain, with a growth assumption on the source term which does not apply for harmonic maps.
\begin{corollary} \label{cor:p-harmonic}
 For any~$p$, $q$, $p_0$ with~$q > p \geq p_0> k-1$ 
 and any~$\sigma\in\pi_{k-1}(\NN)$, we have
 \[ 
  C^{(p-q)f(p, \, q)} E_q(\sigma)^{f(p, \, q)}
   \leq E_p(\sigma) \leq \omega^{1 - p/q}_{k-1} E_q(\sigma)^{p/q} ,
 \]
 where~$\omega_{k-1} := \H^{k-1}(\SS^{k-1})$, $C$ and $\alpha$
as in~\eqref{p-harmonic} in Theorem~\ref{th:p-harmonic} and $f(p, \, q)$ is as in \eqref{f_function}. In particular, $E_p(\sigma)\to E_q(\sigma)$ as~$p\to q$,
 for any~$\sigma\in\pi_{k-1}(\NN)$.
\end{corollary}
\begin{proof}
 The upper bound~$E_p(\sigma) 
  \leq \omega^{1 - p/q}_{k-1} E_q(\sigma)^{p/q}$
 is an immediate consequence of the H\"older inequality.
 To prove the opposite inequality, we consider a
 minimising~$p$-harmonic map~$v_{p,\sigma}$ in the homotopy
 class~$\sigma$. By interpolation, we find that
 \begin{equation*}
  \norm{\nabla_\top v_{p,\sigma}}_{L^q(\SS^{k-1})}^q
  \leq \norm{\nabla_\top v_{p,\sigma}}_{L^{\infty}(\SS^{k-1})}^{q-p}
   \norm{\nabla_\top v_{p,\sigma}}_{L^p(\SS^{k-1})}^{p},
 \end{equation*}
so that by \eqref{p-harmonic} in Theorem \ref{th:p-harmonic}
\begin{equation*}
 \norm{\nabla_\top v_{p,\sigma}}_{L^q(\SS^{k-1})}^q
  \leq C^{q-p} \norm{\nabla_\top v_{p,\sigma}}_{L^p(\SS^{k-1})}^{p(1/\alpha(q/p-1)+1)}
\end{equation*}
and then the result follows since
\begin{equation*}
(1/\alpha(q/p-1)+1)^{-1}=\frac{\alpha p}{q-p+\alpha p}= \frac{p-k+1}{q-k+1}.
\qedhere
\end{equation*}
\end{proof}

\begin{proof}[Proof of Proposition~\ref{prop:pqnorm}]
 \setcounter{step}{0}
 Let~$p$, $q$ and~$\sigma\in\pi_{k-1}(\NN)$ be given, with~$q > p > k-1$.
 We can assume that~$\sigma\neq 0$, for otherwise there is nothing to prove.
 
 \begin{step}[Proof of the upper bound in~\eqref{pqnorm}]
  Since the infimum in \eqref{E_p} is attained, there exist finitely many
  homotopy classes~$\sigma_1$, \ldots, $\sigma_h$ 
  in~$\pi_{k-1}(\NN)$ such that
  \begin{equation} \label{pqnorm11}
   \sigma=\sum_{i=1}^h\sigma_i, \qquad
   \abs{\sigma}_q = \sum_{i=1}^h E_q(\sigma_i).
  \end{equation}
  We can certainly suppose that each~$\sigma_i$ is nonzero.
  Then, each~$\sigma_i$ satisfies~$E_q(\sigma_i)\geq \alpha_q$
  (where~$\alpha_q$ is given by~\eqref{alpha_p}) and hence,
  \begin{equation} \label{pqnorm12}
   h \leq \frac{\abs{\sigma}_q}{\alpha_q} .
  \end{equation}
  The definition of~$E_p$ and Corollary~\ref{cor:p-harmonic}
  imply
  \[
   \begin{split}
    \abs{\sigma}_p 
    \leq \sum_{i=1}^h E_p(\sigma_i)
    \leq \omega^{1 - p/q}_{k-1} \sum_{i=1}^h E_q(\sigma_i)^{p/q} \, ,
   \end{split}
  \]
  where~$\omega_{k-1} := \H^{k-1}(\SS^{k-1})$.
  Since~$p < q$, Jensen's inequality for \emph{concave} functions gives
  \[
   \frac{1}{h}\sum_{i=1}^h E_q(\sigma_i)^{p/q}
   \leq \left(\frac{1}{h}\sum_{i=1}^h E_q(\sigma_i)\right)^{p/q}
  \]
  and, hence,
  \begin{equation} \label{pqnorm1}
   \begin{split}
    \abs{\sigma}_p
    \leq \omega^{1 - p/q}_{k-1} h^{1 - p/q} 
     \left(\sum_{i=1}^h E_q(\sigma_i)\right)^{p/q} 
    \leq \omega^{1 - p/q}_{k-1} 
     \left(\frac{\abs{\sigma}_q}{\alpha_q}\right)^{1 - p/q} 
     \abs{\sigma}_q^{p/q}
     = \left(\frac{\omega_{k-1}}{\alpha_q}\right)^{1-p/q} \abs{\sigma}_q \! .
   \end{split}
  \end{equation}
  Here, we have applied~\eqref{pqnorm11} and~\eqref{pqnorm12}.
  This proves the upper bound in~\eqref{pqnorm}.
 \end{step}

 \begin{step}[Proof of the lower bound in~\eqref{pqnorm}]
  We proceed in a similar way as above
  and find~$\tau_1$, \ldots, $\tau_\ell$ in~$\pi_{k-1}(\NN)$
  such that
  \begin{equation*}
   \sigma=\sum_{i=1}^\ell \tau_i, \qquad
   \abs{\sigma}_p = \sum_{i=1}^\ell E_p(\tau_i).
  \end{equation*}
  Let~$p_0 > k-1$ be chosen in such a way that~$p_0\leq p$, $p_0\leq q$. Let $f(p, \, q)$ be as in \eqref{f_function}, notice that $f(p, \, q)\leq 1$. Thus, by definition of~$E_q$ and by Corollary~\ref{cor:p-harmonic},
  we obtain
  \begin{align} \label{pqnorm2}
   \begin{split}
    C^{(p-q)f(p, \, q)}\abs{\sigma}_q^{f(p, \, q)}
    &\leq C^{(p-q)f(p, \, q)} \left( \sum_{i=1}^\ell E_q(\tau_i) \right)^{f(p, \, q)}
    \leq \sum_{i=1}^\ell C^{(p-q)f(p, \, q)}  E_q(\tau_i)^{f(p, \, q)} \\ &\leq \sum_{i=1}^\ell E_p(\tau_i)=\abs{\sigma}_p,
   \end{split}
  \end{align}
  where~$C$ is the same constant as in~\eqref{p-harmonic},
  which depends on $k$, $\NN$ and~$p_0$ only and $f(p,\,q)$ is as in~\eqref{f_function}. Together,~\eqref{pqnorm1} and~\eqref{pqnorm2}
  imply~\eqref{pqnorm}, with
  \[
   \lambda(p, \, q) := \min\!\left(\left(
   \frac{\alpha_q}{\omega_{k-1}}\right)^{1 - p/q}\!, \, C^{(p-q)f(p, \, q)}\right)<1 \!,
  \]
so that $\lambda(p, \, q)$ satisfies the conditions listed in the statement of Proposition~\ref{prop:pqnorm}.
  \qedhere
 \end{step}
\end{proof}

\begin{proof}[Proof of Proposition~\ref{prop:lowerbound-k}]
 Let~$r\in (0, \, 1/2]$, $p\in (k-1, \, k)$,
 let~$u\in W^{1,p}(D, \, \NN)$ be
 a map that is continuous in~$\Gamma_r$,
 and let~$\sigma\in\pi_{k-1}(\NN)$
 be the homotopy class of~$u$ on~$\partial D$.
 We assume that~$\sigma\neq 0$, otherwise there 
 is nothing to prove.
 Combining Proposition~\ref{prop:lowerbound-p}
 with Proposition~\ref{prop:pqnorm} for $q=k$, we obtain
 \begin{equation*}
  \begin{split}
   \int_{D} \abs{\nabla u}^p \mathrm{d}x
   &\geq \frac{\abs{\sigma}_p}{k - p}
    - C\abs{\sigma}_p\log\left(\frac{\abs{\sigma}_p}{\alpha_p \, r}\right) \\
   &\geq \frac{\lambda(p,\,k) \abs{\sigma}_k^{f(p, \, k)}}{k - p}
    - \frac{C\abs{\sigma}_k}{\lambda(p, \, k)}
     \log\left(\frac{\abs{\sigma}_k}{\alpha_p \, r \,
     \lambda(p, \, k)}\right) \\
   &= \frac{\abs{\sigma}_k}{k - p}
    - X_1 - X_2-X_3, 
  \end{split}  
 \end{equation*}
 where
 \begin{align*}
  X_1   &:= \frac{1 - \lambda(p, \, k)}{k - p} \abs{\sigma}_k, \\
  X_2 &:= \lambda(p, \, k) \frac{\abs{\sigma}_k-\abs{\sigma}_k^{f(p, \, k)}}{k-p}\\
  X_3 &:= \frac{C\abs{\sigma}_k}{\lambda(p, \, k)}
     \log\left(\frac{\abs{\sigma}_k}{\alpha_p \, r \,
     \lambda(p, \, k)}\right) \! .
 \end{align*}
 We estimate~$X_1$ and $X_2$
 separately. For the first term, we apply the estimate~\eqref{pqnorm-lambda},
 which implies 
 \begin{equation*}
  X_1 \leq C\abs{\sigma}_k
 \end{equation*}
 for any~$p$ in a left neighbourhood of~$k$
 and some constant~$C$ that does not depend on~$\sigma$, $p$. Next, we estimate $X_2$. If $\abs{\sigma}_k < 1$, then $X_2 <0$. Otherwise, we recall that $\lambda(p, \, k)<1$ and apply Lagrange's mean value theorem, which implies
 \begin{equation*}
 X_2 \leq \sup_{p \leq x \leq k} \frac{\mathrm{d}}{\mathrm{d}x}\left( \abs{\sigma}_k^{f(x, \, k)} \right) \leq C \abs{\sigma}_k \log \abs{\sigma}_k.
 \end{equation*}
 as $f(p, \, k) \leq 1$ and the derivative of $x \to f(x, \, k)$ is bounded in $[p,k]$. Either way, we have
 \begin{equation*}
 X_2 \leq C \abs{\sigma}_k(\log \abs{\sigma}_k + 1),
 \end{equation*}
 for some constant $C$ that depends only on $k$ and $\NN$. Finally, we estimate~$X_3$. Since~$\lambda(p, \, k)\to 1$ as~$p\to k$,
 for~$p$ close enough to~$k$ we have
 \begin{equation*}
  \begin{split}
   X_3 \leq C\abs{\sigma}_k
     \left(\log\left(\frac{\abs{\sigma}_k}{r}\right) 
     - \log\left(\alpha_p \, \lambda(p, \, k)\right)\right) \! .
  \end{split}
 \end{equation*}
 By a compactness argument, one checks that $\inf_{p \in [p_0,k]}\alpha_p$ is positive for all $p_0 \in [k-1,k]$, see also Corollary \ref{cor:p_harmonic_lowerbound} below. Therefore,
 we obtain
 \begin{equation*}
  \begin{split}
   X_3 \leq C\abs{\sigma}_k
     \left(\log\left(\frac{\abs{\sigma}_k}{r}\right) 
     + 1\right)
  \end{split}
 \end{equation*}
 for some possibly different constant~$C$,
 that still depends only on~$k$, $\NN$.
 This completes the proof.
\end{proof}

\section{Proof of the compactness and lower bound statement in Theorem~\ref{MainThm}}\label{Section:lower_bound}
The scope of this section is to provide the proof of~\ref{MainThm1} in Theorem~\ref{MainThm}, which will require the use of the results in section~\ref{sect_lower_p_energy_bounds}. 
\subsection{A local statement}
Let us begin by fixing $\Omega'$ a bounded domain in $\R^{n+k}$ such that $\Omega \csubset \Omega'$.  Take $U, U^\prime$ arbitrary bounded smooth domains in~$\R^{n+k}$ such that $U\csubset U^\prime \subseteq \Omega'$, to be fixed for the rest of this subsection. We also fix $(u_p)_{p \in (k-1, \, k)}$ a family of maps in~$W^{1,p}(U^\prime, \, \NN)$, such that
\begin{equation}\label{H}
\sup_{p \in (k-1, \, k)} (k-p)D_p(u_p, \, U^\prime) < +\infty.
\end{equation}
As in~\cite{ABO2,CO2}, the main step is to prove the following localized version of the statement~\ref{MainThm1} in Theorem~\ref{MainThm}:
\begin{prop}\label{prop:localcompactness}
There exist a sequence $(p_i)_{i\in \N}$ in $(k-1, \, k)$ such that $p_i \to k$ as $i \to \infty$
and $S\in\M_n(\overline{U}^\prime; \GN)$ such that
\begin{gather}
\lim_{i \to \infty}\F_{U,k}(\mathbf{S}(u_{p_i})-S) = 0,
\label{liminf:flat} \\
\M_{U,k}(S) \leq \liminf_{i \to \infty} 
(k-p_i)D_{p_i}(u_{p_i}, \, U^\prime). \label{liminf:mass}
\end{gather}
\end{prop}
In order to prove Proposition~\ref{prop:localcompactness}, we shall adapt to our setting the building blocks approach developed in~\cite{ABO2}. First, let us fix some notations following~\cite{ABO2}. The grid centered at $a\in \R^{n+k}$ with size $h>0$, denoted by $\GG=\GG(a, \, h)$, is defined as the collection of closed cubes
\begin{equation*}
\GG = \GG(a, \, h) := \left\{a + h z + [0, \ h]^{n+k} \colon z\in\Z^{n+k}\right\} \!.
\end{equation*}
For each $0\leq j \leq n+k$, we denote the collection of the (closed) $j$-cells by $\GG_j=\GG_j(a, \, h)$, and define the $j$-skeleton by $R_j=R_j(a, \, h) := \cup_{K\in\GG_j} K$. Given $\GG(a, \, h)$, we define the \textit{dual grid} of $\GG(a, \, h)$ by 
\[
\GG^\prime=\GG^\prime_k(a, \, h):=\GG(a + (h/2, \, h/2, \, \ldots, \, h/2), \, h)
\]
In addition, we also denote by $\GG^\prime_k=\GG^\prime_k(a, \, h)$, $R^\prime_k=R^\prime_k(a, \, h)$ the collection of 
the closed $k$-cells and the $k$-skeleton of $\GG^\prime$ respectively. In the sequel, we assume also that $h$ is taken small enough so that for any center $a$, $j$ between $0$ and $n+k$ and any $j$-cell $K$, whenever $K \cap U$ is nonempty one has that $K \subseteq U'$. This assumption is not restrictive, as $h$ will be sent to zero for fixed $U$ and $U'$. For any $p \in (k-1,k)$, any $j$ between $0$ and $n+k$, any $j$-cell $K$ and any $u \in W^{1,p}(K \cap U', \,\NN)$, we set
\begin{equation*}
D_p(u, \, K \cap U'):= \int_{K \cap U'}\lvert \nabla u(x) \rvert^p \mathrm{d}x,
\end{equation*}
so that
\begin{equation*}
D_p(u, \, R_j)= \int_{R_j \cap U'}\lvert \nabla u(x) \rvert^p \mathrm{d}x.
\end{equation*}
The result below (analogous to~\cite[Lemma 3.11]{ABO2} and~\cite[Lemma 8]{CO2}) essentially shows that given a length one can choose an associated grid so as to satisfy certain suitable properties:
\begin{lemma}\label{lemma:grid}
	Let $h > 0$, $\delta \in (0, \, 1)$, $p \in (k-1, \, k)$, $L$ be an $n$-plane in $\R^{n+k}$ and $u \in W^{1,p}(U',\,\NN)$  be given. There exists $a \in \mathbb{R}^{n+k}$ independent on $\delta$ and $L$ such that the corresponding grid~$\GG=\GG(a, \, h)$ satisfies:
\begin{enumerate}
\item For any $K \in \GG_k$ such that $K \cap U \not= \emptyset$ one has that $u$ belongs to $W^{1,p}(\partial K,\, \NN)$. In particular, $u$ is continuous on $\partial K$ and, as a consequence, its homotopy class on $\partial K$ is well-defined.
\item The partial integral $D_p(u,\,R_j \cap U')$ is well-defined for all $0 \leq j \leq n+k$ and the following inequalities hold:
	\begin{gather}
		h^n \, D_p(u, \, R_{k,L}\cap U^\prime) \leq (1 +  \delta) 
	D_p(u, \, U^\prime) \label{grid_ktilde}, \\
	h^{n+k-j} \, D_p(u, \, R_{j}\cap U^\prime) 
	\lesssim \delta^{-1} D_p(u, \, U^\prime), \mbox{ for all } 0 \leq j \leq n+k \label{grid_k},
	\end{gather}
	where $R_{k,L}=R_{k,L}(a, \, h)$ is the union of the $k$-cells of $\GG(a,\, h)$ such that their dual $n$-cells are parallel to $L$. We also have set $R_j=R_j(a,\,h)$.
  \end{enumerate}
\end{lemma}
\begin{proof}
The proof is the one provided~\cite[Lemma 3.11]{ABO2} with some slight modifications due to the fact that here we require the property 1., which is not an issue in~\cite{ABO2}. Notice first that by Fubini's Theorem, there exists a measurable set $D \subseteq [0, \, h]^{n+k}$ of full measure such that for any grid $\GG(a, \, h)$ with $a\in D$, we have $D_p(u,\,R_j \cap U')$ is well-defined for all $0 \leq j \leq n+k$. If~$L$ is not one of the coordinate $k$-planes, then~$R_{k,L}$ is empty and~\eqref{grid_ktilde} is trivially satisfied. Therefore, without loss of generality, $R_{k,L}$ can be taken equal to $\tilde{R}_k$, the union of $k$-cells of $\GG$ which are parallel to the $k$-plane spanned by $\{\e_{n+1},\ldots,\e_{n+k}\}$. Hence, the quantity $f(a):=h^nD_p(u,\,\tilde{R}_k \cap U')$ does not change when one modifies the last $k$ components of $a \in D$. Notice also that $f$ is well-defined almost everywhere in $[0, \, h]^{n+k}$ as $D$ has full measure on $[0, \, h]^{n+k}$. Therefore, by invoking again Fubini's Theorem, we deduce that the mean of $f$ on $[0, \, h]^{n+k}$ is equal to $D_p(u,\,U')$, that is
\begin{equation}\label{mean_f}
\frac{1}{h^{n+k}}\int_{[0, \, h]^{n+k}} f(a)\mathrm{d}a=D_p(u,\,U').
\end{equation}
For any $j$ an integer between $0$ and $n+k$ define the measurable function
\begin{equation*}
f_j: a \in D \to f_j(a):=h^{n+k-j} \, D_p(u, \, R_{j}\cap U^\prime) \in \mathbb{R},
\end{equation*}
and, again, observe that~$f_j$ is defined almost everywhere in $[0, \, h]^{n+k}$. Notice that $R_j$ is the union of families of planes parallel to those generated by $j$ elements of the canonical basis $\{\e_1,\ldots,\e_{n+k}\}$. For each such family we can argue as above, and since we have $\binom{n+k}{j}$ of these families, it follows that
\begin{equation}\label{mean_fj}
\frac{1}{h^{n+k}}\int_{[0, \, h]^{n+k}}f_j(a)\mathrm{d}a=\binom{n+k}{j}D_p(u,\,U').
\end{equation}
Generalizing the statement of~\cite[Lemma 8.4]{ABO2}, we claim that there exists a set $F \subseteq [0, \, h]^{n+k}$ of positive measure such that for all $a \in F$
\begin{equation}\label{ineq_f}
f(a)\leq (1+\delta) \frac{1}{h^{n+k}}\int_{[0, \, h]^{n+k}} f(a)\mathrm{d}a
\end{equation}
and
\begin{equation}\label{ineq_f_j}
f_j(a) \leq (1+\delta)\frac{n+k}{\delta}\frac{1}{h^{n+k}}\int_{[0, \, h]^{n+k}}f_j(a)\mathrm{d}a, \mbox{ for all } 0 \leq j \leq n+k.
\end{equation}
Indeed, let $E \subseteq [0, \, h]^{n+k}$ be the set of $a \in [0, \, h]^{n+k}$ such that~\eqref{ineq_f} fails, and for each $j$ an integer between $0$ and $n+k$ define $E_j$ as the set of $a \in [0, \, h]^{n+k}$ such that the inequality in~\eqref{ineq_f_j} fails. A computation shows that $\mathrm{meas}(E)<h^{n+k}/(1+\delta)$ and $\mathrm{meas}(E_j) < h^{n+k}\delta/(m(1+\delta))$. Therefore, if we set $\tilde{E}:= E \cup E_0 \cup \ldots \cup E_{n+k}$ we have that $\mathrm{meas}(\tilde{E})<h^{n+k}$. As a consequence, the claim follows by taking $F:=[0, \, h]^{n+k}\setminus \tilde{E}$. 

Notice now that for a. e. $a$ in $[0, \, h]^{n+k}$, there will only be a finite number $N$ of sets in $\GG_k$ (which varies with $a$) which intersect $U$. Let $K_1,\ldots,K_N$ be those sets. Recall now that by Fubini's Theorem, for all $i$ between $1$ and $N$, $u$ belongs to $W^{1,p}(\partial K_i,\,\NN)$ for a.e. $a \in [0, \, h]^{n+k}$. As a consequence, property 1. holds for a. e. $a \in [0, \, h]^{n+k}$ as the number of relevant sets under consideration is finite. In other words, if we call $\tilde{F}$ the set $a \in [0, \, h]^{n+k}$ such that 1. holds, then $\mathrm{meas}(\tilde{F})=h^{n+k}$. As the set $F$ considered above has positive measure, so does $D \cap F \cap \tilde{F}$. Therefore, it suffices to select $a \in D \cap  F \cap \tilde{F}$ and the result follows. Indeed, property 1. holds for such $a$ by the considerations above. Regarding property 2., inequality~\eqref{grid_ktilde} is obtained by combining identity~\eqref{mean_f} with~\eqref{ineq_f} and~\eqref{grid_k} is obtained by combining identity~\eqref{mean_fj} with~\eqref{ineq_f_j}.
\end{proof}
Lemma~\ref{lemma:grid} will be applied to the elements of the family $(u_p)_{p \in (k-1, \, k)}$. Fix $\delta \in (0, \, 1)$ until the end of this section. For each $p \in (k-1, \, k)$, we fix $\GG^p$ the grid of size $h_p$ and with parameter $\delta$ provided by Lemma~\ref{lemma:grid}, where $h_p$ will be chosen later. The idea is to define a polyhedral chain $T_p$ supported on the dual grid of $\GG^p$ which approximates the singular set $\S(u_p)$ in the flat norm as $p\to k$ and then, relying on the results of Section~\ref{sect_lower_p_energy_bounds}, prove that $T_p$ enjoys suitable properties (see Figure~\ref{singularsetanddual}). The polyhedral chain $T_p$ is defined as
\begin{equation}\label{approxiamtsequenceT_p}
T_p := \sum_{K\in\GG^p_{k}, \ K\cap U\neq\emptyset}
\gamma^p(K) \, \llbracket K^\prime \rrbracket
\in \M_n(\overline{U^\prime}; \, \GN),
\end{equation}
where~$K^\prime\in(\GG^p)^\prime_n$ is the dual cell to~$K$, and $\gamma^p(K)$ is the homotopy class of $u_p$ on $\partial K$. Notice also that the sum in~\eqref{approxiamtsequenceT_p} is finite due to the fact that $U$ is bounded. In order to prove that $T_p$ approximates $\S(u_p)$, the approaches of~\cite{ABO2} and~\cite{CO2} rely on a use of deformation theorems for currents and flat chains respectively in the spirit of Federer and Fleming~\cite{FedererFleming} and White~\cite{White-Deformation}. However, while in the setting of~\cite{ABO2} and~\cite{CO2} the singular sets have finite mass, this is not necessarily true in the present setting: arbitrary maps in $W^{1,p}(U',\,\NN)$ might have singular sets with infinite mass. This represents an obstacle when it comes to applying the deformation theorem to our setting. In order to circumvent this issue, we shall use an argument by Stern~\cite{DanielS2} (with roots on the work by Hang and Lin~\cite{HangLin}) which consists on finding a map $\bar{u}_p$,  that satisfies $\S(\bar{u}_p)=T_p$ along with suitable bounds on~$D_p(\bar{u}_p)$. Subsequently, using property~\eqref{basic_estimate_flat_norm} we can bound $\F_{U,k}(\S(u_{p}) - \S(\bar{u}_p))$ in terms of the $p$-Dirichlet energies $D_p(u_{p}), D_p(\bar{u}_p)$ and the length size $h_p$, which allows to conclude in Proposition~\ref{prop:flat_norm_convergence} due to the properties of~$\bar{u}_p$, summarized in Lemma~\ref{lemma:compositionofradialmaps2} below. 
\begin{figure}
	\centering
	\includegraphics[width=0.4\textwidth]{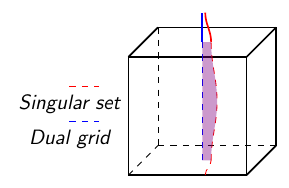}
	\caption{\label{singularsetanddual} An illustration of a singular set and a dual grid within a cube.}
\end{figure}
Let us introduce further notation, which will be used for defining the maps $\bar{u}_p$. Given a $j$-dimensional cell $I_j$ of $\GG^p_j=\GG^p_j(a, \, h_p)$, for some $a \in \mathbb{R}^{n+k}$, $h>0$ and $0\leq j \leq n+k$, and supposing without loss of generality that $I_j=[-\ h_p, \ h_p]^{j}$, we define the projection map corresponding to the cell $I_j$, $\Phi_{j}\,: I_j\setminus\{0\} \to \partial I_j$ as
\begin{equation*}
\Phi_{j}(x)=h_p\displaystyle x/\vert x \vert_{\infty},
\end{equation*}
for any $x\in I_j$. As~$\Phi_j$ coincides with the identity on~$\partial I_j$, we can paste together all these maps to define a continuous map~$R_j^p\setminus (R^{p}_{n + k - j})^\prime\to R_{j-1}^p$, which we also denote as~$\Phi_{j}$. Here~$(R^{p}_{n+k-j})^\prime$ is the $(n+k-j)$-skeleton of the dual grid~$(\GG^{p})^\prime$. Let 
\[
 R_j^p(U) := \bigcup_{K\in\GG_{j}^p, \, K\cap U\neq\emptyset} K.
\]
We have~$U\subseteq R_j^p(U)$ by construction, and~$R_j^p(U)\subseteq U^\prime$ if~$h$ is small enough.
For any $p \in (k-1, \, k)$ and $j$ an integer in $[k,n+k]$, consider the map
\begin{equation}\label{upj}
\bar{u}_{p,j}: x \in R_j^p(U)\setminus R_{n+k-j}^{p\prime} \to \bar{u}_{p,j}(x):=u_p(\phi_j(x)) \in \NN, \mbox{ where } \phi_j:=\Phi_j \circ \ldots \circ \Phi_{n+k}.
\end{equation}
The map $\bar{u}_p$ will be taken equal to $\bar{u}_{p,n+k}$. The next result is a modification of~\cite[Lemma 2.3]{DanielS2}:
\begin{lemma}\label{lemma:compositionofradialmaps} 
For any $p \in (k-1, \, k)$ and $j$ an integer such that $k \leq j\leq n+k$, one has
\begin{equation}\label{inequalityenergy1}
D_p(u_p \circ \Phi_j,\, R_j^p \cap U) \lesssim  \frac{h_p}{j-p} D_p(u_p,\,R^p_{j-1}\cap U')
\end{equation}
and
\begin{equation}\label{inequalityenergy2}
\int_{R^p_j \cap U}\vert u_p- u_p \circ \Phi_{j}\vert^p \mathrm{d}x \lesssim h^p_p\left( \frac{h_p}{j-p} D_p(u_p,\,R_{j-1}^p \cap U')+D_p(u_p,\,R_j^p \cap U')\right).
\end{equation}
\begin{proof}
Notice first that since the projection maps $\Phi_j$ are Lipschitz and $U \subseteq U'$, the integrals appearing in~\eqref{inequalityenergy1} and~\eqref{inequalityenergy2} are well defined due to the selection of a suitable grid made in Lemma~\ref{lemma:grid}. Consider a $j$-dimensional cube with the size $h_p$ which intersects $U$  with $k \leq j\leq n+k$. Recall then that $K$ is contained in $U'$. Without loss of generality, we can assume that $I_j=[-h_p, \, h_p]^j$. Moreover, up to a Lipschitz change of coordinates we can identify $I_j$ with $B^j_{h_p}$. Then, we shall prove the inequality $\eqref{inequalityenergy1}$ for $B^j_{h_p}$, which writes as
\begin{equation}\label{inequalityenergy5} 
D_p(u_p \circ \Phi_j,\, B^j_{h_p} \cap U) \lesssim  \frac{h_p}{j-p} D_p(u_p,\,\partial B^j_{h_p})
\end{equation}
and~\eqref{inequalityenergy2} writes as
\begin{equation}\label{inequalityenergy6}
\int_{B^j_{h_p} \cap U}\vert u_p- u_p \circ \Phi_j\vert^p \mathrm{d}x \lesssim h^p_p\left( \frac{h_p}{j-p} D_p(u_p,\,\partial B^j_{h_p})+D_p(u_p,\,B^j_{h_p})\right),
\end{equation}
where $\bar{\Phi}(x):=h_p\frac{x}{\abs{x}}$ for all $x\in B^j_{h_p}$. One has,
\begin{equation*}
\vert \nabla u_p \circ\Phi_j(x) \vert^p = \frac{h^p_p}{\lvert x \rvert^p}\left\vert \nabla u_p \left(h_p\frac{x}{\abs{x}}\right) \right\vert^p
\end{equation*}
for any $x\in B^j_{h_p}$. Therefore, by applying Fubini's Theorem
\begin{align*}
D_p(u_p \circ \Phi_j, \, B^j_{h_p}) &=  \int_{B^j_{h_p}} \frac{h^p_p}{\lvert x \rvert^p} \left\vert \nabla u_p \left(h_p\frac{x}{\abs{x}}\right) \right\vert^p \mathrm{d}x=  \int_{0}^{h_p} \int_{\partial B^j_r} \frac{h^p_p}{s^p} \left\vert \nabla u_p \left(h_p\frac{x}{s}\right) \right\vert^p \mathrm{d}x \mathrm{d}s\\
&=\int_{0}^{h_p} \int_{\partial B^j_{h_p}}\frac{h^p_p}{s^p} \vert \nabla u_p (x) \vert^p \left(\frac{s}{h_p}\right)^{j-1}\mathrm{d}x \mathrm{d}s=\frac{1}{j-p}h_pD_p(u_p,\partial B^j_{h_p}),
\end{align*}
which establishes~\eqref{inequalityenergy5}. In order to prove~\eqref{inequalityenergy6}, observe first that $u_p$ and $u_p \circ \Phi_j$ agree on $\partial B^j_{h_p}$. Therefore, we can apply the Poincaré inequality to $u_p-u_p \circ \Phi_j$, which yields after a scaling procedure
\begin{equation*}
\int_{B^j_{h_p}}\vert u_p- u_p \circ \Phi_j\vert^p \mathrm{d}x \lesssim h^p_p\int_{B^j_{h_p}}\vert \nabla \left(u_p-u_p \circ \Phi_j   \right)\vert^p \mathrm{d}x.
\end{equation*}
Inequality~\eqref{inequalityenergy6} then follows by~\eqref{inequalityenergy5}.
\end{proof}
\end{lemma}
From Lemma~\ref{lemma:compositionofradialmaps} and Lemma~\ref{lemma:grid}, we deduce the following property, analogous to~\cite[Lemma 2.5]{DanielS2}:
\begin{lemma}\label{lemma:compositionofradialmaps2} 
	Let $p \in (k-1, \, k)$, and $j$ an integer in $[k, \, n+k]$ and $\bar{u}_{p,j}$ be as in~\eqref{upj}. Then, the following inequalities hold:
	\begin{gather}
	D_p(\bar{u}_{p,j},\,U)   \lesssim  \frac{\delta^{-1}}{j-p}D_p(u_p,\,U') , \label{inequalityenergy3}\\
	\int_{U}\vert u_p- \bar{u}_p \vert^p \mathrm{d}x \lesssim  h^p_p \frac{\delta^{-1}}{j-p} 	D_p(u_p,\,U'). \label{inequalityenergy4}
	\end{gather}
Moreover, if $j=k$ then we have $\S(\bar{u}_{p,k})=T_p$ and $T_p$ is a relative boundary in~$U$.
\begin{proof}
By applying~\eqref{inequalityenergy1} in Lemma~\ref{lemma:compositionofradialmaps}, we obtain
\begin{equation*}
D_p(\bar{u}_{p,j},\,U) \lesssim h^{n+k-j}D_p(u,R_{j}^p\cap U'),
\end{equation*}
so that~\eqref{inequalityenergy3} follows by applying inequality~\eqref{grid_k} in Lemma~\ref{lemma:grid}. Inequality~\eqref{inequalityenergy4} is obtained in a similar manner. Finally, we check that for all $K \in \GG_k^p$ such that $K \cap U$ is non-empty, $u$ and $\bar{u}_{p,k}$ coincide on $\partial K$. Moreover, by definition, the map $\bar{u}_{p,k}$ is continuous away from $(R_k^p)' \cap U'$, which is a polyhedral $n$-complex. Hence, by the results in~\cite[Subsection 3.4]{CO1} we obtain that the singular set of $\bar{u}_{p,k}$ is $T_p$ as defined in~\eqref{approxiamtsequenceT_p}. At this point, the fact that $T_p$ is a relative boundary in~$U$ follows from~\ref{S_rel_boundary} in Proposition~\ref{prop:operator_S} as the domain of $\bar{u}_{p,k}$ is $R_k^p(U)$, which contains $U$.
\end{proof}
\end{lemma}
\begin{prop}\label{prop:flat_norm_convergence}
Assume that $(h_p)_{p \in (k-1, \, k)}$ the family of grid sizes is chosen such that $h_p=(k-p)^3$ for all $p \in (k-1, \, k)$. Then it follows that
\begin{equation}\label{flat_norm_convergence}
\F_{U,k}(\S(u_{p}) - T_p) \to 0
\end{equation}
along any subsequence $(p_i)_{i \in \mathbb{N}}$ in $(k-1, \, k)$ converging to $k$ as $i \to \infty$.
\end{prop}
\begin{proof}
As in Lemma~\ref{lemma:compositionofradialmaps2}, for any $p \in (k-1, \, k)$ let $\bar{u}_p:=\bar{u}_{p,n+k}$ as in~\eqref{upj}. According to~\eqref{basic_estimate_flat_norm} we have
\begin{equation*}
\F_{U,k}(\S(u_{p}) - \S(\bar{u}_{p}))\lesssim \int_{U} \left( \abs{\nabla u_p}^{k-1} + \abs{\nabla \bar{u}_{p}}^{k-1}\right)\lvert u_p-\bar{u}_{p} \rvert \mathrm{d}x ,
\end{equation*}
so that, by Hölder's inequality
\begin{equation}\label{flat_norm_convergence_ineq1}
\F_{U,k}(\S(u_{p}) - \S(\bar{u}_{p}))\lesssim  \left( \norm{\nabla u_p}^{k-1}_{L^p(U)}+\norm{\nabla \bar{u}_p}^{k-1}_{L^p(U)} \right)\norm{u_p-\bar{u}_p}_{L^{\frac{p}{p-k+1}}(U)}.
\end{equation}
Notice that $\frac{p}{p-k+1} \geq p$ and therefore~\eqref{flat_norm_convergence_ineq1} becomes
\begin{equation*}
\F_{U,k}(\S(u_{p}) - \S(\bar{u}_{p}))\lesssim \left( \norm{\nabla u_p}^{k-1}_{L^p(U)}+\norm{\nabla \bar{u}_p}^{k-1}_{L^p(U)} \right) \left(\int_U \lvert u_p-\bar{u}_p \rvert^p \mathrm{d}x \right)^{\frac{p-k+1}{p}},
\end{equation*}
where we have used that the family $(u_p)_{p \in (k-1, \, k)}$ is uniformly bounded on $L^\infty(U)$ due to the manifold constraint. Applying now Lemma~\ref{lemma:compositionofradialmaps2} and assumption~\eqref{H} on $(u_p)_{p \in (k-1, \, k)}$, it follows
\begin{equation*}
\F_{U,k}(\S(u_{p}) - \S(\bar{u}_{p}))\lesssim \left( \left( \frac{\delta^{-1}}{(k-p)} \right)^{\frac{k-1}{p}}+\left( \frac{\delta^{-1}}{(k-p)^2} \right)^{\frac{k-1}{p}}  \right)\left( h_p^p \frac{\delta^{-1}}{(k-p)^2}  \right)^{\frac{p-k+1}{p}},
\end{equation*}
so that
\begin{equation}\label{flat_norm_convergence_ineq2}
\F_{U,k}(\S(u_{p}) - \S(\bar{u}_{p}))\lesssim  \frac{\delta^{-1}h_p^{p-k+1}}{(k-p)^2}.
\end{equation}
At this point, one readily checks that if $h_p=(k-p)^3$ then the right-hand side in~\eqref{flat_norm_convergence_ineq2} tends to $0$ as $p \to k$. This proves the convergence property~\eqref{flat_norm_convergence}, as $\S(\bar{u}_p)=T_p$ by Lemma~\ref{lemma:compositionofradialmaps2}.
\end{proof}
Once Proposition~\ref{prop:flat_norm_convergence} has been established, the goal is to obtain uniform bounds on the mass of $T_p$ on $U$. As a first step, we prove result in the spirit of~\cite[Lemma 3.10]{ABO2} and~\cite[Lemma 11]{CO2}, which follows from the lower bounds for maps in cubes of dimension $k$ (Proposition~\ref{prop:lowerbound-k}).  
\begin{lemma}\label{lemma:lowerbounds}
Let $p_0>k-1$ be the constant given by Proposition~\ref{prop:lowerbound-k}. Let $r$ be a positive constant, $Q^k_h := [0, \, h]^k$ be a cube of edge length~$h>0$ and $p \in [p_0, \, k)$. Let~$u\in W^{1,p}(Q^k_h, \, \NN)$ have continuous trace on $\partial Q^k_h$ and let~$\gamma\in\GN$ be the homotopy class of~$u$ on~$\partial Q^k_h$.
Then,
	\[
\frac{\abs{\gamma}_k}{k - p}
- C\abs{\gamma}_k \left( \log\left(\frac{\abs{\gamma}_k}{r}\right)+1 \right)	\leq h^{p-k}D_p(u, \, Q_h^k) + \, h \, r D_p(u, \, \partial Q_{h}^k) 
	\]
where $C$ is a universal constant.
\end{lemma}
\begin{proof}
The proof follows the lines of~\cite[Lemma 3.10]{ABO2}. Assume first that $h=1$ and define
\begin{equation*}
\tilde{u}(x):= \begin{cases}
u\left( \frac{x}{\lvert x \rvert_{\infty}} \right) &\mbox{ if } x \in Q_{1+r}^k \setminus Q_1^k,\\
u(x) &\mbox{ if } x \in Q_1^k.
\end{cases}
\end{equation*}
Notice that $\tilde{u} \in W^{1,p}(Q_h^k, \, \NN)$ and $\tilde{u}$ is continuous on $Q_{1+r}^k \setminus Q_1^k$ because the trace of $u$ on $Q_1^k$ is continuous. Applying Proposition~\ref{prop:lowerbound-k} to $\tilde{u}$ on $Q_{1+r}$, we obtain
\begin{equation}\label{ineq_1_lowerbounds_cube}
\int_{Q_{1+r}^k \setminus Q_1^k}\frac{1}{\lvert x \rvert_{\infty}^p}\left\lvert \nabla u\left( \frac{x}{\lvert x \rvert_{\infty}} \right)\right\rvert^p \mathrm{d}x+D_p(u,Q_1^k) \geq \frac{\abs{\gamma}_k}{k - p}
- C\abs{\gamma}_k \left( \log\left(\frac{\abs{\gamma}_k}{r}\right)+1 \right).
\end{equation}
Subsequently, notice that
\begin{align*}
\int_{Q_{1+r}^k \setminus Q_1^k}\frac{1}{\lvert x \rvert_{\infty}^p}\left\lvert \nabla u\left( \frac{x}{\lvert x \rvert_{\infty}} \right)\right\rvert^p \mathrm{d}x &=\int_{1}^{1+r}\int_{\partial Q_{s}^k} \frac{1}{s^p}\lvert\nabla u ( s^{-1}x) \rvert^p\mathrm{d}x \mathrm{d}s=\int_1^{1+r}s^{k-1-p}\mathrm{d}sD_p(u,\partial Q_1^k)\\
&\leq rD_p(u,\,\partial Q_1^k)
\end{align*}
since $p>k-1$. Plugging this last inequality on~\eqref{ineq_1_lowerbounds_cube} we obtain the result for $h=1$. The inequality for arbitrary $h$ follows by scaling.
\end{proof}
From Lemma~\ref{lemma:lowerbounds} together with Lemma~\ref{lemma:grid} and Proposition~\ref{prop:pqnorm}, we get a result which is analogous to~\cite[Lemma 12]{CO2}:
\begin{lemma}\label{lemma:mass_T}
Let $r$ be a positive constant and $\delta \in (0, \, 1)$. Assume that $(h_p)_{p \in (k-1, \, k)}$ is the family of grid sizes chosen in Proposition~\ref{prop:flat_norm_convergence}. There exists $p_1 \in (k-1, \, k)$ depending on $\delta$ and $r$ such that for all $p \in (p_1, \, k)$ it holds:
	\begin{equation} \label{mass_T}
	\begin{split}
	\left(1 - c_{r,p}(p)\right)\M_k(T_p\mres U) 
	&\lesssim \delta^{-1}\left(1 + r\right) 
	(k-p)^{-3(k-p)+1} D_p(u_p, \, U^\prime),
	\end{split}
	\end{equation}
	where~$c_{r,\delta}(p)>0$ is such that 
	$c_{r,\delta}(p)\to 0$ as~$p\to k$. Moreover, if~$L$ is a $n$-plane, then there holds
	\begin{equation} \label{mass_Tproj}      
	\begin{split}
	\left(1 - c_{r,\delta}(p)\right)\M_k(\pi_{L,*}(T_p\mres U))
	&\leq \left(1 + \delta + C r \, \delta^{-1}\right)(k-p)^{-3(k-p)}
	(k-p)D_p(u_p, \, U^\prime).
	\end{split}
	\end{equation}
\end{lemma}
\begin{proof}
By the definition of $T_p$ in~\eqref{approxiamtsequenceT_p} we get
\begin{equation} \label{mass-1}
\M_k(T_p\mres U) \leq h_p^n 
\sum_{K\in\GG^p_k, \ K\cap U\neq\emptyset}
|\gamma^p(K)|_k.
\end{equation}
Fix $K=a+hz_K+[0, \, h]^{k} \in \GG_k^p$ (here $z_K \in \Z^k$) such that $K \cap U \not = \emptyset$, and recall that~$K \subseteq U'$. Define the function 
\begin{equation*}
w_{p,k}: x \in \partial((0, \, 1)^{k}) \to \tilde{u}_{p,k}(x):=u_p(a+hz_K+hx) \in \NN,
\end{equation*}
which is well-defined and belongs to $W^{1,p}(\partial((0,\,1)^k), \, \NN)$ by 1. in Lemma~\ref{lemma:grid}. By~\eqref{pqnorm} in Proposition~\ref{prop:pqnorm} along with the definition of $\lvert \cdot \rvert_p$ in~\eqref{groupnorm}, we get 
\begin{equation*}
|\gamma^p(K)|_k \lesssim |\gamma^p(K)|_p^{\frac{k}{p}} \lesssim D_p(w_{p,k},\,\partial ((0, \, 1)^{k-1}))^{\frac{k}{p}}.
\end{equation*}
After a change of variables, we deduce
\begin{equation*}
|\gamma^p(K)|_k  \lesssim (h_p^{p-k+1}D_p(u_p,\,\partial K))^{\frac{k}{p}}.
\end{equation*}
Using now~\eqref{grid_k} in Lemma~\ref{lemma:grid} with $j=k-1$ and the assumption~\eqref{H}, the inequality above becomes
\begin{equation*}
|\gamma^p(K)|_k  \lesssim (h_p^{p-k-n}\delta^{-1}D_p(u_p,\,U'))^{\frac{k}{p}}\lesssim  (h_p^{p-k-n}\delta^{-1}(k-p)^{-1})^\frac{k}{p}
\end{equation*}
Since~$h_p=(k-p)^3$, one finally gets
\begin{equation}\label{ap_ineq}
|\gamma^p(K)|_k  \lesssim \delta^{-\frac{k}{p}}h_p^{a(p)},
\end{equation}
where $a(p)=\frac{k}{p}(3(p-k-n)-1)$. Notice now that, by a change of variables, Lemma~\ref{lemma:lowerbounds} holds when one replaces $[0, \, h_p]^{k}$ by $K$. Applying it to $u_p$ (which we can by 1. in Lemma~\ref{lemma:grid}) for $p \geq p_0$ and then using~\eqref{ap_ineq}, we get 
\begin{equation*}
\left( 1-C(k-p)\left( \log\left(\frac{\delta^{-\frac{k}{p}}(k-p)^{a(p)} }{r}\right)+1 \right)\right)\abs{\gamma^p(K)}_k
 	\leq
(k-p)(h_p^{p-k}D_p(u_p,\, K) + h_p r D_p(u_p,\, \partial K)),
\end{equation*}
for some universal constant $C>0$. After some simplifications, we obtain
\begin{equation}\label{crdelta_ineq}
(1- c_{r,\delta}(p))\abs{\gamma^p(K)}_k \leq (k-p)^{-3(k-p)}(k-p)(D_p(u_p, \,K) + h_p r D_p(u_p, \,\partial K)),
\end{equation}
where
\begin{equation*}
c_{r,\delta}(p)=C(k-p) \left(a(p)\log(k-p)+\log \frac{\delta^{-\frac{k}{p}}}{r} +1\right).
\end{equation*}
Since $\sup_{p \in (k-1, \, k)}a(p) <0$, $\lim_{p \to k}c_{r,\delta}(p)=0$ and there exists $\hat{p} \in (k-1, \, k)$ depending on $\delta$ and $r$ such that $1/2>c_{r,\delta}(p)>0$ for all $p \in [\hat{p}, \, k)$. At this point, set $p_1:= \max\{p_0,\hat{p}\} \in (k-1, \, k)$. Going back to~\eqref{crdelta_ineq} and using~\eqref{mass-1}, we obtain
\begin{equation*}
(1- c_{r,\delta}(p))\M_k(T_p\mres U) \leq (k-p)^{-3(k-p)+1}(h_p^n D_p(u, \,R_k^p \cap U') + h_p^{n+1}rD_p(u_p,\,R_{k-1}^p \cap U')),
\end{equation*}
so that~\eqref{mass_T} follows by~\eqref{grid_k} in Lemma~\ref{lemma:grid}. In order to get~\eqref{mass_Tproj}, notice that
\begin{equation*}
\M_k(\pi_{L,*}(T_p\mres U)) \leq h_p^n 
\sum_{K\in\tilde{\GG}^p_k, \ K\cap U\neq\emptyset}
|\gamma^p(K)|_k,
\end{equation*}
where $\tilde{\GG}_k^p$ is the collection of $k$-cells such that their dual $n$-cells are parallel to $L$. Therefore, if we add in~\eqref{crdelta_ineq} for the $k$-cells $\tilde{\GG}_k^p$ and recall the definition of $T_p$ in~\eqref{approxiamtsequenceT_p}, we get
\begin{equation*}
\M_k(\pi_{L,*}(T_p\mres U)) \leq (k-p)^{-3(k-p)+1}(h_p^n D_p(u,\,\tilde{R}_k^p \cap U') + h_p^{n+1}rD_p(u_p,\,R_{k-1}^p \cap U'))
\end{equation*}
where $\tilde{R}_k^p=\cup_{K \in \tilde{\GG}_k^p}K$ is as in Lemma~\ref{lemma:grid}. At this point,~\eqref{mass_Tproj} follows by~\eqref{grid_ktilde} and~\eqref{grid_k} in Lemma~\ref{lemma:grid}.
\end{proof}
We can now proceed with the end of the proof of Proposition~\ref{prop:localcompactness} by repeating the arguments in~\cite{ABO2}. Before that, let us recall a localization property for the mass of a flat chain:
\begin{lemma}\label{lemma:co2planes}
Let $S \in \M_n(\R^{n+k};\pi_{k-1}(\NN))$. Then,
\begin{equation*}
\M_k(S)=\sup_{(U_i,L_i)_{i \in \N}}\sum_{i=0}^{+\infty}\M(\pi_{L_i,*}(S \mres U_i)),
\end{equation*}
where the supremum is taken among sequences $(U_i,L_i)_{i \in \mathbb{N}}$, where $(U_i)_{i \in \mathbb{N}}$ is a sequence of pairwise disjoint sets in $\R^{n+k}$ and $L_i$ is an $n$-plane for all $i \in \mathbb{N}$.
\end{lemma}
For a proof of Lemma~\ref{lemma:co2planes}, see e.g.~Lemma 14 in~\cite{CO2}.
\begin{proof}[Proof of Proposition~\ref{prop:localcompactness} completed]
Fix a positive $r$ and $\delta \in (0, \, 1)$. Let $p_1 \in (k-1, \, k)$ be as in Lemma~\ref{lemma:mass_T}. By the assumption~\eqref{H} and~\eqref{mass_T} in Lemma~\ref{lemma:mass_T} one gets that $\sup_{p \in (p_1, \, k)}\M(T_p \mres U)<+\infty$. Moreover, Lemma~\ref{lemma:compositionofradialmaps2} implies that $\partial T_p \mres U=0$ for all $p \in (p_1, \, k)$. By compactness (see  e.~g.~\cite[Lemma 6]{CO1} for a compactness statement that is tailored to flat chains relative to an open set), there exists $S \in \M_n(\overline{U},\,\pi_{k-1}(\NN))$ such that $\F_U(T_{p_i}-S) \to 0$ along a subsequence $(p_i)_{n \in \mathbb{N}}$ in $(k-1, \, k)$ such that $p_i \to k$ as $i \to \infty$. Therefore, by Proposition~\ref{prop:flat_norm_convergence}
\begin{equation*}
\lim_{i\to \infty}\F_{U,k}(\S(u_{p_i})-S)=0.
\end{equation*} 
Applying now~\eqref{mass_Tproj} and using the lower semicontinuity of the mass with respect to the flat norm convergence, we get
\begin{equation*}
\M_k(\pi_{L,*}(S)) \leq \left(1+\delta+Cr\delta^{-1} \right)\liminf_{i \to \infty}(k-p_i)D_{p_i}(u_{p_i},\,U')
\end{equation*}
where $L$ is any $n$-plane. We have used that $\lim_{p \to k}(k-p)^{-3(k-p)}=1$. Sending now first $r$ to $0$ and then $\delta$ to $0$, we obtain
\begin{equation*}
\M_k(\pi_{L,*}(S)) \leq \liminf_{i \to \infty}(k-p_i)D_{p_i}(u_{p_i},\,U').
\end{equation*}
As $U$ and $L$ are arbitrary, a localization argument along with
Lemma~\ref{lemma:co2planes} gives
\begin{equation*}
\M_{k,U}(S) \leq  \liminf_{i \to \infty}(k-p_i)D_{p_i}(u_{p_i},\,U'),
\end{equation*}
which completes the proof of Proposition~\ref{prop:localcompactness}.
\end{proof}
\subsection{Proof of the statement~\ref{MainThm1} in Theorem~\ref{MainThm} completed}
We now use the local statement provided by Proposition~\ref{prop:localcompactness}, in combination with an averaging argument, in order to deduce~\ref{MainThm1} in Theorem~\ref{MainThm}.

\begin{lemma}\label{lemma:goodlemma}
Let $\eta \in (0, \, 1]$ and $V$ be an open bounded set such that $\partial \Omega \subseteq \partial V$. Let $u \in W^{1,k}_v(V,\,\mathbb{R}^m) \cap L^\infty(V,\,\mathbb{R}^m)$. There exists $\nu  \in (0,\, 1]$ such that, for any open set $E \subseteq V$ satisfying~$\partial \Omega \subseteq \partial E$ and $\mathcal{L}^{n+k}(E) \leq \nu$ and any sequence~$(p_i)_{i\in\N}\subseteq (k-1, \, k)$, there exists a positive measure set $A_{E}$ contained in $B^m_{r_{\RR}}$ such that for any $y \in A_{E}$ and~$i\in\N$ the map $w_y:= \RR_{y} \circ u$ belongs to $W^{1,p_i}(V,\NN)$ and satisfies
\begin{equation}\label{extensionniceproperty_boundV}
\liminf_{i\in\N} \ (k-p_i)\int_V \abs{\nabla w_y}^{p_i} \mathrm{d}x \leq C
\end{equation}
and
\begin{equation}\label{extensionniceproperty}
\liminf_{i\in\N} \ (k-p_i)\int_{E}\abs{\nabla w_y}^{p_i} \mathrm{d}x \leq C \eta,
\end{equation}
where $C>0$ is a constant depending on $\lVert u\rVert_{L^\infty(V)}, v, k$ and $\NN$.
\end{lemma}
\begin{proof}
Along the proof will denote by $C$ a positive constant depending on $\lVert u\rVert_{L^\infty(V)}, v, k$ and $\NN$ which might change from line to line. There exists $\nu >0$ such that for any open set $E \subseteq V$ satisfying~$\mathcal{L}^{n+k}(E) \leq \nu$ we have
\begin{equation}\label{w_delta_eps}
\rVert \nabla u \rVert_{L^{k}(E)} \leq \eta.
\end{equation}
Let us take a sequence~$(p_i)_{i\in\N} \subset (k-1, \, k)$. By combining~\eqref{average_projection_inequality} in Lemma~\ref{lemma:average_projection} with~\eqref{w_delta_eps} and Fatou lemma we obtain
\begin{equation}\label{averageextension1_def}
 \begin{split}
  \int_{B^m_{r_\RR}} \liminf_{i\to\infty} (k-p_i)\lVert \nabla w_y \rVert^{p_i}_{L^{p_i}(E)} \, \mathrm{d}y
  \leq \liminf_{i\to\infty} (k-p_i)\int_{B^m_{r_\RR}}\lVert \nabla w_y \rVert^{p_i}_{L^{p_i}(E)} \, \mathrm{d}y \leq  C\eta
 \end{split}
\end{equation}
and
\begin{equation}\label{averageextensionU_def}
\int_{B^m_{r_\RR}} \liminf_{i\to\infty} (k-p_i)\lVert \nabla w_y \rVert^{p_i}_{L^{p_i}(V)} \, \mathrm{d}y
  \leq \liminf_{i\to\infty} (k-p_i)\int_{B^m_{r_\RR}}\lVert \nabla w_y \rVert^{p_i}_{L^{p_i}(V)} \, \mathrm{d}y \leq  C.
\end{equation}
We claim that there exists a set~$A_{E}$ of positive measure such that for all $y$ in such set one has
\begin{equation}\label{ineq_averageextension1}
 \liminf_{i\to\infty} (k-p_i) \lVert \nabla w_y \rVert^{p_i}_{L^{p_i}(E)} \leq \frac{2}{\mathcal{L}^m(B^m_{r_\RR})}\int_{B^m_{r_\RR}} \liminf_{i\to\infty} (k-p_i) \lVert \nabla w_y \rVert^{p_i}_{L^{p_i}(E)} \, \mathrm{d}y
\end{equation}
and
\begin{equation}\label{ineq_averageextensionU}
\liminf_{i\to\infty} (k-p_i) \lVert \nabla w_y \rVert^{p_i}_{L^{p_i}(V)} \leq \frac{2}{\mathcal{L}^m(B^m_{r_\RR})}\int_{B^m_{r_\RR}} \liminf_{i\to\infty} (k-p_i) \lVert \nabla w_y \rVert^{p_i}_{L^{p_i}(V)} \, \mathrm{d}y
\end{equation}
Indeed, let $A_1 \subseteq B^m_{r_\RR}$ be the set where~\eqref{ineq_averageextension1} fails and define $A_2$ in the same way for~\eqref{ineq_averageextensionU}. That is, $A_{E}=B^m_{r_{\RR}} \setminus (A_1 \cup A_2)$. From~\eqref{ineq_averageextension1} and~\eqref{ineq_averageextensionU} one infers by a direct computation that 
\begin{equation*}
\mathcal{L}^m(A_1) < \frac{1}{2}\mathcal{L}^m(B^m_{r_{\RR}}) \quad \mbox{ and } \quad \mathcal{L}^m(A_2) < \frac{1}{2}\mathcal{L}^m(B^m_{r_{\RR}})
\end{equation*}
which means that $A_{E}$ has positive measure, as we wanted to show. At this point, inequality~\eqref{extensionniceproperty_boundV} follows by combining~\eqref{averageextension1_def} with~\eqref{ineq_averageextension1} and~\eqref{extensionniceproperty} is obtained as a combination of~\eqref{averageextensionU_def} and~\eqref{ineq_averageextensionU}.
\end{proof} 
\begin{remark} \label{rk:badremark}
For any~$y\in A_E$, the map~$w_y = \RR_y\circ u$ given by Lemma~\ref{lemma:goodlemma} satisfies~$\S(w_y) = \S_y(u)$. In case~$u$ is smooth, this claim follows from~\cite[Lemma~16]{CO1}. The general case follows from a density argument, using the continuity of~$\S$ given by~\eqref{basic_estimate_flat_norm}.
\end{remark}

At this point, one gets the result by extending $u_p$ outside $\Omega$ using 
Lemma~\ref{lemma:goodlemma} and then applying Proposition~\ref{prop:localcompactness} with $\Omega=U$. 
\begin{proof}[\textbf{Proof of the statement (i) in Theorem~\ref{MainThm} completed}]
Let~$(u_p)_{p\in (k-1,  k)}$ be a family of maps such that~$u_p\in W^{1,p}(\Omega, \, \NN)$ and~$\sup_{p\in (k-1, \, k)} (k - p) D_p(u_p, \, \Omega) < +\infty$. We extract a countable subsequence~$p_i\to k$ in such a way that
\[
 \lim_{i\to\infty} (k-p_i) D_{p_i}(u_{p_i}, \, \Omega)
 = \liminf_{p\to k} (k-p) D_p(u_{p}, \, \Omega)
 < +\infty.
\]
Then, we will be able to extract further subsequences, without changing the value of the inferior limit.
Fix a function $u \in W^{1,k}(\mathbb{R}^{n+k},\,\mathbb{R}^m) \cap L^\infty(\mathbb{R}^{n+k},\,\mathbb{R}^m)$ such that $u=v$ in the sense of traces on $\partial \Omega$. Let $\eta \in (0, \, 1]$, consider $\nu>0$ be the quantity given by Lemma~\ref{lemma:goodlemma}, applied with $V=\Omega' \setminus \Omega$. For any $m \in \N^*$, define 
\begin{equation*}
\Omega'_m :=\{x \in \Omega'\setminus \Omega \colon \dist(x,\, \partial \Omega)< 1/m\}.
\end{equation*}
There exists $m_0 \in \N^*$ be such that $\mathcal{L}^{n+k}(\Omega'_{m_0}) \leq \nu$. We fix an arbitrary value~$y\in A_{\Omega_{m_0}}$, where~$A_{\Omega_{m_0}}\subset B^m_{r_{\RR}}$ is given by Lemma~\ref{lemma:goodlemma}, and let $w_{m_0} := \RR_{y}\circ u \colon\Omega_{m_0}'\to\NN$.
Set $\Omega_{m}:=\Omega \cup \Omega'_{m}$ for~$m\geq m_0$. Consider now the map  $u_{p_i, m_0} \in W^{1,p_i}(\Omega_{m_0}, \,\NN)$ obtained by setting $u_{p_i, m_0}=u_{p_i}$ in $\Omega$ and $u_{p_i,m_0}=w_{m_0}$ in $\Omega_{m_0} \setminus \Omega$. Notice that, at least along a (non-relabelled) subsequence, $(u_{p_i,m_0})_{i\in\N}$ satisfies the boundedness assumption~\eqref{H} with $U'=\Omega_{m_0}$ due to~\eqref{extensionniceproperty} in Lemma~\ref{lemma:goodlemma}. Therefore, we can apply Proposition~\ref{prop:localcompactness} to $(u_{p_i,m_0})_{i\in\N}$ for $U= \Omega_{2m_0}$ and $U'=\Omega_{m_0}$. This implies the existence of $S'\in\M_n(\overline{\Omega_{2m_0}}; \, \GN)$ and a further subsequence such that 
\begin{equation*}
\lim_{i \to \infty}\F_{\Omega_{2m_0},k}(\S(u_{p_i,m_0}) - S') = 0.
\end{equation*}
Next, we claim that the chain~$S := S^\prime\mres\overline{\Omega}$ belongs to~$\mathscr{C}(\Omega, \, v)$. To this end, we consider the auxiliary map~$\zeta_{p_i,m_0}$ defined by~$\zeta_{p_i,m_0} := u_{p_i}$ in~$\Omega$ and~$\zeta_{p_i,m_0} := u$ in~$\Omega_{2m_0}\setminus\Omega$. By Remark~\ref{rk:badremark} we have
\[
 \S(u_{p_i,m_0})\mres\left(\Omega_{2m_0}\setminus\Omega\right)
 = \S(w_{m_0})\mres\left(\Omega_{2m_0}\setminus\Omega\right) 
 = \S_{y}(u)\mres\left(\Omega_{2m_0}\setminus\Omega\right) 
 = \S_{y}(\zeta_{p_i,m_0})\mres\left(\Omega_{2m_0}\setminus\Omega\right) 
\]
for any~$y$ in the (non-null) set~$A_{\Omega_{m_0}'}$  given by Lemma~\ref{lemma:goodlemma}. Since~$u_{p_i, m_0}$ coincide with~$\zeta_{p_i, m_0}$ in~$\Omega$, we have
\[
 \S_y(\zeta_{p_i, m_0}) = \S(u_{p_i,m_0}) \qquad \textrm{for any } y\in A_{\Omega_{m_0}^\prime}.
\]
We can now apply Lemma~\ref{lemma:CUv} to the maps~$\zeta_{p_i, m_0}$ and conclude that~$S = S^\prime\mres\overline{\Omega}\in\mathscr{C}(\Omega, \, v)$, as claimed.

Regarding the lower bounds, fix $A \subseteq \R^{n+k}$ an open set and apply Proposition~\ref{prop:localcompactness} to $(u_{p_i,m_0})_{i\in\N}$ for $U=\Omega \cap A$ and $U'=\Omega_{2m_0} \cap A$. By uniqueness of the limit, we obtain that
\begin{align*}
\M_k(S\mres A) &\leq \liminf_{i \to \infty} 
(k-p_i)D_{p_i}(u_{p_i,m_0}, \,\Omega_{2m_0} \cap A)\\
&=\liminf_{i \to \infty} \left(
(k-p_i)D_{p_i}(u_{p_i},\,  \Omega\cap A)+(k-p_i)D_{p_i}(w_{m_0}, \,\Omega_{m_0} \cap A)\right),
\end{align*}
so that, by Lemma~\ref{lemma:goodlemma}
\begin{equation*}
\M_k(S \mres A) \leq \liminf_{i \to \infty}(k-p_i)D_{p_i}(u_{p_i}, \, \Omega\cap A)+C\eta,
\end{equation*}
for some constant $C>0$. Since $\eta$ is an arbitrary number in $(0, \, 1]$, the proof is completed.
\end{proof}

\section{Proof of the upper bound statement in Theorem~\ref{MainThm}}\label{Section:upper_bound}

The goal of this Section is to establish the upper bound statement in Theorem~\ref{MainThm}. The proof of this statement follows by adapting the arguments of~\cite[Section~3]{CO2}, to which we refer the reader for further details.

\subsection{Notations and preliminary results}
Let us now introduce some of the notations and results from
\cite{ABO1,ABO2}. Let~$K\subseteq \R^{n+k}$ be a $n$-dimensional
simplex, and let $K^\perp$ be 
the $k$-plane orthogonal to~$K$ passing through the origin. For any $\delta$, $\gamma$ two positive parameters, the  diamond neighborhood of $K$ corresponding to $\delta$ and $\gamma$ is defined as
\begin{equation}\label{U-diamond}
U(K, \, \delta, \, \gamma) := 
\left\{x^\prime + x^{\prime\prime}\colon
x^\prime\in K, \ x^{\prime\prime}\in K^\perp, \ 
|x^{\prime\prime}|\leq \min\left(\delta, \,
\gamma\dist(x^\prime, \, \partial K)\right)\right\}
\end{equation} 
Let $M$ and $A$ be polyhedral sets in $\mathbb{R}^{n+k}$ of dimensions $n$ and $n-1$ respectively.
\begin{definition}[\cite{ABO1, ABO2}]\label{def:nice_singularity}
	We say that~$u$ has a \emph{nice singularity at~$M$} if~$u$ is locally Lipschitz
	on~$\overline{\Omega}\setminus M$ and
	\[
	\abs{\nabla u(x)} \lesssim \dist^{-1}(x, \, M)
	\qquad \textrm{for a.e. } x\in\Omega\setminus M.
	\]
	We say that~$u$ has a \emph{nice singularity at~$(M, \, A)$} if~$u$ is locally
	Lipschitz on~$\overline{\Omega}\setminus(M\cup A)$ and, for any~$r>1$,
	there is a constant~$C_r$ depending on $r$, $n$, $k$, $\Omega$ and $\NN$ only such that
	\[
	\abs{\nabla u(x)} \leq C_r\left(\dist^{-1}(x, \, M)
	+ \dist^{-r}(x, \, A)\right)
	\quad \textrm{for a.e. } x\in\Omega\setminus (M\cup A).
	\]
	We say that~$u$ has a \emph{locally nice singularity at~$M$}
	(respectively, at $(M, \, A)$)
	if, for any open subset~$W$ whose closure~$\overline{W}$ 
	is contained in~$\Omega$, the restriction~$u_{|W}$
	has a nice singularity at~$M$ (respectively, at~$(M, \, A)$).
\end{definition}
\begin{definition}[\cite{ABO2, CO2}] \label{def:minimal}
	Let~$u\colon\Omega\to\NN$ be a map with nice 
	singularity at~$(M, \, A)$, and let~$\eta>0$.
	We say that~$u$ is \emph{$\eta$-minimal} 
	if there exist positive numbers~$\delta$, $\gamma$, 
	a triangulation of~$M$ and, for any~$n$-simplex~$K$ of the triangulation, 
	a Lipschitz map~$\phi_K\colon\SS^{k-1}\to\NN$
	that satisfy the following properties.
	\begin{enumerate}[label=(\roman*)]	
		\item If~$K\subseteq M$, $K^\prime\subseteq M$ 
		are $n$-simplices with~$K\neq K^\prime$,
		then $U(K, \, \delta, \, \gamma)$ and~$U(K^\prime, \, \delta, \, \gamma)$
		have disjoint interiors.
		\item For any $n$-dimensional simplex $K\subseteq M$
		and a.e.~$x = (x^\prime, \, x^{\prime\prime})\in U(K, \, \delta, \, \gamma)$,
		we have $u(x) = \phi_K(x^{\prime\prime}/|x^{\prime\prime}|)$.
		
		\item For any $n$-dimensional simplex $K\subseteq M$
		and any map~$\zeta\in W^{1,k}(\SS^{k-1}, \, \NN)$
		that is homotopic to~$\phi_K$, we have
		\begin{equation*} 
		\int_{\SS^{k-1}} \abs{\nablaT\phi_K}^k \d\H^{k-1}
		\leq \int_{\SS^{k-1}} \abs{\nablaT\zeta}^k \d\H^{k-1} + \eta.
		\end{equation*}
	\end{enumerate}
\end{definition}
We say chain $S$ in $\M_n(\overline{\Omega}; \, \pi_{k-1}(\NN))$ is \emph{locally polyhedral} if for any $K \subseteq \Omega$ compact there exists $T$ a polyhedral chain such that $(S-T) \mres K =0$. The following was proved in~\cite{CO2}. 
\begin{lemma}[{\cite[Lemma 3]{CO2}}]
\label{lemma:3_co2}
There exists $u_* \in W^{1,k}_v(\Omega,\,\mathbb{R}^m) \cap L^\infty_v(\Omega,\,\mathbb{R}^m)$ bounded which satisfies the following properties for a. e. $y \in \R^m$:
\begin{enumerate}
\item $\mathbb{M}(\mathbf{S}_y(u_*))<+\infty$ and $\mathbf{S}_y(u_*) \mres \partial \Omega =0$.
\item The chain $\mathbf{S}_y(u_*)$ is locally polyhedral.
\item The chain $\mathbf{S}_y(u_*)$ takes its multiplicities in a finite subset of $\pi_{k-1}(\NN)$, which depends only on $k$ and $\NN$.
\item There exists a locally $(n-1)$-polyhedral set $P_y$ such that $\RR \circ (u_*-y)$ has a locally nice singularity at $\spt \mathbf{S}_y(u_*) \cup P_y$.
\end{enumerate}
\end{lemma}
For the rest of this Section, the map $u_*$ will be fixed and $C$ will always denote a (possibly different) constant depending only on $\lVert u_* \rVert_{L^\infty(\Omega)}, v, k$ and $\NN$.

\subsection{Reduction to a dense class of chains}
Following \cite{CO2}, in this Subsection we introduce a $\F$-dense set of chains in which chains can be decomposed as in \eqref{decomposition_S} below. The main idea is to define a map $w_*$ as $w_*:= \RR_{y^*}\circ u_* \in W^{1,k-1}_v(\Omega,\,\NN)$ for some  $y_* \in B^m_{r_{\RR}}$ so that the properties in Lemma~\ref{lemma:3_co2} are satisfied (recall the notations in Subsection~\ref{Subsection:retraction}). Here the choice of $y_*$ will be made according to Lemma~\ref{lemma:Fatou_ub} below, a result analogous to~\cite[Proposition 6.4 (iv)]{ABO2}. Let us first recall that by~\cite[Lemma 2]{CO2} we have that
\begin{equation}\label{property_projection_ub}
\mbox{for a. e. } y \in B^m_{r_{\RR}} \mbox{ the map } w_y:= \RR_y \circ u_* \mbox{ belongs to } W^{1,k-1}_v(\Omega,\NN) \mbox{ and } \mathbf{S}(w_y)=\mathbf{S}_y(u_*).
\end{equation}
We denote by $B$ the set of $y$ in $B^m_{r_{\RR}}$ such that the properties in Lemma~\ref{lemma:3_co2} and~\eqref{property_projection_ub} hold. The set~$B$ has full measure in $B^m_{r_\RR}$. We also fix $\eta \in (0, \, 1]$ until the end of the proof, as well as $\nu$ a small enough positive quantity so that $\lVert \nabla u_* \rVert_{L^k(E_\nu)} \leq \eta$, where $E_\nu:= \{ x \in \Omega: \mathrm{dist}(x,\partial \Omega)< \nu\}$.
\begin{lemma}\label{lemma:Fatou_ub}
Let $(p_i)_{i \in \N}$ be a sequence in $(k-1, \, k)$ such that $p_i \to k$ as $i \to \infty$. There exists $y_* \in B$ such that the map $w_*:= w_{y_*}$ satisfies
\begin{equation}\label{liminf_limsup}
\liminf_{i \to \infty}(k-p_i)D_{p_i}(w_*,\,E_\nu) \leq C\eta.
\end{equation}
\end{lemma}
\begin{proof}
By inequality~\eqref{average_projection_inequality} in Lemma~\ref{lemma:average_projection} we obtain for all $i \in \N$
\begin{equation*}
\int_{B^m_{r_{\RR}}}(k-p_i)D_{p_i}(w_y,\,E_\nu)\mathrm{d}y\leq C\eta,
\end{equation*} 
so that, by Fatou's Lemma
\begin{equation*}
\int_{B^m_{r_{\RR}}}\liminf_{i \to \infty}(k-p_i)D_{p_i}(w_y,\,E_\nu)\mathrm{d}y \leq \liminf_{i \to \infty}\int_{B^m_{r_{\RR}}}(k-p_i)D_{p_i}(w_y,\,E_\nu)\mathrm{d}y\leq C\eta
\end{equation*}
and the result follows since $B$ has full measure in $B^m_{r_{\RR}}$.
\end{proof}
We fix an arbitrary sequence $(p_i)_{i \in \N}$ as in Lemma~\ref{lemma:Fatou_ub} as well as the corresponding map $w_*$. We now restrict our attention to chains $S$ in $S\in\mathscr{C}(\Omega, \, v)$ which can be written as
\begin{equation}\label{decomposition_S}
S=\mathbf{S}(w_*)+\partial R
\end{equation}
where $R$ is a polyhedral $(n+1)$-chain, compactly supported on $\Omega$. This is possible because any chain in $\mathscr{C}(\Omega, \, v)$ can be approximated in the~$\F$-norm by chains of the form~\eqref{decomposition_S}, in such a way that the mass is preserved in the limit. See \cite[Proposition 5]{CO2} for a proof.

\subsection{Insertion of dipoles}
If a chain $S$ can be written as in \eqref{decomposition_S}, then the map $w_*$ can be modified in a way such that its singular set coincides with $S$. This was proven in \cite{CO2} by employing the so-called \emph{insertion of dipoles} procedure (see the references \cite{Bethuel1990,BethuelBrezisCoron,BrezisCoronLieb,GiaquintaModicaSoucek-I,PakzadRiviere} for earlier variants). The purpose of this section is to recall the main steps of such a procedure as carried out in \cite{CO2}. We begin by introducing some notation, slightly modifying from~\cite{CO2}. Consider the set
\begin{equation*}
\mathfrak{S}:= \inf\{ \sigma \in \pi_{k-1}: \lvert \sigma \rvert_k=E_k(\sigma) \}.
\end{equation*}
By~\cite[Proposition 1]{CO2}, the set $\mathfrak{S}$ is finite and for any $\sigma \in \pi_{k-1}(\NN)$ there exists a decomposition $\sigma= \sum_{l=1}^q \sigma_l$ such that $\lvert \sigma \rvert_k = \sum_{l=1}^q\lvert \sigma_l \rvert_k$ and $\sigma_l \in \mathfrak{S}$ for all $l \in \{1,\ldots,q\}$. In the sequel, $W_{\mathfrak{S}}$ will denote an open set such that:
\begin{itemize}
\item $W_{\mathfrak{S}} \csubset \Omega$.
\item $W_{\mathfrak{S}}$ has polyhedral boundary.
\item $\spt R \subseteq \overline{W_{\mathfrak{S}}}$.
\item $S \mres  W_{\mathfrak{S}}$ takes its multiplicities in $\mathfrak{S}$.
As before, this property does not hold in general, but we can assume that it is satisfied with no loss of generality, by an approsmation argument~\cite[Proposition~6]{CO2}.
\item There exist triangulations of $\partial W_{\mathfrak{S}}$ and $\spt S$ such that any simplex of the triangulation of $\partial W_{\mathfrak{S}}$ is transverse to any simplex of the triangulation of $\spt S$. We will simply write that $\partial W_{\mathfrak{S}}$ is transverse to $\spt S$.
\end{itemize}
The last condition implies that $\spt S \cap \partial W_{\mathfrak{S}}$ has dimension $(n-1)$ at most, which means that $S \mres \partial W_{\mathfrak{S}}=0$. Let now $W$ be an open set such that
\begin{itemize}
\item $W_{\mathfrak{S}} \csubset W \csubset \Omega$.
\item $\Omega \setminus W \subseteq E_\nu$.
\item $\mathbb{M}(S \mres (\overline{W} \setminus W_{\mathfrak{S}})) \leq \eta$.  This is possible because $S \mres \partial \Omega = \mathbf{S}(w_*) \mres \partial \Omega= \mathbf{S}_{y_*}(u_*)=0$ where the latter equality is due to 1. in Lemma~\ref{lemma:3_co2}.
\item $W$ has polyhedral boundary.
\item There exist triangulations of $\partial W$ and $\spt S$ such that any simplex of the triangulation of $\partial W$ is transverse to any simplex of the triangulation of $\spt S$. In the sequel we simply write that $\partial W$ is transverse to $\spt S$.
\end{itemize}
The following is then proven in \cite{CO2}.
\begin{lemma}[{\cite[Lemma 5]{CO2}}]\label{lemma:5_co2}
Let $w_*$, $W_{\mathfrak{S}}$ and $W$ be as above. There exists $w$ in $W^{1,k-1}_v(\Omega,\,\NN)$ such that
\begin{enumerate}
\item $w=w_*$  in $\Omega \setminus W$.
\item $w$ has a locally nice singularity at $(\spt S, Q_{*})$, where $Q_{*}$ is a locally $(n-1)$-polyhedral set containing $(\spt S)_{n-1}$. 
\item $\mathbf{S}(w)=S$.
\item $w|_{W}$ is $\eta$-minimal in the sense of Definition~\ref{def:minimal}.
\end{enumerate}
\end{lemma}

\subsection{Local upper bounds}
We now give an upper bound on the $\limsup$ of $(k-p_i)D_{p_i}(w_*,\,W)$, following the lines of~\cite[Lemma 7 Step 3]{CO2}.
\begin{lemma}\label{lemma:limsup_local}
Under the previous notations, it holds
\begin{equation}\label{limsup_local}
\limsup_{i \to \infty}(k-p_i)D_{p_i}(w_{*},\,W) \leq (1+C\eta)\mathbb{M}_k(S)+C\eta.
\end{equation}
\end{lemma}
\begin{proof}
We argue as in~\cite[Lemma 7, Step 3]{CO2}. By Lemma~\ref{lemma:5_co2}, $w_{*}|_{W}$ is $\eta$-minimal and with nice singularity at $((\spt S) \cap W,\,Q_* \cap W)$. Therefore, there exists a triangulation of $(\spt S) \cap W$ such that for any $K$ a $n$-simplex of the triangulation there exists a Lipschitz map $\phi_{K}: \mathbb{S}^{k-1} \to \NN$ satisfying the properties listed in Definition~\ref{def:minimal} with $u=w$, $M=(\spt S) \cap W$ and $A=Q_{*} \cap W$. Up to decreasing $\delta$ and $\gamma$, we can assume that the interior of  $U(K,\delta,\gamma)$ is included in $W$ for any $K$ a $n$-simplex of the triangulation. Let us fix such a $K$ and write $U:= U(K,\delta,\gamma)$ for simplicity. For a. e. $x=(x',x'')$ in $U$ we have $w(x)=\phi_K(x''/\lvert x'' \rvert)$ and, as a consequence
\begin{equation*}
\lvert \nabla w(x) \rvert \leq \frac{1}{\lvert x'' \rvert} \left\lvert \nablaT \phi_K\left( \frac{x''}{\lvert x'' \rvert} \right) \right\rvert
\end{equation*}
for all such $x$. Fix now $i \in \mathbb{N}$. By applying Fubini's Theorem, we get
\begin{align*}
\int_{U}\lvert \nabla w(x) \rvert^{p_i} \mathrm{d}x &\leq \int_0^{\delta} \int_{\partial B^k_s} \frac{1}{s^{p_i}} \left\lvert \nablaT \phi_K\left( \frac{x''}{s} \right) \right\rvert^{p_i} \mathrm{d}x'' \mathrm{d}s \mathscr{H}^n(K) \\ &= \int_0^\delta \int_{\mathbb{S}^{k-1}}\lvert \nablaT \phi_K(x'') \rvert^{p_i} s^{k-1-p_i} \mathrm{d}x''\mathrm{d}s
\mathscr{H}^n(K)\\
&= \frac{\delta^{k-p_i}}{k-p_i} \int_{\mathbb{S}^{k-1}}\lvert \nablaT \phi_K(x'') \rvert^{p_i} \mathrm{d}x''
\mathscr{H}^n(K)
\end{align*}
so that, by using Hölder's inequality we obtain
\begin{equation}\label{limsup_local_ineq1}
(k-p_i)D_{p_i}(w,U)\leq \delta^{k-p_i}\omega_k^{1-p_i/k} \left(  \int_{\mathbb{S}^{k-1}}\lvert \nablaT \phi_K(x'') \rvert^k \mathrm{d}x'' \right)^{\frac{p_i}{k}}\mathscr{H}^n(K).
\end{equation}
Using now (iii) in Definition~\ref{def:minimal} we obtain from~\eqref{limsup_local_ineq1}
\begin{equation}\label{limsup_local_ineq2}
(k-p_i)D_{p_i}(w,U) \leq  \delta^{k-p_i}\omega_k^{1-p_i/k} \left( E_k(\sigma_K)+\eta \right)^{\frac{p_i}{k}}\mathscr{H}^n(K),
\end{equation}
where $\sigma_K \in \pi_{k-1}(\NN)$ is the homotopy class of $\phi_K$. Since $\partial W_{\mathfrak{S}}$ is transverse to $\spt S$, up to refining the triangulation the interior of each simplex $K$ is either contained in $W_{\mathfrak{S}}$ or in $\overline{W} \setminus W_{\mathfrak{S}}$. Therefore, we obtain from~\eqref{limsup_local_ineq2} that
\begin{align}\label{limsup_local_ineq3}
(k-p_i)D_{p_i}(w,D) \leq \delta^{k-p_i}\omega_k^{1-p_i/k}\left( \Sigma_{\mathrm{int}}^i+\Sigma_{\mathrm{out}}^i \right) 
\end{align}
where
\begin{equation*}
\Sigma_{\mathrm{int}}^i:=\sum_{\mathrm{int}(K) \subseteq W_{\mathfrak{S}}} \left(E_k(\sigma_K)+\eta\right)^{\frac{p_i}{k}}\mathscr{H}^n(K)
\end{equation*}
\begin{equation*}
\Sigma_{\mathrm{out}}^i:=\sum_{\mathrm{int}(K) \subseteq \overline{W} \setminus W_{\mathfrak{S}}} \left(E_k(\sigma_K)+\eta \right)^{\frac{p_i}{k}}\mathscr{H}^n(K)
\end{equation*}
and $D:=\cup_K U(K,\delta,\rho)$, the union being taken among all the $n$-simplices of the triangulation. Notice that since $\lvert \cdot \rvert_k$ induces the discrete topology on $\pi_{k-1}(\NN)$ we have that $\mathscr{H}^{n}(K) \leq C \mathbb{M}(S \mres K)$ for any $K$ a $n$-simplex of the triangulation. Therefore, using that $p_i/k \in (0, \, 1)$ along with Jensen's inequality for concave functions, we deduce
\begin{align}\label{sigma_int}
\Sigma_{\mathrm{int}}^i &\leq \left(\sum_{\mathrm{int}(K) \subseteq W_{\mathfrak{S}}} \left(E_k(\sigma_K)+\eta \right) \mathscr{H}^n(K)  \right)^{\frac{p_i}{k}}\left(\sum_{\mathrm{int}(K) \subseteq W_{\mathfrak{S}}}\mathscr{H}^n(K)\right)^{1-p_i/k}, \nonumber \\
&\leq C^{1-p_i/k} \left(\sum_{\mathrm{int}(K) \subseteq W_{\mathfrak{S}}} E_k(\sigma_K) \mathscr{H}^n(K) +C\eta \M_k(S \mres K) \right)^{\frac{p_i}{k}}\left(\mathbb{M}_k(S)\right))^{1-p_i/k}
\end{align}
and
\begin{align}\label{sigma_out}
\Sigma_{\mathrm{out}}^i &\leq \left( \sum_{\mathrm{int}(K) \subseteq \overline{W} \setminus W_{\mathfrak{S}}} (E_k(\sigma_K)+\eta)\mathscr{H}^n(K) \right)^{\frac{p_i}{k}}\left( \sum_{\mathrm{int}(K) \subseteq \overline{W} \setminus W_{\mathfrak{S}}}\mathscr{H}^n(K) \right)^{1-p_i/k}\nonumber \\
&\leq C\left( \sum_{\mathrm{int}(K) \subseteq \overline{W} \setminus W_{\mathfrak{S}}} (E_k(\sigma_K)+\eta)\mathbb{M}_k(S \mres K) \right)^{\frac{p_i}{k}}\left(\mathbb{M}_k(S \mres (\overline{W} \setminus W_{\mathfrak{S}}) )\right)^{1-p_i/k}
\end{align}
Let now $K$ be a $n$-simplex of the triangulation such that $\mathrm{int}(K) \subseteq W_{\mathfrak{S}}$. Since $S \mres W_{\mathfrak{S}}$ takes its multiplicities in $\mathfrak{S}$ we have that $E_k(\sigma_K)=\lvert \sigma_K \rvert_k$ and $\lvert \sigma_K \rvert_k \mathscr{H}^n(K) \leq \mathbb{M}_k(S \mres K)$. Therefore~\eqref{sigma_int} becomes
\begin{equation}\label{sigma_int_def}
\Sigma_{\mathrm{int}}^i \leq C^{1-p_i/k}(1+C\eta)^{p_i/k} \mathbb{M}_k(S).
\end{equation}
Assume now that $K$ is a $n$-simplex of the triangulation such that $\mathrm{int}(K)\subseteq \overline{\Omega} \setminus W_{\mathfrak{S}}$. Recall that $S \mres K = \mathbf{S}(w_{*}) \mres K= \mathbf{S}_{y_{*}}(u_*) \mres K$ because $\spt R \subseteq \overline{W_{\mathfrak{S}}}$. It then follows by 3. in Lemma~\ref{lemma:3_co2} that $E_k(\sigma_k) \leq C$. Therefore,~\eqref{sigma_out} becomes
\begin{equation}\label{sigma_out_def}
\Sigma_{\mathrm{out}}^i \leq C \mathbb{M}_k(S \mres (\overline{W} \setminus W_{\mathfrak{S}}) ).
\end{equation}
By plugging~\eqref{sigma_int_def} and~\eqref{sigma_out_def} into~\eqref{limsup_local_ineq3} one finally obtains
\begin{equation}\label{limsup_local_ineq_Ci}
(k-p_i)D_{p_i}(w,D) \leq \delta^{k-p_i} \omega_k^{1-p_i/k} \left( C^{1-p_i/k}(1+C\eta)^{\frac{p_i}{k}}\mathbb{M}_k(S)+C \mathbb{M}_k(S \mres (\overline{W} \setminus W_{\mathfrak{S}}) ) \right).
\end{equation}
Consider now the set $F:= \overline{W} \setminus D$. Recall that $w$ has nice singularity at $((\spt S) \cap W, Q_* \cap W)$. Since $Q_*$ contains $(\spt S)_{n-1}$ we have that
\begin{equation}\label{limsup_local_ineq4}
\mathrm{dist}(x,Q_*)  \leq C(\delta,\gamma) \mathrm{dist}(x,\spt S) 
\end{equation}
for all $x \in F$, where $C(\delta,\gamma)>0$ depends on $\delta$ and $\gamma$. Using now the definition of nice singularity with $r=1+1/(2k)$ we obtain from~\eqref{limsup_local_ineq4}
\begin{equation*}
\lvert \nabla w(x) \rvert \leq C(w,W\delta,\gamma) \mathrm{dist}^{-r}(x,Q_*),
\end{equation*}
for a. e. $x \in F$ and where $C(w,W\delta,\gamma)$ depends on $w, W,\delta$ and $\gamma$. Hence, we get
\begin{equation}\label{limsup_local_ineq_F_i}
(k-p_i)D_{p_i}(w,F) \leq (k-p_i)C(w,W\delta,\gamma) \int_{F} \mathrm{dist}^{-rp_i}(x,Q_*)\mathrm{d}x \leq (k-p_i)C(w,W\delta,\gamma).
\end{equation}
One then obtains~\eqref{limsup_local} by combining~\eqref{limsup_local_ineq_Ci} and~\eqref{limsup_local_ineq_F_i} along with the fact that $W_\mathfrak{S}$ and $W$ are chosen so that $\mathbb{M}(S \mres (W \setminus W_{\mathfrak{S}})) \leq \eta$. 
\end{proof}
\subsection{Proof of the statement~\ref{MainThm2} in Theorem~\ref{MainThm} completed}
We now complete the proof of~\ref{MainThm2} in Theorem~\ref{MainThm} by combining Lemmas~\ref{lemma:3_co2} and~\ref{lemma:5_co2} along with a diagonal argument.
\begin{proof}[$\textbf{Proof of the statement~\ref{MainThm2} in Theorem~\ref{MainThm} completed}$]
We keep the previous notations. Assume that the chain $S$ can be decomposed as in \eqref{decomposition_S}. We claim that, up to passing to a subsequence for $(p_i)_{i \in \N}$ one finds that
\begin{equation}\label{claim_ub_final}
\limsup_{i \to \infty}(k-p_i)D_{p_i}(w_*,\Omega) \leq (1+C\eta)\mathbb{M}_k(S)+C\eta.
\end{equation}
If~\eqref{claim_ub_final} is established, then the proof is complete, by \cite[Proposition 5]{CO2} and a diagonal argument. In order to prove~\eqref{claim_ub_final}, as in~\cite[Page 44]{ABO2} notice that Lemma~\ref{lemma:Fatou_ub} implies that one can pass to a subsequence for $(p_i)_{i \in \N}$ so that
\begin{equation}\label{lim_final_out}
\lim_{i \to \infty}(k-p_i)D_{p_i}(w_*,E_\nu) \leq C\eta
\end{equation}
and, since we assume that $\Omega \setminus W \subseteq E_\nu$, the claim~\eqref{claim_ub_final} follows by combining~\eqref{lim_final_out} along with~\eqref{limsup_local} in Lemma~\ref{lemma:limsup_local}.
\end{proof}
\section{A \texorpdfstring{$\Gamma$}{Gamma}-convergence statement without boundary conditions}\label{Section:GammaNB}

We present here a variant of Theorem~\ref{MainThm} for the problem with no boundary conditions. Although minimizers of~\eqref{pDirichletFunctional_intro} without any boundary condition are just constant maps, Proposition~\ref{prop:Gamma-nobd} below might be still applied to non-trivial minimisation problems that include lower-order terms or under integral constraints, as long as these are compatible with the topology of $\Gamma$-convergence. We will say that a chain~$S$ is a finite-mass, $n$-dimensional relative boundary if it can be written as~$S = (\partial R)\mres\Omega$, where~$R\in\M_{n+1}(\R^{n+k}; \, \pi_{k-1}(\NN))$ is such that~$\M(\partial R)<+\infty$.

\begin{prop} \label{prop:Gamma-nobd}
 Assume that~\eqref{hp:N} holds. Then, the following properties hold.
 \begin{enumerate}[label=(\roman*)]
 \item\emph{Compactness and lower bound.}
 Let $(u_p)_{p \in (k-1, \, k)}$  be a family such that $u_p \in W^{1,p}(\Omega,\, \NN)$ for all $p$ and
 \begin{equation*}
  \sup_{p \in (k-1, \, k)}(k-p)D_p(u_p) <+\infty.
 \end{equation*}
 Then, there exists a (non relabelled) countable
 sequence $p\to k$ and a finite-mass, $n$-dimensional relative 
 boundary~$S$ such that~$\F_{\Omega,k}(\S(u_p) - S) \to 0$ as $p \to k$ and, for any open subset~$A\subseteq\R^{n+k}$,
 \[
  \M_k(S\mres A) \leq \liminf_{p\to k}
  (k-p)D_p(u_p, \, A\cap\Omega).
 \]
 
 \item \emph{Upper bound.} For any finite-mass, $n$-dimensional
  relative boundary~$S$, there exists a sequence $(u_p)_{p \in (k-1, \, k)}$ such that $u_p \in W^{1,p}(\Omega,\, \NN)$ for all $p$, $\F_{\Omega,k}(\S(u_p)- S) \to 0$ as $p \to k$ and
  \[
   \limsup_{p\to k} (k-p)D_p(u_p) \leq \M_k(S).
  \]
 \end{enumerate}
\end{prop}

The proof of Proposition~\ref{prop:Gamma-nobd} follows along the same lines as that of Theorem~\ref{MainThm}. We only need some care in constructing a map~$w_p\in W^{1,p}(\Omega, \, \NN)$ such that~$\S(u_p) = S$, where~$S$ is only assumed to be a finite-mass relative boundary, which does not necessarily belong to the class~\eqref{C_Omega_v}. This can be achieved by a ``dipole construction'' \`a la~\cite{ABO1}, as explained in~\cite[Proof of Proposition~D.(ii)]{CO2}. Therefore, we skip the proof.

\appendix

\section{Uniform gradient estimates for supercritical \texorpdfstring{$p$}{p}-harmonic maps between smooth closed Riemannian manifolds}\label{appendix}
The purpose of this appendix is to provide a self-contained proof for Theorem~\ref{th:p-harmonic}, which holds in a more general setting. Let $\MM$ be a closed smooth Riemannian manifold and let $M:= \mathrm{dim}(\MM)$, $N:= \mathrm{dim}(\NN)$. Recall that $v \in W^{1,p}(\MM, \,\NN)$ is a \textit{(weakly) $p$-harmonic map} if for all $\varphi \in \mathscr{C}_c^1(\MM, \, \mathbb{R}^m)$ one has
\begin{equation}\label{eq_weakly_pharmonic}
\left. \frac{\mathrm{d}}{\mathrm{d}t}\right|_{t=0} \int_{\MM} \lvert \nabla \left( \Pi_{\NN}(v+t\varphi) \right) \rvert^p =0,
\end{equation}
where $\Pi_\NN$ is the nearest-point projection on $\NN$. 
Here~$\nabla$ denotes the Riemannian gradient on~$\MM$ applied component-wise to~$u\colon\MM\to\R^m$.
We prove the following:
\begin{theorem}\label{th:p_harmonic_appendix}
 Let $M < p_0 < p < +\infty$. There exists a positive constant $C$ depending only on $\MM$, $\NN$ and $p_{0}$ such that for all $p \geq p_0$ and $v\in W^{1,p}(\MM, \, \NN)$ a $p$-harmonic map one has
 \begin{equation*}
  \norm{\nabla v}_{L^\infty(\MM)}
  \leq C \norm{\nabla v}_{L^p(\MM)}^{1/\alpha},
 \end{equation*} 
 where~$\alpha := 1 - M/p \in (0, \, 1)$.
\end{theorem}
Theorem~\ref{th:p_harmonic_appendix} provides a \emph{nonlinear} inequality, which is consistent with the scaling of the problem.
We will deduce Theorem~\ref{th:p-harmonic} as a consequence of Theorem~\ref{th:p_harmonic_appendix}, see Section~\ref{sect:pharmonic} below. Classical results dealing with the regularity of (minimizing) $p$-harmonic maps are due to Hardt and Lin~\cite{HardtLin-Minimizing}, see also Luckhaus~\cite{Luckhaus-PartialReg}. A key fact is that if $v$ is Hölder continuous (which is the case if $p>M$ as we assume), then it is locally a map between coordinate neighborhoods. Therefore, one is lead to considering a Hölder continuous solution to a standard quasilinear system for which one can adapt the results by Di Benedetto~\cite{DiBenedetto} and Tolksdorf~\cite{Tolksdorf} (which extend earlier results by Ladyzhenskaya and Uraltseva~\cite{LadyzhenskayaUraltseva,Uraltseva}, Uhlenbeck~\cite{Uhlenbeck} and Evans~\cite{Evans1982}). This is done in~\cite[Section 3]{HardtLin-Minimizing}, but all estimates are formulated in terms of constants depending possibly on $p$. Hence, in order to prove Theorem~\ref{th:p-harmonic} we must repeat the procedure carried out in~\cite[Section 3]{HardtLin-Minimizing} while making sure that all estimates can be written according to a constant $C$ independent on $p$ as long as the latter does not approach~$M$.
\begin{remark}
 The arguments in~\cite{Tolksdorf}, \cite{DiBenedetto}, \cite{HardtLin-Minimizing} do provide a \emph{linear} estimate for the~$L^\infty$-norm of~$\nabla v$ in a sufficiently small ball~$B_{\MM}(x, \, r_{\MM})$, with coefficients that depend on~$r_{\MM}$. 
 (See e.g.~\cite[Proposition~1.1]{BattDiBenedettoManfredi}).
 The nonlinearity in Theorem~\ref{th:p_harmonic_appendix} comes from our choice of the radius~$r_{\MM}$, see~\eqref{rM_def} below. Indeed, when passing from a local estimate to a global one in~$\MM$, for technical reasons we need smallness for a term of the form~$r_{\MM}^\alpha\norm{\nabla u}_{L^p(\MM)}$ (see Equation~\eqref{wherenonlinearitycomesfrom} below). 
 This forces us to choose~$r_{\MM}$ as a nonlinear function of the gradient. Incidentally, the same condition guarantees that the image~$v(B_{\MM}(x, \, \RR_{\MM}))$ is contained in a single coordinate chart of~$\NN$, since~$v\in\mathscr{C}^{0,\alpha}(\MM, \, \NN)$ by Sobolev embeddings.
 However, we do not know whether the exponent~$1/\alpha$ is optimal in general.
\end{remark}
\begin{remark}
 With some additional work, one might be able improve Theorem~\ref{th:p_harmonic_appendix} by establishing uniform~$\mathscr{C}^{1,\alpha}$-estimates on~$v$. 
 However, we did not pursue this direction because Theorem~\ref{th:p_harmonic_appendix} is enough for our purposes.
\end{remark}
\begin{remark}\label{remark_p_harmonic_2}
Without any substantial modification on the proofs, one could also consider the case in which $\MM$ is a manifold with boundary and establish analogous uniform interior estimates.
\end{remark}

\subsection{Proof of Theorem~\ref{th:p_harmonic_appendix}}
Along this Subsection, $C$ will denote a constant larger than $1$ depending only on  $\MM$, $\NN$, $p_{0}$ and which is possibly changing from line to line. Let $v$ be as in the statement of Theorem~\ref{th:p_harmonic_appendix}. We mostly repeat the procedure carried out in~\cite[Section 3]{HardtLin-Minimizing} and~\cite[Sections 2 and 3]{DiBenedetto}. Given $x \in \MM$ and $r$ positive, we denote by $B_{\MM}(x, \, r)$ the geodesic ball in $\MM$ with center $x$ and radius $r$. Analogously, for $y \in \NN$ we denote by $B_{\NN}(y, \, r)$ the geodesic ball in $\NN$ with center $y$ and radius $r$. In the sequel, we shall use the symbol $\#$ to denote the cardinality of a set. 
\subsubsection{A few preliminary results}

We recall here a few results that will be useful in the sequel.
The first one is the Sobolev-Morrey embedding.
We denote by~$\dist_\MM$, $\dist_\NN$ the geodesic distances induced by the Riemannian metrics in~$\MM$, $\NN$, respectively.

\begin{lemma} \label{lemma:Morrey}
 Let~$p_0 > M = \dim\MM$.
 There exists a constant~$C_E$, depending only on~$\MM$, $\NN$ and~$p_0$, such that for any~$p\geq p_0$, any~$u\in W^{1,p}(\MM, \, \NN)$ and any~$x\in\MM$, $y\in\NN$, there holds
 \begin{equation*} 
  \mathrm{dist}_{\NN}(u(x), \, u(y)) 
  \leq C_E\lVert \nabla u \rVert_{L^{p}(\MM)} 
  \dist_{\MM}(x, \, y)^\alpha,
 \end{equation*}
 where~$\alpha := 1 - M/p$.
\end{lemma}
Although the result is well-known, we still provide a proof to make sure that the constant~$C_E$ is bounded uniformly with respect to~$p\geq p_0$.
\begin{proof}[Proof of Lemma~\ref{lemma:Morrey}]
 First, we claim that there is a positive constant~$C$, depending only on $\MM$, $m$ and $p_0$, such that
 \begin{equation} \label{Holder_SE_unif}
  \abs{u(x) - u(y)} \leq C \lVert u \rVert_{W^{1,p}(\MM)} \dist_{\MM}(x, \, y)^\alpha 
 \end{equation}
 for all~$u \in W^{1,p}(\MM,\, \R^m)$ and~$x\in\MM$, $y\in\MM$.
 In case~$\MM$ is a Euclidean ball, say~$\MM = B^M(0, \, 1)$, we have explicit bounds from above for the norm of the embedding~$W^{1,p}_0(B^M(0, \, 1))\hookrightarrow\mathscr{C}^{0,\alpha}(B^M(0, \, 1))$ (see for instance the proof of Theorem~11.34 in~\cite[p.~336]{Leoni-Sobolev}), which prove that the norm is uniformly bounded with respect to~$p \geq p_0$. (Of course, the norm blows up as~$p_0\to M$.) For a closed manifold~$\MM$, the inequality~\ref{Holder_SE_unif} follows by using a partition of unity.
 In a similar way, we have
 \begin{equation} \label{Cont_SE_unif}
  \norm{u}_{L^\infty(\MM)} \leq C \lVert u \rVert_{W^{1,p}(\MM)}, 
 \end{equation}
 where~$C$ can be chosen uniformly with respect to~$u\in W^{1,p}(\MM, \, \R^m)$ and~$p\geq p_0$. In fact, the inequality~\eqref{Cont_SE_unif} follows from~\eqref{Holder_SE_unif}, by choosing~$y\in\MM$ in such a way that~$\abs{u(y)}\abs{\MM}^{1/p} \leq \norm{u}_{L^p(\MM)}$.
 
 Next, let~$u\in W^{1,p}(\MM, \, \R^m)$ and let~$u_{\MM}$ be the integral average of~$u$ in~$\MM$. Let~$\delta > 0$.
 By interpolation and Young's inequality,
 \[
  \begin{split}
   \norm{u - u_{\MM}}_{W^{1,p}(\MM)}
   &\leq \norm{\nabla u}_{L^p(\MM)} 
    + \norm{u - u_{\MM}}_{L^1(\MM)}^{1/p}
    \cdot \norm{u - u_{\MM}}_{L^\infty(\MM)}^{1 - 1/p} \\
   &\leq \norm{\nabla u}_{L^p(\MM)} 
    + \frac{1}{\delta p}\norm{u - u_{\MM}}_{L^1(\MM)}
    + \delta\left(1 - \frac{1}{p}\right)\norm{u - u_{\MM}}_{L^\infty(\MM)}.
  \end{split}
 \]
 We further bound the right-hand side by applying Poincar\'e inequality for~$W^{1,1}(\MM, \, \R^m)$-maps, \eqref{Cont_SE_unif} and H\"older's inequality:
 \[
  \begin{split}
   \norm{u - u_{\MM}}_{W^{1,p}(\MM)}
   &\leq \norm{\nabla u}_{L^p(\MM)} 
    + \frac{C}{\delta}\norm{\nabla u}_{L^1(\MM)}
    + C\delta\norm{u - u_{\MM}}_{W^{1,p}(\MM)} \\
   &\leq \left(1 + \frac{C}{\delta}\right)
    \norm{\nabla u}_{L^p(\MM)} 
    + C\delta\norm{u - u_{\MM}}_{W^{1,p}(\MM)}
  \end{split}
 \]
 The constant~$C$ is independent of~$p\geq p_0$.
 Therefore, choosing~$\delta$ small enough 
 (uniformly with respect to~$p\geq p_0$), we obtain
 \begin{equation} \label{Poinc_unif}
  \begin{split}
   \norm{u - u_{\MM}}_{W^{1,p}(\MM)}
   &\leq C\norm{\nabla u}_{L^p(\MM)} .
  \end{split}
 \end{equation}
 Now, we write the inequality~\eqref{Holder_SE_unif} with~$u - u_{\MM}$ instead of~$u$ and use~\eqref{Poinc_unif} to estimate the right-hand side. We obtain
 \begin{equation} \label{Holder_SE_unif_bis}
  \abs{u(x) - u(y)} \leq C\lVert \nabla u \rVert_{L^p(\MM)} \dist_{\MM}(x, \, y)^\alpha ,
 \end{equation}
 where again~$C$ is independent of~$u$ and~$p$ (but does depend on~$p_0$). 
 Finally, in case~$u$ is~$\NN$-valued, a similar inequality holds with the geodesic distance~$\dist_{\NN}(u(x), \, u(y))$ in place of the Euclidean distance~$\abs{u(x) - u(y)}$, because in a compact, smooth submanifold~$\NN\subseteq\R^m$, the geodesic and the Euclidean distance are comparable to one another. This completes the proof.
\end{proof}

The following (classical) observation will also be useful.
We provide a proof for the reader's convenience.

\begin{lemma} \label{lemma:maxp-harmonic}
 There exists a radius~$\rho(\NN)$, depending on~$\NN$ only, such that the following property holds: If~$v$ is a $p$-harmonic map on a closed manifold~$\MM$ with values in a goedesic disk of the form~$B_{\NN}(\overline{y}, \, \rho(\NN))$, then~$v$ is constant.
\end{lemma}
\begin{proof}
 Since the manifold~$\NN$ is compact and smooth, there exists a number~$\rho(\NN) > 0$ such that, for any~$\overline{y}\in\NN$, the squared geodesic distance~$\psi := \dist_{\NN}(\cdot, \, \overline{y})^2$ is smooth and (geodesically) strictly convex in~$B_{\NN}(\overline{y}, \, \rho(\NN))$. 
 In fact, there exists a number~$\sigma(\NN) > 0$ that satisfies
 \begin{equation} \label{smallballconvex}
  \left\langle\xi, \, \nabla^2\psi(y)\,\xi\right\rangle 
  \geq \sigma(\NN) \abs{\xi}^2 
 \end{equation}
 for any~$y\in B_{\NN}(y)$ and any tangent vector~$\xi\in T_{y}\NN$. 
 Here~$\nabla^2\psi(y)\colon\T_y\NN\to\T_y\NN$ denotes the Hessian of~$\psi$ at~$y$ and the bracket at the left-hand side denotes the Riemannian scalar product in~$\T_y\NN$. Now, by writing the $p$-harmonic map equation~\eqref{eq_weakly_pharmonic} with~$\varphi = \nabla\psi$ and applying~\eqref{smallballconvex}, we obtain
 \begin{equation*}
  0 
  = \int_{M}  \abs{\nabla v}^{p-2} \left\langle\nabla v, \,  (\nabla^2\psi)\nabla v \right\rangle \mathrm{vol}_{\MM}
  \geq \sigma(\NN)
   \int_{\MM} \abs{\nabla v}^p \mathrm{vol}_{\MM}
 \end{equation*}
 and the lemma follows.
\end{proof}

Combining Lemma~\ref{lemma:Morrey} with~\ref{lemma:maxp-harmonic}, we immediately deduce the following consequence.

\begin{corollary} \label{cor:p_harmonic_lowerbound}
 There exists a constant~$\eps_* > 0$, depending ony on~$\MM$, $\NN$ and~$p_0 > M$, such that for any~$p\geq p_0$ and any non-constant $p$-harmonic map~$v\in W^{1,p}(\MM, \, \NN)$ there holds
 \[
  \norm{\nabla v}_{L^p(\MM)} \geq \eps_*.
 \]
\end{corollary}
\begin{proof}
 If~$v$ is not constant, then the diameter of~$v(\MM)$ must be greater than the uniform constant~$\rho(\NN) > 0$ given by Lemma~\ref{lemma:maxp-harmonic}.
 Then, Lemma~\ref{lemma:Morrey} implies
 \begin{equation*}
  \frac{\rho(\NN)}{C_E\,\mathrm{diam}(\MM)^{1 - M/p}}
  \leq \norm{\nabla v}_{L^p(\MM)} \! ,
 \end{equation*}
 where~$\mathrm{diam}(\MM)$ is the diameter of~$\MM$. 
 Since~$\mathrm{diam}(\MM)^{1-M/p}$ is bounded between~$\mathrm{diam}(\MM)$ and $\mathrm{diam}(\MM)^{1 - M/p_0}$ when~$p\geq p_0$, the corollary follows.
\end{proof}

\subsubsection{Reduction to the case of a \texorpdfstring{$p$}{p}-harmonic map between coordinate neighborhoods}
\label{sect:localcoord}
In order to prove Theorem~\ref{th:p_harmonic_appendix}, we reduce to a local problem by using local coordinates, as we explain now.
Notice first that since $\NN$ is compact we can find $I_{\NN} \subseteq \mathbb{N}^*$ finite, $\{ y_j \}_{j \in I_{\NN}}$ a set of points in $\NN$ and $r_\NN>0$ such that $\# I_{\NN}$ and $r_{\NN}$ depend only on $\NN$, $\cup_{j \in I_{\NN}}B_{\NN}(y_j, \, r_{\NN}/2) = \NN$ and for each $j \in I_{\NN}$, $B_{\NN}(y, \, r_{\NN})$ writes
\begin{equation*}
\{ (y, \, f_j(y)): y \in B^N\},
\end{equation*}
up to scaling, rotations and translations, where $f_j: B^N \to \mathbb{R}^{m-N}$ is smooth and such that $f_j(0)=0$, $\nabla f_j(0)=0$ and $\lvert \nabla^2 f_j \rvert \leq 1/4$. Next, as $\MM$ is also compact we can find $A_{\MM} \subseteq \N^*$ finite along with a family $\{ U_{\beta} \}_{\beta \in A_{\MM}}$ of open subsets of $\MM$ such that for each $\beta \in A_{\MM}$ we have a diffeomorphism $\phi_{\beta}: U_\beta \to B^{M}(0, \, 1)$
such that
\begin{equation}\label{bounds_phi}
 \norm{\phi_\beta}_{W^{2,\infty}(U_\beta)}
 + \|\phi_\beta^{-1}\|_{W^{2,\infty}(B^M(0, \, 1))}
 \leq C
\end{equation}
Notice that there exists $\tilde{r}_{\MM}$ such that $B_{\MM}(x, \, \tilde{r}_{\MM}) \subseteq U_{\beta_x}$ for any~$x$ and some $\beta_x \in A_{\MM}$. In order to see this, let~$\eps(x) := \sup\left\{r > 0 \colon \textrm{there is } \beta\in A_{\MM} \textrm{ such that } B_{\MM}(x, \, r)\subseteq U_\beta \right\}$. The function~$\eps\colon\MM\to\R$ is positive and lower semicontinuous
(because~$\eps(y) \geq \eps(x) - \abs{x - y}$ for any~$x$, $y\in\MM$), so it has a positive minimum.
Let now $p$ and $v$ be as in the statement of Theorem~\ref{th:p_harmonic_appendix}, which will be fixed for the rest of the proof. Set
\begin{equation}\label{rM_def}
r_{\MM}:= \left( \frac{\theta_0\, r_{\NN}}{2 C_E \,  p^{\frac{1}{2}} \, \lVert \nabla v \rVert_{L^{p}(\MM)}}\right)^{\frac{1}{\alpha}}\!,
\end{equation}
where~$\alpha := 1 - M/p$ and~$\theta_0\in (0, \, 1]$ is a parameter to be chosen later, as a function of~$\MM$, $\NN$, $p_0$ only. Since~$p^{1/(2\alpha)}\geq \sqrt{p} > 1$, Corollary~\ref{cor:p_harmonic_lowerbound} implies
\[
 r_{\MM} \leq \left(\frac{\theta_0\,r_{\NN}}{2C_E \, \eps_*}\right)^{\frac{p}{p - M}} 
 \leq \min\left\{\left(\frac{\theta_0 \, r_{\NN}}{2 C_E \, \eps_*}\right)^{\frac{p_0}{p_0 - M}}, \, \frac{\theta_0 \, r_{\NN}}{2 C_E \, \eps_*}\right\} \! .
\]
Therefore, by taking~$\theta_0$ small enough, we can make sure that
\begin{equation} \label{r_MM_tilde}
 r_{\MM} \leq \frac{1}{2} \tilde{r}_{\MM}.
\end{equation}
In particular, any ball of the form~$B_{\MM}(\overline{x}, \, r_{\MM})$ for~$\overline{x}\in\MM$ is contained in some~$U_\beta$. Moreover, for any~$\overline{x}\in\MM$ there exists~$\ell(\overline{x})\in I_{\NN}$ such that
\begin{equation} \label{localcoord}
v\left( B_{\MM}(\overline{x}, \, r_{\MM}) \right) \subseteq B_{\NN}(y_{\ell(\overline{x})}, \, r_{\NN}).
\end{equation}
Indeed, since the balls~$B_{\NN}(y_\ell,  r_{\NN}/2)$ cover~$\NN$, for any~$\overline{x}\in\MM$ there exists~$\ell(\overline{x})$ such that 
\begin{equation*}\label{ineq1_localcoord}
 \mathrm{dist}_{\NN}(v(\overline{x}), \, y_{\ell(\overline{x})}) \leq \frac{1}{2}r_{\NN}.
\end{equation*}
On the other hand, according to definition of~$r_{\MM}$ given in \eqref{rM_def} and Lemma~\ref{lemma:Morrey} we have
\begin{equation*}\label{ineq2_localcoord}
 \mathrm{dist}_{\NN}(v(\overline{x}), \, v(x)) 
 \leq \frac{1}{2} r_{\NN}
 \qquad \textrm{for any } x \in B_{\MM} (\overline{x}, \, r_{\MM}),
\end{equation*}
so~\eqref{localcoord} follows.
Most of our efforts will be devoted to proving the following
result. For simplificty of notation, given~$\overline{x}\in\MM$
we define~$B(\overline{x}) := B_{\MM}(\overline{x}, \, r_{\MM})$
and~$B_{1/2}(\overline{x}) := B_{\MM}(\overline{x}, \, r_{\MM}/2)$.
\begin{prop}\label{prop:local_regularity}
For any~$\sigma \in (0, \, 1/M)$ there exists a constant~$C_\sigma$, depending only on~$\MM$, $\NN$, $p_0$ and~$\sigma$, such that the following inequality holds for any~$\overline{x}\in\MM$ and~$p\geq p_0$:
\begin{equation}\label{est_local_regularity}
\norm{\nabla v}_{L^\infty(B_{1/2}(\overline{x}))}^p
\leq C_\sigma p^{1/2} r_{\MM}^{M\sigma}
\norm{\nabla v}_{L^\infty(B(\overline{x}))}^p
\norm{\nabla v}_{L^{2s}(B(\overline{x}))} 
+ C_\sigma r_{\MM}^{-M/2} 
\norm{\nabla v}_{L^{2p}(B(\overline{x}))}^p ,
\end{equation}
where
\begin{equation} \label{s_sigma}
 s = s_\sigma := \frac{M}{2 - 2\sigma M} .
\end{equation}
\end{prop}

Note that the~$L^\infty(\MM)$-norm of the gradient appears in both sides of the inequality. 
Nevertheless, the right-hand side of~\eqref{est_local_regularity} is finite --- we know already that~$v$ is Lipschitz-continuous from the results in~\cite{Tolksdorf, DiBenedetto, HardtLin-Minimizing} --- and, as we will see shortly, the coefficient of~$L^\infty(\MM)$ is small enough.
Before proceeding to the proof of Proposition~\ref{prop:local_regularity}, let us show how it implies Theorem~\ref{th:p_harmonic_appendix}.

\begin{proof}[Proof of Theorem~\ref{th:p_harmonic_appendix}]
 Since the point~$\overline{x}\in\MM$ is arbitrary,
 Proposition~\ref{prop:local_regularity} immediately implies 
 \begin{equation}\label{pharm_final0}
  \norm{\nabla v}_{L^\infty(\MM)}^p
  \leq C_\sigma p^{1/2} r_{\MM}^{M\sigma}
   \norm{\nabla v}_{L^\infty(\MM)}^p
   \norm{\nabla v}_{L^{2s}(B(\overline{x}))}
   + C_\sigma r_{\MM}^{-M/2} 
   \norm{\nabla v}_{L^{2p}(\MM)}^p.
 \end{equation}
 Note that the~$L^{2s}$-norm at the right-hand side is still evaluated on~$B(\overline{x})$.
 We claim that, for a suitable choice of~$\sigma\in (0, \, 1/M)$ (and of~$\theta_0$ in~\eqref{rM_def}), there holds
 \begin{equation}\label{pharm_final1}
  C_\sigma p^{1/2} \, r_{\MM}^{M\sigma} 
  \norm{\nabla v}_{L^{2s}(\overline{x}))} \leq \frac{1}{2}.
 \end{equation}
 Indeed, the exponent~$s$ given in~\eqref{s_sigma} satisfies~$s\to M/2$ as~$\sigma\to 0$, so we can assume that~$2s < p_0$ by taking~$ \sigma$ small enough. 
 Then, we can bound the~$L^{2s}(B(\overline{x}))$-norm using the   H\"older inequality:
 \begin{equation} \label{wherenonlinearitycomesfrom}
  r_{\MM}^{M\sigma} \norm{\nabla v}_{L^{2s}(\overline{x}))}
  \leq r_{\MM}^{M\sigma} \norm{\nabla v}_{L^p(\MM)} 
   \mathrm{vol}\!\left(B(\overline{x})\right)^{\frac{1}{2s} - \frac{1}{p}}
  \stackrel{\eqref{s_sigma}}{=}
   C r_{\MM}^{\alpha}\norm{\nabla v}_{L^p(\MM)}
 \end{equation}
 with~$\alpha = 1 - M/p$. Here~$\mathrm{vol}$ denotes the Riemannian volume on~$\MM$.
 Recalling the definition of~$r_{\MM}$, i.e.~Equation~\eqref{rM_def}, we deduce
 \begin{equation*}
  C_\sigma p^{1/2} \, r_{\MM}^{M\sigma}
   \norm{\nabla v}_{L^{2s}(\overline{x}))}
  \leq C \theta_0,
 \end{equation*}
 for some constant~$C$ that depends on~$\MM$, $\NN$, $p_0$ only. 
 (In fact, $C$ depends on~$\sigma$ too, but the latter is fixed as a function of~$M$, $p_0$ only.)
 Therefore, we can choose~$\theta_0$ small enough that~\eqref{pharm_final1} holds true.
 
 Thanks to~~\eqref{pharm_final1}, we can absorb one of the terms at the right-hand side of~\eqref{pharm_final0} into the left-hand side. After taking the $p$-root, we obtain
 \begin{equation*}
  \norm{\nabla v}_{L^\infty(\MM)}
  \leq C r_{\MM}^{-\frac{M}{2p}} \norm{\nabla v}_{L^{2p}(\MM)}.
 \end{equation*}
 We drop the subscript~$\sigma$ for the constant, because~$\sigma$ is chosen now. 
 We can further bound the right-hand side by interpolation: since
 \begin{equation*}
  \begin{split}
   \norm{\nabla v}_{L^{2p}(\MM)}
   \leq \norm{\nabla v}_{L^{p}(\MM)}^{1/2} 
    \norm{\nabla v}_{L^{\infty}(\MM)}^{1/2},
  \end{split}
 \end{equation*}
 we deduce
 \begin{equation}\label{pharm_final2}
  \begin{split}
   \norm{\nabla v}_{L^\infty(\MM)}
   \leq C r_{\MM}^{-\frac{M}{p}} \norm{\nabla v}_{L^{p}(\MM)}.
  \end{split}
 \end{equation}
 Finally, we inject the definition of~$r_{\MM}$ into the right-hand side of~\eqref{pharm_final2}:
 \begin{equation*}
  \begin{split}
   \norm{\nabla v}_{L^\infty(\MM)}
   \leq C p^{\frac{M}{2\alpha p}} \,
    \norm{\nabla v}_{L^{p}(\MM)}^{1 + \frac{M}{\alpha p}}
   = C p^{\frac{M}{2p - 2M}} 
    \norm{\nabla v}_{L^{p}(\MM)}^{\frac{1}{\alpha}}. 
  \end{split}
 \end{equation*}
 The factor~$p^{M/(2p-2M)}$ is bounded uniformly with respect to~$p\geq p_0 > M$, because~$p^{M/(2p-2M)}\to 1$ as~$p\to+\infty$. 
 This completes the proof of Theorem~\ref{th:p_harmonic_appendix}.
\end{proof}

\subsubsection{Proof of Proposition~\ref{prop:local_regularity}}
\label{sect:regLq}
Let $\overline{x}\in\MM$ be fixed throughout the proof of Proposition~\ref{prop:local_regularity}. Let us also drop all the subscripts for the rest of the proof. In particular, with a slight abuse of notation, the restriction $v|_{B_{\MM}(\overline{x}, \, r_{\MM})}$ will still be denoted as $v$. Recall then (see~\eqref{localcoord}) that $v$ takes values in some ball $B_{\NN}(\overline{y}, \, r_{\NN})$ which, up to scaling, rotations and translations (independent on $v$), writes as
\begin{equation*}
\{ (y, \, f(y)): y \in B^N \},
\end{equation*}
with $f: B^N \to \mathbb{R}^{m-N}$ smooth and such that $f(0)=0$, $\nabla f(0)=0$ and $\lvert \nabla^2 f \rvert \leq 1/4$. Let $\phi$ be the coordinate chart on $\MM$ associated to $x_i$ as we defined above. Let us set $\overline{v}:= v \circ \phi^{-1}$, which is a map defined in~$D:= \phi(B_{\MM}(\overline{x}, \, r_{\MM})) \subseteq \mathbb{R}^M$.
(According to~\eqref{bounds_phi}, the diameter of $D$ is bounded by $Cr_{\MM}$.) We can then write
\begin{equation}\label{decomposition_solution}
\overline{v}=(u, \, f \circ u),
\end{equation}
where $u:=(\overline{v}_1, \ldots, \overline{v}_N)$. Hence, Proposition~\ref{prop:local_regularity} is equivalent to the following statement:
\begin{prop}\label{prop:local_regularity_flat}
For any~$\sigma \in (0, \, 1/M)$ there exists a constant~$C_\sigma$, depending only on~$\MM$, $\NN$, $p_0$ and~$\sigma$, such that the following inequality holds:
\begin{equation}\label{est_local_regularity_flat}
\norm{\nabla u}_{L^\infty(D_{1/2})}^p
\leq C_\sigma p^{1/2} r_{\MM}^{M\sigma}
 \norm{\nabla u}_{L^\infty(D)}^p
\norm{\nabla u}_{L^{2s}(D)}
 + C_\sigma r_{\MM}^{-M/2} \norm{\nabla u}_{L^{2p}(D)}^p ,
\end{equation}
where~$s = s_\sigma$ is defined by~\eqref{s_sigma}.
\end{prop}
In Proposition~\ref{prop:local_regularity_flat} we have set $D_{1/2}:=\phi(B_{\MM}(\overline{x}, \, r_{\MM}/2))$
and we have used~$\nabla$ to denote the gradient of~$u$ with respect to the Euclidean metric. 
In fact, since we are now left with a problem defined in a Euclidean domain, throughout the rest of this appendix we will use the symbols~$\nabla$ and~$\div$ to denote the Euclidean differential operators.
We have further simplified our original problem, which now reduces to proving Proposition~\ref{prop:local_regularity_flat}. One then derives from~\eqref{eq_weakly_pharmonic} the Euler-Lagrange equations satisfied by $u$.
\begin{lemma}\label{lemma_local_flat_weak_solution}
The map $u$ is a weak solution of
\begin{equation}\label{EL_flat}
-\mathrm{div} A( \cdot, \, u, \, \nabla u)+ b(\cdot, \, u, \, \nabla u) = 0 \hspace{1mm} \mbox{ in } D,
\end{equation}
where $A: D \times \mathbb{R}^N \times \mathbb{R}^{M \times N} \to \mathbb{R}^{M \times N}$ and $b :  D \times \mathbb{R}^N \times \mathbb{R}^{M \times N} \to \mathbb{R}^{N}$ are smooth and satisfying the following properties for all $(x, \, y, \, z) \in D \times \mathbb{R}^N \times \mathbb{R}^{M \times N}$ and $\xi \in \mathbb{R}^{M \times N}$:
\begin{enumerate}[label=(\roman*)]
\item\label{positive_definite_local_flat}\hspace{-2mm}. $\sum_{\ell,\ell'=1}^M\sum_{k,k'=1}^N \left[ A_{z_\ell^k} \right]_{\ell'}^{k'} \xi_\ell^k \xi_{\ell'}^{k'} \geq C^{-1}\lvert \xi \rvert^2 \lvert z \rvert^{p-2}$.
\item\hspace{-2mm}. $\left\lvert \left[ A_{z_\ell^k} \right]_{\ell'}^{k'} \right\rvert \leq C\lvert z \rvert^{p-2}$ for all $\ell, \ell'$ in $\{1,\ldots,M\}$ and $k,k'$ in $\{1,\ldots,N\}$.
\item\hspace{-2mm}. $\left\lvert \left[ A_{x_i} \right]_{\ell'}^{k'} \right\rvert, \left\lvert \left[ A_{y_j} \right]_{\ell'}^{k'} \right\rvert \leq C\lvert z \rvert^{p-1}$ for all $i, \ell, \ell'$ in $\{1,\ldots,M\}$ and $j, k,k'$ in $\{1,\ldots,N\}$.
\item\hspace{-2mm}. $\lvert b^h \rvert \leq C \lvert y \rvert \lvert z \rvert^p$ for all $h \in \{1,\ldots,N\}$.
\item\hspace{-2mm}. $\lvert b^h_{x_i} \rvert$, $\lvert b^h_{y_j} \rvert \leq C \lvert z \rvert^p$, $\lvert b^h_{z_\ell^k} \rvert \leq C\lvert z \rvert^{p-1}$ for all $h \in \{1,\ldots,N\}$.
\end{enumerate}
\end{lemma}
\begin{proof}
In this proof, we denote by~$\nabla_{\MM}$ the Riemannian gradient induced by the metric on~$\MM$, to distinguish it from its Euclidean counterpart~$\nabla$. We recall that~$\nabla_{\MM}$ can be expressed in terms of~$\nabla$ and the metric coefficients, which are smooth and satisfy uniform bounds because of~\eqref{bounds_phi}.
From~\eqref{eq_weakly_pharmonic} we get that
\begin{equation}\label{eq_weakly_pharmonic_change}
\left. \frac{\mathrm{d}}{\mathrm{d}t} \right|_{t=0}\int_D \lvert \nabla_{\MM} ( \Pi( \overline{v}+t \overline{\varphi})) \rvert^p \lvert \mathrm{det} \nabla(\phi^{-1}) \rvert=0,
\end{equation}
for all $\overline{\varphi} \in \mathscr{C}^\infty_c(D, \, \mathbb{R}^m)$. Take $\varphi:=(\psi, \, \nabla f(u)\psi )$ for $\psi \in \mathscr{C}^\infty_c(D, \, B^N)$. Then,
\begin{align}\label{expansion_pi_pharmonic}
\Pi( \overline{v}+t \overline{\varphi})& = \Pi( (u+t\psi, \, f(u)+t\nabla f(u)\psi))= \Pi((u+t\psi, \, f(u+t\psi))+o(t)) \nonumber \\&=\Pi((u+t\psi, \, f(u+t\psi)))+o(t)=(u+t\psi, \, f(u+t\psi))+o(t).
\end{align}
Consider the smooth positive function $G$ defined for $(y, \, z) \in B^N \times \mathbb{R}^{M \times N}$ as
\begin{equation*}
G(y, \, z):=\left(\lvert z \rvert^2+\lvert \nabla f(y)z \rvert^2 \right)^{\frac{p}{2}}.
\end{equation*}
Then, by plugging~\eqref{expansion_pi_pharmonic} into~\eqref{eq_weakly_pharmonic_change} it follows
\begin{equation*}
\left.\frac{\mathrm{d}}{\mathrm{d}t}\right|_{t=0} \int_D G(u+t \, \psi, \nabla_{\MM} u + t \, \nabla_{\MM} \psi) \lvert \mathrm{det} \nabla (\phi^{-1}) \rvert = 0.
\end{equation*}
The Euler-Lagrange system~\eqref{EL_flat} is then obtained after some computations with
\begin{equation*}
A(x, \, y, \, z):= G^{(p-2)/p}(y, \, z)\left(z+(\nabla f(y))^T(\nabla f(y))z \right)\lvert \mathrm{det}\nabla (\phi^{-1}(x))\rvert
\end{equation*}
and
\begin{equation*}
b(x, \, y, \, z):= G^{(p-2)/p}(y, \, z) \left(\nabla^2f(y)z \cdot \nabla f(y)z \right)\lvert \mathrm{det}\nabla (\phi^{-1}(x))\rvert.
\end{equation*}
The fact that $A$ and $b$ satisfy the required properties is verified after performing the necessary computations, which we skip. Let us simply mention that for proving~\ref{positive_definite_local_flat} one uses~\eqref{bounds_phi} and the fact that $\lvert \nabla^2 f \rvert \leq 1/4<1$.
\end{proof}
In the sequel, we shall follow Einstein's summation convention and omit the sets in which the indices run in order to lighten the notations. By Lemma~\ref{lemma_local_flat_weak_solution}, we know that $u \in \mathscr{C}^{1, \alpha}(D_{1/2})$ due to the classical regularity results. However, we will need to carry out computations which involve the second-order partial derivatives of $u$. The standard procedure (\cite{HardtLin-Minimizing, DiBenedetto}) consists of performing suitable regularizations of $A$ and $b$ and obtain the estimates for the solutions of the corresponding regularized problems. 
\begin{lemma}\label{lemma_local_flat_approximation}
For all $\varepsilon>0$, there exist $A_\varepsilon$ and $b_\varepsilon$  mapping $D \times \mathbb{R}^N \times \mathbb{R}^{M \times N}$ into $\mathbb{R}^{M \times N}$ and $\mathbb{R}^{N}$ respectively such that:
\begin{enumerate}[label=(\roman*$_\varepsilon$)]
\item\hspace{-2mm}. \label{enum1_eps} $\left[A_{\varepsilon z_{\ell}^k}\right]_{\ell'}^{k'} \geq C^{-1}\lvert \xi \rvert^2\left( \varepsilon+\lvert z \rvert^2 \right)^{(p-2)/2}$,
\item\hspace{-2mm}. \label{enum2_eps} $\left\lvert \left[ A_{\varepsilon z_{\ell}^k}\right]_{\ell'}^{k'} \right\rvert \leq C\left(\varepsilon+\lvert z\rvert^2 \right)^{(p-2)/2}$,
\item\hspace{-2mm}. \label{enum3_eps} $\left\lvert \left[ A_{\varepsilon x_i} \right]_{\ell'}^{k'} \right\rvert$, $\left\lvert \left[ A_{\varepsilon y_j} \right]_{\ell'}^{k'} \right\rvert \leq C \left(\varepsilon+\lvert z \rvert^2 \right)^{(p-1)/2}$,
\item \hspace{-2mm}.\label{enum4_eps} $\lvert b_\varepsilon^h \rvert \leq C\left(\varepsilon+\lvert z \rvert^2 \right)^{p/2}$,
\end{enumerate}
and, one can find a family $(u_\varepsilon)_{\varepsilon>0}$ such that for each $\varepsilon>0$ and for some $\alpha>0$ independent on $\varepsilon$, $u_\varepsilon$ is of class $C^2$ in $D$ and of class $C^{0,\alpha}$ on $\overline{D}$, as well as a classical solution of
\begin{equation}\label{EL_flat_epsilon}
-\mathrm{div} A_\varepsilon( \cdot, \, u_\varepsilon, \, \nabla u_\varepsilon)+ b_\varepsilon(\cdot, \, u_\varepsilon, \, \nabla u_\varepsilon) = 0 \hspace{1mm} \mbox{ in } D.
\end{equation}
Moreover, $\sup_{\varepsilon>0}\lVert u_\varepsilon \rVert_{L^\infty(D)} \leq \lVert u \rVert_{L^\infty(D)}$, $u_\varepsilon \to u$ and $\nabla u_\varepsilon \to \nabla u$ locally uniformly on $D$ and $u_\varepsilon \rightharpoonup u$ weakly in $W^{1,p}(D, \, \mathbb{R}^N)$, as $\varepsilon \to 0$.
\end{lemma}
The proof of Lemma~\ref{lemma_local_flat_approximation} follows by reproducing verbatim the arguments in~\cite[Section 2]{DiBenedetto}. We shall prove that the regularized solutions satisfy the following inequality.
For convenience, we define
\begin{equation} \label{z_eps}
 z_\eps := \left(\eps + \abs{\nabla u_\eps}^2\right)^{\frac{p}{2}}
\end{equation}
We also set~$D_{3/4} := \phi(B_{\MM}(\overline{x}, \, 3r_{\MM}/4))$. 
Lemma~\ref{lemma_local_flat_approximation} implies that~$z_\eps \to\ \abs{\nabla u}^p$ uniformly in~$D_{3/4}$.
\begin{prop}\label{prop:local_regularity_flat_eps}
For any~$\sigma \in (0, \, 1/M)$ there exists a constant~$C_\sigma$, depending only on~$\MM$, $\NN$, $p_0$ and~$\sigma$, that satisfies
\begin{equation}\label{est_local_regularity_flat_eps}
\norm{z_\eps}_{L^\infty(D_{1/2})}
\leq C_\sigma p^{1/2} \norm{z_\eps}_{L^\infty(D_{3/4})}
\left(\eps^{1/2} + \norm{\nabla u_\eps}_{L^{2s}(D_{3/4})} \right)r_{\MM}^{M\sigma}
+ C_\sigma r_{\MM}^{-M/2} \norm{z_\eps}_{L^{2}(D_{3/4})} ,
\end{equation}
with~$s = s_\sigma$ as in~\eqref{s_sigma}.
\end{prop}
One readily observes that Proposition~\ref{prop:local_regularity_flat} follows from Proposition~\ref{prop:local_regularity_flat_eps} due to Lemma~\ref{lemma_local_flat_approximation} and, hence, so does Proposition~\ref{prop:local_regularity}. Hence, the rest of this sub-subsection will be devoted to proving Proposition~\ref{prop:local_regularity_flat_eps} by reproducing the corresponding steps in~\cite{DiBenedetto} while checking that there is no dependence on $p$ for the constants that appear. Let us fix $\varepsilon>0$. For simplicity, we drop the subscript~$\eps$ and write~$u$, $z$, $A$, $b$ instead of~$u_\eps$, $z_\eps$, $A_\eps$, $b_\eps$. Set $w:= \varepsilon+\lvert \nabla u \rvert^2 = z^{2/p}$ and  $\alpha=1-M/p$, as above. Let~$B^M(x, \, r)$ be a ball contained in~$D$, $\eta\colon [0, \, +\infty)\to[0, \, +\infty)$ a nonnegative and nondecreasing Lipschitz function, and~$\zeta\in\mathscr{C}^{\infty}_c(B^M(x, \, r))$ a cut-off function (specific choices will be made later). 
We consider the quantity
\begin{equation}\label{integral_I}
 I_{x, \, r}:= \int_{B^M(x, \, r)} w^{(p-2)/2} \left(\left(\sum_{i=1}^M \lvert \nabla u_{x_i}\rvert^2 \right) \eta(w)+\lvert \nabla w \rvert^2\eta'(w)\right)\zeta^2.
\end{equation}
We can provide an estimate for~$I_{x,r}$ using Equation~\eqref{EL_flat_epsilon}.
\begin{lemma} \label{lemma:I}
 For any~$x$, $r$, $\eta$ and~$\zeta$ as above, we have
 \begin{equation}\label{ineq_integral_I}
  \begin{split}
   I_{x, \, r} &\leq C\int_{B^M(x,\,r)}w^{(p-2)/2}\left\lvert \sum_{\ell=1}^M u_{x_i}u_{x_ix_{\ell}} \right\rvert \eta(w)\zeta \lvert \nabla \zeta \rvert +C\int_{B^M(x,\,r)}w^{p/2}\eta(w)\zeta^2\\ 
   &+C\int_{B^M(x,\,r)}\left(w^{p/2} + w^{(p+1)/2}\right)\eta(w)\zeta\lvert \nabla \zeta \rvert \\
   &+ C\int_{B^M(x,\,r)}w^{(p+2)/2}(\eta(w)+(1+\sqrt{w})^2\eta'(w))\zeta^2 .
  \end{split}
 \end{equation}
\end{lemma}
\begin{proof}
We begin by fixing $i \in \{1, \ldots, M\}$. By differentiating~\eqref{EL_flat_epsilon} with respect to $x_i$, we obtain
\begin{equation}\label{EL_flat_epsilon_i}
-\mathrm{div}\left(A_{x_i}+\sum_{j=1}^NA_{y_j}u_{x_i}+A_{z_{\ell}^k}u_{x_ix_{\ell}}^k \right)+ \frac{\partial}{\partial x_i}(b)=0
\end{equation}
where we have dropped the arguments of $A$ and $b$. Let $\varphi \in \mathscr{C}^\infty_c(D, \, \mathbb{R}^N)$. By scalar multiplying~\eqref{EL_flat_epsilon_i} by $\varphi$, integrating on $D$ and then integrating by parts, one gets
\begin{equation}\label{int_flat_epsilon_i}
\int_{D}\left(A_{x_i}+\sum_{j=1}^NA_{y_j}u_{x_i}+A_{z_{\ell}^k}u_{x_ix_{\ell}}^k \right)\cdot \nabla \varphi- \int_D b \cdot \varphi_{x_i}=0.
\end{equation}
We make the choice~$\varphi:= u_{x_i}\eta(w)\zeta^2$ and study each term of~\eqref{int_flat_epsilon_i} separately.
We begin by noticing that
\begin{align*}
\int_{B^M(x, \, r)}&A_{z_{\ell}^k}u_{x_ix_{\ell}}^k \cdot \nabla \varphi =\int_{B^M(x, \, r)} A_{z_{\ell}^k}u_{x_ix_{\ell}}^k \cdot \left( \nabla u_{x_i}\eta(w)\zeta^2+u_{x_i}\eta'(w)\nabla w+2u_{x_i}\eta(w)\zeta\nabla \zeta \right)\\
&=\int_{B^M(x, \, r)} \left[ A_{z_{\ell}^k} \right]_{\ell'}^{k'}u_{x_i x_{\ell}}^k u_{x_{\ell'} x_i}^{k'} (\eta(w)+2\eta'(w)) \zeta^2+\int_{B^M(x, \, r)} A_{z_{\ell}^k}\cdot u_{x_i x_{\ell}}u_{x_i}\eta(w)\zeta\nabla\zeta,
\end{align*}
so that by adding for $i \in \{1, \ldots, M\}$ and then using~\ref{enum1_eps} and~\ref{enum2_eps} in Lemma~\ref{lemma_local_flat_approximation} we get
\begin{equation}\label{ineq_eps_Az}
\sum_{i=1}^M\int_{B^M(x, \, r)}A_{z_{\ell}^k}u_{x_ix_{\ell}}^k \cdot \nabla \varphi \geq C I_{x, \, r}-C\int_{B^M(x, \, r)} w^{(p-2)/2}\left\lvert \sum_{\ell=1}^M u_{x_i}u_{x_ix_{\ell}} \right\rvert \eta(w)\zeta\lvert \nabla \zeta\rvert.
\end{equation}
Subsequently, for any $\delta>0$ and using~\ref{enum3_eps} in Lemma~\ref{lemma_local_flat_approximation}, 
\begin{align}\label{ineq_eps_Axy}
\sum_{i=1}^M \int_{B^M(x, \, r)}&\left(A_{x_i}+A_yu_{x_i} \right) \cdot \nabla \varphi \leq C \sum_{i=1}^M \int_{B^M(x, \, r)} w^{(p-1)/2}(1+\lvert u_{x_i}\rvert) \lvert \nabla \varphi \rvert\nonumber\\
&= C \sum_{i=1}^M \int_{B^M(x, \, r)} w^{(p-1)/2}(1+\sqrt{w}) \lvert \nabla u_{x_i}\eta(w)\zeta^2+u_{x_i}\eta'(w)\nabla w+2u_{x_i}\eta(w)\zeta\nabla \zeta\rvert\nonumber\\
&\leq \delta I_{x, \, r}+C\delta^{-1}\int_{B^M(x,\,r)}w^{p/2}\eta(w)\zeta^2+C\delta^{-1}\int_{B^M(x,\, r)}w^{(p+2)/2}(\eta(w)+(1+\sqrt{w})^2\eta'(w))\zeta^2\\
&+C\int_{B^M(x,\, r)}w^{p/2}(1+\sqrt{w})\eta(w)\zeta\lvert\nabla \zeta\rvert \nonumber.
\end{align}
Similarly, but using~\ref{enum4_eps} in Lemma~\ref{lemma_local_flat_approximation}
\begin{align}\label{ineq_eps_b}
\int_{B^M(x,\,r)} &b\cdot (u_{x_ix_i}\eta(w)\zeta^2+u_{x_i}w_{x_i}\eta'(w)\zeta^2+2u_{x_i}\eta(w)\zeta\zeta_{x_i})\nonumber\\ &\leq C\int_{B^M(x,\,r)}w^{p/2}\lvert u_{x_ix_i}\eta(w)\zeta^2+u_{x_i}w_{x_i}\eta'(w)\zeta^2+2u_{x_i}\eta(w)\zeta\zeta_{x_i} \rvert\nonumber \\
&\leq \delta I_{x, \, r}+C\delta^{-1}\int_{B^M(x,\, r)}w^{(p+2)/2}(\eta(w)+w\eta'(w))\zeta^2+C\int_{B^M(x,\,r)}w^{(p+1)/2}\eta(w)\zeta\lvert \nabla \zeta\rvert.
\end{align}
After a suitable choice of $\delta$ (depending only on $C$), we obtain the desired inequality~\eqref{ineq_integral_I} by combining~\eqref{int_flat_epsilon_i}, \eqref{ineq_eps_Az},~\eqref{ineq_eps_Axy} and~\eqref{ineq_eps_b}.
\end{proof}

In the next lemma, we apply the inequality~\eqref{ineq_integral_I} for a specific choice of~$\eta$.
Recall that~$z = w^{p/2}$, as in~\eqref{z_eps}.
We denote the positive part of a number~$t\in\R$ as~$t_+ := \max(t, \, 0)$.

\begin{lemma} \label{lemma:DeGiorgiclass} 
 For any ball~$B^M(x, \, r)\subseteq D$, any~$\tau\in (0, \, 1)$ and any~$k\geq 1$, we have
 \begin{equation} \label{DeGiorgiclass}
  \begin{split}
   \int_{B^M(x, \, \tau r)} \abs{\nabla (z - k)_+}^2 
   \leq \frac{C}{(1 - \tau)^2 r^2} 
    \int_{B^M(x, \, r)} (z - k)_+^2 
    + Cp\int_{A_k} z^{2 + 2/p},
  \end{split}
 \end{equation}
 where
 \begin{equation} \label{A_k}
  A_k := \left\{x\in B^M(x, \, r)\colon z(x) > k\right\} \!.
 \end{equation}
\end{lemma}
\begin{proof}
 Given~$B^M(x, \, r)\subseteq D$, $\tau \in (0, \, 1)$ and~$k\geq 1$, we consider a cut-off function~$\zeta\in \mathscr{C}^\infty_c(B^M(x, \, r))$ that is equal to~$1$ on~$B_{\MM}(x, \, \tau r)$ and satisfies $\abs{\nabla\zeta} \leq 2(r - \tau r)^{-2}$.
 We consider again the inequality~\eqref{ineq_integral_I} and choose
 the function~$\eta$ given by~$t\in [0, \, +\infty)\mapsto \eta(t) := (t^{p/2} - k)_+$, so that~$\eta(w) = (z-k)_+$.
 We obtain
 \begin{align}
 I_{x, \, r} &\leq C\int_{B^M(x,\,r)}w^{(p-2)/2}\left\lvert \sum_{\ell=1}^M u_{x_i}u_{x_ix_{\ell}} \right\rvert (z-k)_+\zeta \lvert \nabla \zeta \rvert+C\int_{B^M(x,\,r)}w^{p/2} (z-k)_+\zeta^2\nonumber\\
 &\qquad +C\int_{B^M(x,\,r)}\left(w^{p/2} + w^{(p+1)/2}\right) (z-k)_+\zeta\lvert \nabla \zeta \rvert \label{I_bis}\\
 &\qquad + C\int_{B^M(x,\,r)}w^{(p+2)/2}((z-k)_+ +(1+\sqrt{w})^2\eta'(w))\zeta^2 \nonumber.
 \end{align}
 Here the left-hand side~$I_{x,r}$ is defined by~\eqref{integral_I}.
 We estimate all the terms in this inequality as follows.
 By neglecting one term in~\eqref{integral_I}, we obtain a lower bound for the left-hand side as
 \begin{equation}\label{I_bis0}
  \begin{split}
  I_{x, \, r} \geq C\int_{B^M(x, \, r)} w^{(p-2)/2} 
  \abs{\nabla w}^2\eta'(w) \zeta^2
  &\geq Cp \int_{A_k} w^{p - 2} \abs{\nabla w}^2  \zeta^2 \\
  &= \frac{C}{2} \int_{B^M(x, \, r)}
   \abs{\nabla (z - k)_+}^2  \zeta^2
  \end{split} 
 \end{equation}
 For the first term at the right-hand side, we observe that
 \begin{align}
  &\int_{B^M(x,\,r)}w^{(p-2)/2}\left\lvert \sum_{\ell=1}^M u_{x_i}u_{x_ix_{\ell}} \right\rvert (z-k)_+\zeta \lvert \nabla \zeta \rvert \nonumber \\
  &\hspace{2cm} \leq \frac{1}{p} \int_{B^M(x,\,r)} \abs{\nabla(z - k)_+}
   (z-k)_+\zeta \lvert \nabla \zeta \rvert \label{I_bis1} \\
  &\hspace{2cm}\leq \delta \int_{B^M(x,\,r)} \abs{\nabla(z - k)_+}^2 \zeta^2 
   + \frac{4}{p^2\delta} \int_{B^M(x,\,r)} {(z-k)_+}^2 \abs{\nabla\eta}^2 \nonumber
 \end{align}
 for any~$\delta > 0$. For the second term, we immediately have
 \begin{align} \label{I_bis2}
  \int_{B^M(x,\,r)} w^{p/2} (z-k)_+\zeta^2
  \leq \frac{1}{k} \int_{B^M(x,\,r)}w^{(p+2)/2}(z-k)_+\zeta^2
 \end{align}
 and we can neglect the factor~$1/k$ in front of the right-hand side, because~$k\geq 1$. 
 Similarly,
 \begin{equation} \label{I_bis3}
  \begin{split}
   &\int_{B^M(x,\,r)} \left(w^{p/2} + w^{(p+1)/2}\right) (z-k)_+\zeta\abs{\nabla\zeta} \\
   &\hspace{2cm} \leq C \int_{A_k} w^{p + 1} \zeta^2 
   + C\int_{B^M(x, \, r)} (z - k)_+^2 \abs{\nabla\zeta}^2
  \end{split}
 \end{equation}
 For the last term at the right-hand side, we obtain
 \begin{align}
  \int_{B^M(x,\,r)}w^{(p+2)/2}((z-k)_+ +(1+\sqrt{w})^2\eta'(w))\zeta^2 
  \leq Cp\int_{A_k} w^{p + 1} \zeta^2 .
 \end{align}
 Now the lemma follows by combining~\eqref{I_bis}, \eqref{I_bis0}, \eqref{I_bis1}, \eqref{I_bis2} and~\eqref{I_bis3}.
\end{proof}

\begin{proof}[Proof of Proposition~\ref{prop:local_regularity_flat_eps}]
 The estimate in Proposition~\ref{prop:local_regularity_flat_eps} is a consequence of Lemma~\ref{lemma:DeGiorgiclass}, via an argument that originates in De Giorgi's proof of regularity of minimizers of regular integrals. First, however, we need to further estimate the term~$p\int_{A_k}z^{2 + 2/p}$ at the right-hand side of~\eqref{DeGiorgiclass}.
 We know on, we assume that the ball~$B^M(x, \, r)$ is contained in~$D_{3/4}$.
 We know already that~$z\in L^\infty(D_{3/4})$, thanks to Hardt and Lin's result~\cite{HardtLin-Minimizing}.
 Therefore, we can write
 \[
  p\int_{A_k} z^{2 + 2/p} 
  \leq p \norm{z}_{L^\infty(D_{3/4})}^2 
   \int_{A_k}\left(\eps + \abs{\nabla u}^2\right) \!,
 \]
 recalling the definition of~$z$ in~\eqref{z_eps}.
 Now, we fix~$\sigma\in (0, \, 1/M)$ and take~$s_\sigma$ as in~\eqref{s_sigma}.
 The number~$s$ is defined in such a way that
 $1/s = 2/M - 2\sigma$. Therefore, H\"older's inequality implies
 \[
  \begin{split}
   p\int_{A_k} z^{2 + 2/p} 
   &\leq p \norm{z}_{L^\infty(D_{3/4})}^2 
   \left(\eps \mathcal{L}^M(A_k)^{\frac{2}{M} - 2\sigma}
   + \left(\int_{D_{3/4}}\abs{\nabla u}^{2s}\right)^{\frac{1}{s}} \right)
   \mathcal{L}^M(A_k)^{1 - \frac{2}{M} + 2\sigma} 
  \end{split}
 \]
 By injecting this inequality into~\eqref{DeGiorgiclass}, we obtain
 \begin{equation} \label{DeGiorgiclassbis}
  \begin{split}
   &\int_{B^M(x, \, \tau r)} \abs{\nabla (z - k)_+}^2  \\
   &\hspace{1cm} \leq \frac{C}{(1 - \tau)^2 r^2} 
    \int_{B^M(x, \, r)} (z - k)_+^2 
    + C p \norm{z}_{L^\infty(D_{3/4})}^2 
    \left(\eps + \norm{\nabla u}_{L^{2s}(D_{3/4})}^2 \right)
    \mathcal{L}^M(A_k)^{1 - \frac{2}{M} + 2\sigma}
  \end{split}
 \end{equation}
 The quantity~$\mathcal{L}^M(A_k)^{\frac{2}{M} - 2\sigma}$ is bounded in terms of~$\MM$ only, thanks to~\eqref{r_MM_tilde}.
 Now, we are in position to apply De Giorgi's argument (see for instance Theorem~2.1 in~\cite[Chapter~10]{DiBenedetto-PDEs}).
 This gives
 \begin{equation} \label{DeGiorgiestimate}
  \begin{split}
   &\norm{z}_{L^\infty(B^M(x, \, \tau r))} \\
   &\hspace{1cm} \leq Cp^{1/2} 
    \norm{z}_{L^\infty(D_{3/4})}
    \left(\eps^{1/2} 
    + \norm{\nabla u}_{L^{2s}(D_{3/4})} \right) r^{2\sigma}
   + \frac{C_\sigma}{r^{\frac{M}{2}} 
    \left(1 - \tau\right)^{\frac{1}{\sigma}}}
    \norm{z}_{L^2(B^M(x, \, r))}
  \end{split}
 \end{equation}
 for any two concentric balls~$B^M(x, \, \tau r)\subseteq B^M(x, \, r)$ contained in~$D_{3/4}$. 
 We know that both the diameter of~$D_{1/2}$ and the distance between~$D_{1/2}$ and the boundary of~$D_{3/4}$ are comparable to~$r_{\MM}$, as a consequence of~\eqref{bounds_phi}.
 Therefore, we can choose uniform parameters~$\tau \in (0, \, 1)$ and~$\lambda \in (0, \, 1)$, depending on~$\MM$ only, and cover~$D_{3/4}$ with a finite number of balls of the form~$B^M(x, \, \tau \lambda r_{\MM})$, such that~$B^M(x, \, \lambda r_{\MM})\subseteq D_{3/4}$.
 Then, the proposition follows from~\eqref{DeGiorgiestimate}. 
\end{proof}

\begin{proof}[Proof of Proposition~\ref{prop:local_regularity} completed]
As explained above, once we have proven Proposition~\ref{prop:local_regularity_flat_eps}, we deduce that Proposition~\ref{prop:local_regularity_flat} holds and this yields the proof of Proposition~\ref{prop:local_regularity}.
\end{proof}

\subsection{Proof of Theorem~\ref{th:p-harmonic}}
 \label{sect:pharmonic}
The proof of Theorem \ref{th:p-harmonic} follows directly from Theorem \ref{th:p_harmonic_appendix} by taking $\MM=\SS^{k-1}$ and $v$ a minimiser of the $p$-energy in $W^{1,p}(\SS^{k-1},\NN)$ within a given class $\sigma \in \pi_{k-1}(\NN)$.  \qed
\bibliographystyle{abbrv}
\bibliography{singular_set}
\end{document}